\documentclass[12pt,reqno]{amsart}%
\usepackage{amsmath, amsfonts, amssymb, amsthm, amscd, amsbsy}
\usepackage{fancyhdr}
\usepackage[usenames,dvipsnames,svgnames,x11names,hyperref]{xcolor}
\usepackage{geometry}
\usepackage{graphicx}
\usepackage[pagebackref]{hyperref}%
\usepackage{amsmath}%
\usepackage{amsfonts}%
\usepackage{amssymb}
\usepackage{enumerate}

\providecommand{\U}[1]{\protect\rule{.1in}{.1in}}
\hypersetup{backref=true,
pagebackref=true,
hyperindex=true,
colorlinks=true,
breaklinks=true,
urlcolor=NavyBlue,
linkcolor=Fuchsia,
bookmarks=true,
bookmarksopen=false,
filecolor=black,
citecolor=ForestGreen,
linkbordercolor=red
}
\newtheorem{theorem}{Theorem}[section]
\newtheorem*{theorem*}{Theorem}
\theoremstyle{plain}

\newtheorem{corollary}{Corollary}[section]

\newtheorem{lemma}{Lemma}[section]

\newtheorem{proposition}{Proposition}[section]

\numberwithin{equation}{section}

\newtheorem{theoremalph}{Theorem}

\newtheorem{lemmaalph}{Lemma}

\allowdisplaybreaks

\def\a{\alpha}
\def\b{\beta}
\def\d{\delta}
\def\e{\epsilon}
\def\la{\lambda}
\def\p{\partial}
\def\g{\gamma}

\def\oo{\infty}

\def\dim{\operatorname{dim}}

\def\R{\mathbb{R}}
\def\N{\mathbb{N}}
\def\mbF{\mathbb{F}}

\def\mcL{\mathcal{L}}

\def\mcH{\mathcal{H}}

\def\mfc{\mathfrak{c}}

\newcommand{\gen}[1]{\ensuremath{\left\langle #1\right\rangle}}

\newcommand{\norm}[1]{\left\lVert#1\right\rVert}

\def\B{\mathbb{B}}
\def\Q{\mathbb{Q}}
\def\R{\mathbb{R}}
\def\C{\mathbb{C}}
\def\H{\mathbb{H}}

\def\mcU{\mathcal{U}}

\def\vr{\varrho}

\def\spec{\operatorname{spec}}

\def\Re{\operatorname{Re}}
\def\Im{\operatorname{Im}}

\newtheorem{manualtheoreminner}{Theorem}
\newenvironment{manualtheorem}[1]{%
  \renewcommand\themanualtheoreminner{#1}%
  \manualtheoreminner
}{\endmanualtheoreminner}

\begin{document}
\title[Higher order Poincar\'e-Sobolev and Hardy-Sobolev-Maz'ya inequalities]{Sharp Hardy-Sobolev-Maz'ya, Adams and Hardy-Adams inequalities on quaternionic hyperbolic spaces and the Cayley hyperbolic plane}
\author{Joshua Flynn}
\address{Joshua Flynn: Department of Mathematics\\
University of Connecticut\\
Storrs, CT 06269, USA}
\email{joshua.flynn@uconn.edu}
\author{Guozhen Lu }
\address{Guozhen Lu: Department of Mathematics\\
University of Connecticut\\
Storrs, CT 06269, USA}
\email{guozhen.lu@uconn.edu}
\author{Qiaohua Yang}
\address{School of Mathematics and Statistics, Wuhan University, Wuhan, 430072, People's  Republic of China}
\email{qhyang.math@whu.edu.cn}

\thanks{The first two authors were partially supported by a grant from the Simons Foundation. The third author was  partially supported by the National Natural Science Foundation of China (No.12071353)}
  \textbf{ }

\begin{abstract}
Though Adams and Hardy-Adams inequalities can be extended to general symmetric spaces of noncompact type fairly straightforwardly by following closely the systematic approach developed in our early works on hyperbolic spaces \cite{LuYang3}, \cite{ly4}, \cite{ly2}, \cite{LiLuy1}, \cite{LiLuy2} and more recently on complex hyperbolic spaces in \cite{ly5}, higher order Poincar\'e-Sobolev and Hardy-Sobolev-Maz'ya inequalities
are more difficult to establish. The main purpose of this goal is to establish the Poincar\'e-Sobolev and Hardy-Sobolev-Maz'ya inequalities on
quaternionic hyperbolic spaces and the Cayley hyperbolic plane. A crucial part of our work is to establish appropriate factorization theorems on these spaces which are of their independent interests. To this end, we need to identify and introduce the ``Quaternionic Geller's operators" and
``Octonionic Geller's operators" which have been absent on these spaces. Combining the factorization theorems and the Geller type operators  with the Helgason-Fourier analysis on symmetric spaces,
 the precise heat and  Bessel-Green-Riesz kernel estimates
and the Kunze-Stein phenomenon for connected real simple groups of real rank one with finite center, we succeed to establish the higher order Poincar\'e-Sobolev and Hardy-Sobolev-Maz'ya inequalities on quaternionic hyperbolic spaces and the Cayley hyperbolic plane. The kernel estimates
 required to prove these inequalities are also sufficient for us to establish, as a byproduct, the Adams and Hardy-Adams inequalities on these spaces. This paper, together with earlier works \cite{LuYang3}, \cite{ly4}, \cite{ly2}, \cite{ly5}, \cite{LiLuy1} and \cite{LiLuy2}, completes our study of the factorization theorems, higher order Poincar\'e-Sobolev, Hardy-Sobolev-Maz'ya, Adams and Hardy-Adams inequalities on all rank one symmetric spaces of noncompact type. The factorization theorems and higher order Poincar\'e-Sobolev and Hardy-Sobolev-Maz'ya inequalities on general higher rank symmetric spaces of noncompact type will be studied in a forthcoming paper.

\end{abstract}

\maketitle
\tableofcontents

\section{Introduction}
Let $G$ be a simple Lie group of real rank one. That is,  $G$ is one of the four groups $SO(n,1)$, $SU(n,1)$, $Sp(n,1)$ and $F_{4}$.  Let $K$ be  a maximal compact
subgroup of $G$ and denote by $\mathbb{X}=G/K$. Then $\mathbb{X}$ is a rank one
symmetric space of non-compact type, which is known as the real, complex and quaternionic hyperbolic
spaces, and the Cayley  hyperbolic plane, which we denote them by $H^{n}_{\mathbb{R}}$, $H^{n}_{\mathbb{C}}$, $H^{n}_{\mathbb{Q}}$ and $H^{2}_{\mathbb{O}}$.   Throughout this paper,  we let $\Delta_{\mathbb{X}}$ be the Laplace-Beltrami operator of $X$ and $\rho_{\mathbb{X}}$ be
 the half-sum of the positive roots of $\mathbb{X}$. we note that
 \[
\rho_{\mathbb{X}}=\left\{
                    \begin{array}{ll}
                    \frac{n-1}{2}, & \hbox{$\mathbb{X}=H_{\mathbb{R}}^{n}$;} \\
                    n, & \hbox{$\mathbb{X}=H_{\mathbb{C}}^{n}$;} \\
                      2n+1, & \hbox{$\mathbb{X}=H_{\mathbb{Q}}^{n}$;} \\
                      11, & \hbox{$\mathbb{X}=H_{\mathbb{O}}^{2}$}
                    \end{array}
                  \right.
\]
and $\rho_{\mathbb{X}}^{2}$ is  the spectral gaps of $-\Delta_{\mathbb{X}}$.

 Our main object of study is the  sharp higher order Poincar\'e-Sobolev and Hardy-Sobolev-Maz'ya inequalities and
 their borderline case, Adams and Hardy-Adams inequalities, on $\mathbb{X}$. The Hardy-Sobolev-Maz'ya inequalities,  studied  firstly by Maz'ya \cite{maz}, combine the Hardy and Sobolev inequalities into a single inequality
 and we state it as follows:
\begin{equation*}
\int_{\mathbb{R}^{n}_{+}}|\nabla u|^{2}dx-\frac{1}{4}\int_{\mathbb{R}^{n}_{+}}\frac{u^{2}}{x^{2}_{1}}dx\geq C_n\left(\int_{\mathbb{R}^{n}_{+}}
x^{\gamma}_{1}|u|^{p}dx\right)^{\frac{2}{p}},\;\;\; u\in C^{\infty}_{0}(\mathbb{R}^{n}_{+}), n\geq 3,
\end{equation*}
where $2<p\leq\frac{2n}{n-2}$, $\gamma=\frac{(n-2)p}{2}-n$, $\mathbb{R}^{n}_{+}=\{(x_{1},x_{2},\cdots,x_{n})\in\mathbb{R}^{n}: x_{1}>0\}$ and $C_n$ is a positive constant which is independent of $u$.
 In terms of half-space model of real hyperbolic
spaces, one can see such inequality is equivalent to the Poincar\'e-Sobolev inequality on $H^{n}_{\mathbb{R}}$. The  borderline case when  $n=2$ has been studied by Wang and Ye in  \cite{wy}and the second and third authors \cite{ly}. The higher order inequalities of such type, namely the so-called Hardy-Adams inequalities have been established by the second and third authors and Li (see \cite{LiLuy1, LiLuy2, ly2}).

\subsection{The case $\mathbb{X}=H^{n}_{\mathbb{R}}$} We firstly recall the  Poincar\'e half space model and ball model of $H^{n}_{\mathbb{R}}$. The
Poincar\'e half space model  is given by $\mathbb{R}_{+}\times\mathbb{R}^{n-1}=\{(x_{1},\cdots,x_{n}): x_{1}>0\}$ equipped with the Riemannian metric
$ds^{2}=\frac{dx_{1}^{2}+\cdots+dx_{n}^{2}}{x^{2}_{1}}$. The induced Riemannian measure can be written as $dV=\frac{dx}{x^{n}_{1}}$, where $dx$ is the Lebesgue measure on
$\mathbb{R}^{n}$. The ball model is given by the unit ball
\[\mathbb{B}^{n}=\{x=(x_{1},\cdots,x_{n})\in \mathbb{R}^{n}| |x|<1\}\]
equipped with the usual Poincar\'e metric
\[
ds^{2}=\frac{4(dx^{2}_{1}+\cdots+dx^{2}_{n})}{(1-|x|^{2})^{2}}.
\]
The  factorization theorem on $H^{n}_{\mathbb{R}}$ is given by:
\begin{itemize}
  \item ball model. (see  \cite{liu}.)
  \begin{equation*}
\left(\frac{1-|x|^{2}}{2}\right)^{k+\frac{n}{2}}(-\Delta)^{k}\left[\left(\frac{1-|x|^{2}}{2}\right)^{k-\frac{n}{2}}f\right]=P_{k}f,
\end{equation*}
  \item half space model. (see \cite{LuYang3}.)
  \begin{equation*}
x^{\frac{n}{2}+k}_{1}(-\Delta)^{k}(x^{k-\frac{n}{2}}_{1}f)=P_{k}f,
\end{equation*}
\end{itemize}
where $f\in C^{\infty}(H^{n}_{\mathbb{R}})$,  $\Delta$ is the Laplacian on Euclidean space,   $P_{1}=-\Delta_{\mathbb{X}}-\frac{n(n-2)}{4}$ and
\begin{equation*}
\; P_{k}=P_{1}(P_{1}+2)\cdot\cdots\cdot(P_{1}+k(k-1))
\end{equation*}
is the GJMS operators of order $2k$ on $H^{n}_{\mathbb{R}}$. (see \cite{GJMS}, \cite{FeffermanGr}, \cite{j}.)
On the other hand, the Poincar\'e-Sobolev inequalities reads as
 \begin{equation*}
\int_{H^{n}_{\mathbb{R}}}(\zeta^{2}-\rho_{\mathbb{X}}^{2}-\Delta_{\mathbb{X}})^{s}(-\rho_{\mathbb{X}}^{2}-\Delta_{\mathbb{X}})^{\alpha/2}u\cdot udV
\geq C\|u\|^{2}_{L^{p}(H^{n}_{\mathbb{R}})},
\end{equation*}
where  $0<\alpha<3$,  $0<\zeta$ and $u\in C^{\infty}_{0}(H^{n}_{\mathbb{R}})$.
Therefore, in terms of the  Poincar\'e half space model and ball model of $H^{n}_{\mathbb{R}}$, we have the following
Hardy-Sobolev-Maz'ya inequalities for higher order (see \cite{LuYang3}).
\begin{theoremalph}\label{th1.3}
Let $2\leq k<\frac{n}{2}$ and $2<p\leq\frac{2n}{n-2k}$. There exists a positive constant $C$ such that for each $u\in C^{\infty}_{0}(\mathbb{R}^{n}_{+})$,
\begin{equation*}
\int_{\mathbb{R}^{n}_{+}}|\nabla^{k}u|^{2}dx- \prod^{k}_{i=1}\frac{(2i-1)^{2}}{4}\int_{\mathbb{R}^{n}_{+}}\frac{u^{2}}{x^{2k}_{1}}dx\geq C\left(\int_{\mathbb{R}^{n}_{+}}x^{\gamma}_{1}|u|^{p}dx\right)^{\frac{2}{p}},
\end{equation*}
where $\gamma=\frac{(n-2k)p}{2}-n$.

In terms of the Poincar\'e ball model $\mathbb{B}^{n}$, such  inequality  can be written as
\begin{equation*}
\int_{\mathbb{B}^{n}}|\nabla^{k}u|^{2}dx- \prod^{k}_{i=1}(2i-1)^{2}\int_{\mathbb{B}^{n}}\frac{u^{2}}{(1-|x|^{2})^{2k}}dx\geq C\left(\int_{\mathbb{B}^{n}}(1-|x|^{2})^{\gamma}|u|^{p}dx\right)^{\frac{2}{p}}.
\end{equation*}
\end{theoremalph}

We mention in passing that the best constant  in the above Hardy-Sobolev-Maz'ya inequalities when $k=1$ and $n=3$ is the same as the Sobolev constant (see \cite{bfl}) and is strictly smaller than the Sobolev constant (see \cite{Hebey}). In the higher order derivative  cases, it was proved
in all the cases of $n=2k+1$, the best constants are the same as the Sobolev constants \cite{ly4} (see also \cite{Hong}).

In the borderline case, there holds the Hardy-Adams inequality and we state it as follows (see \cite{ly4}, \cite{LiLuy1}, \cite{LiLuy2}).
\begin{theoremalph} \label{introthm:adams-inequality ball model1}
Let $n\geq3$,  $\zeta>0$ and  $0<s<3/2$.
Then there exists a constant $C_{\zeta,n}>0$ such that for all
$u\in C^{\infty}_{0}(H^{n}_{\mathbb{R}})$ with
\[\int_{H^{n}_{\mathbb{R}}}(\zeta^{2}-\rho_{\mathbb{X}}^{2}-\Delta_{\mathbb{X}})^{s}(-\rho_{\mathbb{X}}^{2}-\Delta_{\mathbb{X}})^{\alpha/2}u\cdot udV\leq1.\]
there holds
\[
\int_{H^{n}_{\mathbb{R}}}(e^{\beta_{0}\left(n/2,n\right) u^{2}}-1-\beta_{0}\left(n/2,n\right)
u^{2})dV\leq C_{\zeta,n},
\]
where
  \[
\beta_{0}(\alpha,n)=\frac{n}{\omega_{n-1}}\left[\frac{\pi^{n/2}2^{\alpha}\Gamma(\alpha/2)}{\Gamma(\frac{n-\alpha}{2})}\right]^{p'},\;\;0<\alpha<n,
\]
  is the best  Adams' constant on $\mathbb{R}^{n}$ and $\omega_{n-1}$ is the area of the surface of the unit $n$-ball.

\end{theoremalph}
In terms of the ball model, we have the following Hardy-Adams
inequalities on $\mathbb{B}^{n}$ (see \cite{wy}, \cite{ly2}, \cite{LiLuy1}.) 
\begin{theoremalph}
There exists a constant $C>0$ such that for all $u\in
C^{\infty}_{0}(\mathbb{B}^{n})$ with
\[
\int_{\mathbb{B}^{n}}|\nabla^{\frac{n}{2}}
u|^{2}dx-\prod^{n/2}_{k=1}(2k-1)^{2}\int_{\mathbb{B}^{n}}\frac{u^{2}}{(1-|x|^{2})^{n}}dx\leq1,
\]
there holds
\[
\int_{\mathbb{B}^{n}}\frac{e^{\beta_{0}\left(n/2,n\right)
u^{2}}-1-\beta_{0}\left(n/2,n\right) u^{2}}{(1-|x|^{2})^{n}}dx\leq C.
\]
\end{theoremalph}

\subsection{The case $\mathbb{X}=H^{n}_{\mathbb{C}}$}
The complex hyperbolic space is a simply connected complete
Kaehler manifold of constant holomorphic sectional curvature $-4$. There are two   models of complex  hyperbolic space,
 the Siegel domain  model $\mathcal{U}^{n}$ and the  ball model $\mathbb{B}^{n}_{\mathbb{C}}$.
The Siegel domain  $\mathcal{U}^{n}\subset \mathbb{C}^{n}$ is defined as
$$
\mathcal{U}^{n}:=\{z\in\mathbb{C}^{n}: \varrho(z)>0\},
$$
where
\begin{equation}\label{2.1}
\varrho(z)=\textrm{Im}z_{n}-\sum_{j=1}^{n-1}|z_{j}|^{2}.
\end{equation}
The Bergman metric on $\mathcal{U}^{n}$ is the metric with Kaehler form $  \omega=\frac{i}{2}\partial\bar{\partial}\log\frac{1}{\varrho}$.
Its boundary $\partial \mathcal{U}^{n}:=\{z\in\mathbb{C}^{n}: \varrho(z)=0\}$ can be identified with the Heisenberg group
$\mathbb{H}^{2n-1}$, which is a nilpotent group of step two  with the group law
\begin{equation*}
  (z,t)\circ(z',t')=(z+z',t+t'+2 \textrm{Im}(z,z')),
\end{equation*}
where $z,z'\in \mathbb{C}^{n-1}$ and $(z,z')$ is the  Hermite inner product
\begin{equation*}
(z,z')=\sum^{n}_{j=1}z_{j}\bar{z}'_{j}.
\end{equation*}
Set $z_{j}=x_{j}+iy_{j} (1\leq j\leq n-1)$ and define
\begin{equation*}
  X_{j}=\frac{\partial}{\partial x_{j}}+ 2y_{j}\frac{\partial}{\partial t},\;
    Y_{j}=\frac{\partial}{\partial y_{j}}- 2x_{j}\frac{\partial}{\partial t}\;\textrm{for}\;j=1,\cdots,n-1, \; T=\frac{\partial}{\partial t}.
\end{equation*}
The $2n-1$ vector fields $X_{1},\cdots, X_{n-1},Y_{1},\cdots, Y_{n-1}, T$ are
 left-invariant and form a basis for Lie algebra of $\mathbb{H}^{2n-1}$.
Let $$\mathcal{L}_{0} =\frac{1}{4}\sum^{n-1}_{j=1}(X^{2}_{j}+Y^{2}_{j})$$
be the sub-Laplacian on $\mathbb{H}^{2n-1}$.
 Then the Laplace-Beltrami operator  is given by
\begin{equation*}
  \Delta_{\mathbb{X}}=4\varrho[\varrho(\partial_{\varrho\varrho}+T^{2})+\mathcal{L}_{0} -(n-1)\partial_{\varrho}].
\end{equation*}

The  ball model is given by the unit ball
\[\mathbb{B}_{\mathbb{C}}^{n}=\{z=(z_{1},\cdots,z_{n})\in \mathbb{C}^{n}: |z|<1\}\]
equipped with the Kaehler metric
\[
ds^{2}=-\partial\bar{\partial}\log (1-|z|^{2}).
\]
The Laplace-Beltrami operator is given by
\[
\Delta_{\mathbb{X}}=4(1-|z|^{2})\sum^{n}_{j,k=1}(\delta_{jk}-z_{j}\bar{z}_{k})\frac{\partial^{2}}{\partial z_{j}\partial \bar{z}_{k}},
\]
where
\[
\delta_{j,k}=\left\{
               \begin{array}{ll}
                 1, & \hbox{$j=k$;} \\
                 0, & \hbox{$j\neq k$.}
               \end{array}
             \right.
\]

The  Geller's operator $\Delta_{\alpha,\beta}$  is defined by (see \cite{ge})
 \begin{equation}\label{b2.7}
   \Delta_{\alpha,\beta}=4(1-|z|^{2})\left[\sum^{n}_{j,k=1}(\delta_{jk}-z_{j}\bar{z}_{k})\frac{\partial^{2}}{\partial z_{j}\partial \bar{z}_{k}}
   +\alpha\sum^{n}_{j=1}z_{j}\frac{\partial}{\partial z_{j}}+\beta\sum^{n}_{j=1}\overline{z}_{j}\frac{\partial}{\partial \overline{z}_{j}}-\alpha\beta \right].
   \end{equation}
   Denote by
   \begin{equation*}
  \begin{split}
R=\sum^{n}_{j=1}z_{j}\frac{\partial}{\partial z_{j}},\;\;\overline{R}=\sum^{n}_{j=1}\overline{z}_{j}\frac{\partial}{\partial \overline{z}_{j}}.
\end{split}
\end{equation*}
Then we have
  \begin{equation*}
  \begin{split}
\Delta_{\alpha,\beta}
=&4(1-|z|^{2})\left[\frac{1-|z|^{2}}{|z|^{2}}R\overline{R}-\frac{1}{|z|^{2}}\mathcal{L}'_{0}+\frac{n-1}{2}\cdot\frac{1}{|z|^{2}}(R+\overline{R})
+\alpha R+\beta\bar{R}-\alpha\beta\right],
\end{split}
\end{equation*}
where
 $\mathcal{L}'_{0}$ is the Folland-Stein operator \cite{fs} on CR sphere defined as follows:
\begin{equation*}
  \mathcal{L}'_{0}=-\frac{1}{2}\sum_{j<k}\left(M_{jk}\overline{M}_{jk}+\overline{M}_{jk}M_{jk}\right), \, where\, \;\;M_{j,k}=z_{j}\partial_{\overline{z}_{k}}
-\overline{z}_{k}\partial_{z_{j}}.
\end{equation*}
For simplicity, we set $\Delta_{\alpha,\beta}'=\frac{1}{4(1-|z|^{2})}\Delta_{\alpha,\beta}$.

The factorization theorem involving Geller's operators on the complex hyperbolic space play an important role in establishing the higher order Poincar\'e-Sobolev and Hardy-Sobolev-Maz'ya inequalities on the complex hyperbolic spaces and can be stated as follows (see\cite{ly5}).
\begin{theoremalph}
Let $a\in\mathbb{R}$ and $k\in\mathbb{N}\setminus\{0\}$.  In terms of  the Siegel domain model, we have,  for $u\in C^{\infty}(\mathcal{U}^{n})$,
\begin{equation}\label{a1.5}
\begin{split}
&\prod^{k}_{j=1}\left[\varrho\partial_{\varrho\varrho}+a\partial_{\varrho}+\varrho T^{2}+\mathcal{L}_{0} -i(k+1-2j)T\right](\varrho^{\frac{k-n-a}{2}} u) \\
=&4^{-k}\varrho^{-\frac{k+n+a}{2}}\prod^{k}_{j=1}\left[\Delta_{\mathbb{X}}+n^{2}-(a-k+2j-2)^{2}\right]u.
\end{split}
 \end{equation}
In terms of  the ball  model, we have,  for $f\in C^{\infty}(\mathbb{B}_{\mathbb{C}}^{n})$,
  \begin{equation}\label{a1.6}
\begin{split}
&\prod^{k}_{j=1}\left[\Delta'_{\frac{1-a-n}{2},\frac{1-a-n}{2}}+\frac{(k+1-2j)^{2}}{4}-
\frac{k+1-2j}{2}(R-\bar{R})\right][(1-|z|^{2})^{\frac{k-n-a}{2}} f] \\
=&4^{-k}(1-|z|^{2})^{-\frac{k+n+a}{2}}\prod^{k}_{j=1}\left[\Delta_{\mathbb{X}}+n^{2}-(a-k+2j-2)^{2}\right]f.
\end{split}
 \end{equation}
 \end{theoremalph}
We note the left sides of (\ref{a1.5}) and (\ref{a1.6}) are closely related to
the CR invariant differential operators on the Heisenberg group
and CR sphere, respectively.

We also have the following  Poincar\'e-Sobolev inequalities on $H_{\mathbb{C}}^{n}$:
\begin{equation*}
\int_{H^{n}_{\mathbb{C}}}(\zeta^{2}-\rho_{\mathbb{X}}^{2}-\Delta_{\mathbb{X}})^{s}(-\rho_{\mathbb{X}}^{2}-\Delta_{\mathbb{X}})^{\alpha/2}u\cdot udV
\geq C\|u\|^{2}_{L^{p}(H^{n}_{\mathbb{R}})},
\end{equation*}
where  $0<\alpha<3$,  $\zeta>0$ and $u\in C^{\infty}_{0}(H^{n}_{\mathbb{C}})$.
Therefore, in terms of two models of $H^{n}_{\mathbb{C}}$, we have the following Hardy-Sobolev-Maz'ya inequalities:
\begin{theoremalph}
  Let $a\in\mathbb{R}$,  $1\leq k<n$ and $2<p\leq\frac{2n}{n-k}$.   In terms of  the  Siegel  domain model, there exists a positive constant $C$ such that   for each $u\in C^{\infty}_{0}(\mathcal{U}^{n})$ we have
\begin{equation*}
  \begin{split}
 & \int_{\mathbb{H}^{2n-1}}\int^{\infty}_{0}u\prod^{k}_{j=1}\left[-\varrho\partial_{\varrho\varrho}-a\partial_{\varrho}-\varrho T^{2}-\mathcal{L}_{0} +i(k+1-2j)T\right]u\frac{dzdtd\varrho}{\varrho^{1-a}}\\
  &-\prod^{k}_{j=1}\frac{(a-k+2j-2)^{2}}{4}\int_{\mathbb{H}^{2n-1}}\int^{\infty}_{0}\frac{u^{2}}{\varrho^{k+1-a}}dzdtd\varrho\\
 \geq&C\left(\int_{\mathbb{H}^{2n-1}}\int^{\infty}_{0}|u|^{p}\varrho^{\gamma}dzdtd\varrho\right)^{\frac{2}{p}},
  \end{split}
\end{equation*}
where $\gamma=\frac{(n-k+a)p}{2}-n-1$. In terms of  the ball  model, we have  for $f\in C_{0}^{\infty}(\mathbb{B}_{\mathbb{C}}^{n})$,
\begin{equation*}
  \begin{split}
 &\int_{\mathbb{B}_{\mathbb{C}}^{n}}f \prod^{k}_{j=1}\left[\Delta'_{\frac{1-a-n}{2},\frac{1-a-n}{2}}+\frac{(k+1-2j)^{2}}{4}-
\frac{k+1-2j}{2}(R-\bar{R})\right]f\frac{dz}{(1-|z|^{2})^{1-a}}\\
  &-\prod^{k}_{j=1}\frac{(a-k+2j-2)^{2}}{4}\int_{\mathbb{B}_{\mathbb{C}}^{n}}\frac{f^{2}}{(1-|z|^{2})^{k+1-a}}dz\\
 \geq&C\left(\int_{\mathbb{B}_{\mathbb{C}}^{n}}|f|^{p}(1-|z|^{2})^{\gamma}dz\right)^{\frac{2}{p}}.
  \end{split}
\end{equation*}
\end{theoremalph}
In the borderline case, there holds the Hardy-Adams inequality and we state  it as follows.
\begin{theoremalph} \label{introthm:adams-inequality ball model}
Let $n\geq3$,  $\zeta>0$ and  $0<s<3/2$.
Then there exists a constant $C_{\zeta,n}>0$ such that for all
$u\in C^{\infty}_{0}(H^{n}_{\mathbb{C}})$ with
\[\int_{H^{n}_{\mathbb{C}}}(\zeta^{2}-\rho_{\mathbb{X}}^{2}-\Delta_{\mathbb{X}})^{s}(-\rho_{\mathbb{X}}^{2}-\Delta_{\mathbb{X}})^{\alpha/2}u\cdot udV\leq1.\]
there holds
\[
\int_{H^{n}_{\mathbb{C}}}(e^{\beta_{0}\left(n,2n\right) u^{2}}-1-\beta_{0}\left(n,2n\right)
u^{2})dV\leq C_{\zeta,n}.
\]
Furthermore, in terms of  the Siegel domain model, we have that for all $u\in
C^{\infty}_{0}(\mathcal{U}^{n})$ with
\begin{equation*}
  \begin{split}
 & 4^{n}\int_{\mathbb{H}^{2n-1}}\int^{\infty}_{0}u\prod^{n}_{j=1}\left[-\varrho\partial_{\varrho\varrho}-a\partial_{\varrho}-\varrho T^{2}-\mathcal{L}_{0} +i(k+1-2j)T\right]u\frac{dzdtd\varrho}{\varrho^{1-a}}\\
  &-\prod^{n}_{j=1}(a-n+2j-2)^{2}\int_{\mathbb{H}^{2n-1}}\int^{\infty}_{0}\frac{u^{2}}{\varrho^{n+1-a}}dzdtd\varrho\leq1,
    \end{split}
\end{equation*}
there holds
\[
 \int_{\mathbb{H}^{2n-1}}\int^{\infty}_{0}\frac{e^{\beta_{0}\left(n,2n\right)
\varrho^{a}u^{2}}-1-\beta_{0}\left(n,2n\right)\varrho^{a} u^{2}}{\varrho^{n+1}}dzdtd\varrho\leq C.
\]
In terms of  the ball model, we have  that for all $u\in
C^{\infty}_{0}(\mathbb{B}_{\mathbb{C}}^{n})$ with
\begin{equation*}
  \begin{split}
 &4^{n}\int_{\mathbb{B}_{\mathbb{C}}^{n}}f \prod^{n}_{j=1}\left[\Delta'_{\frac{1-a-n}{2},\frac{1-a-n}{2}}+\frac{(n+1-2j)^{2}}{4}-
\frac{n+1-2j}{2}(R-\bar{R})\right]f\frac{dz}{(1-|z|^{2})^{1-a}}\\
  &-\prod^{k}_{j=1}(a-k+2j-2)^{2}\int_{\mathbb{B}_{\mathbb{C}}^{n}}\frac{u^{2}}{(1-|z|^{2})^{n+1-a}}dz\leq1,
  \end{split}
\end{equation*}
there holds
\[
\int_{\mathbb{B}_{\mathbb{C}}^{n}} \frac{e^{\beta_{0}\left(n,2n\right)
(1-|z|^{2})^{a}u^{2}}-1-\beta_{0}\left(n,2n\right)(1-|z|^{2})^{a} u^{2}}{(1-|z|^{2})^{n+1}}dz\leq C.
\]

\end{theoremalph}

\subsection{Our Main Results} In this paper, we will consider higher order Poincar\'e-Sobolev and Hardy-Sobolev-Maz'ya inequalities on the  remaining  two  rank one
symmetric spaces of non-compact type, i.e., the quaternionic hyperbolic
spaces $H_{\mathbb{Q}}^{m}$ and the Cayley hyperbolic plane  $H_{\mathbb{O}}^{2}$. The first main result is the factorization theorems.
We shall use the $NA$ group model (or Damek-Ricci space) and the ball model. We note
Damek-Ricci space (see \cite{dr1} \cite{dr2} \cite{MR1469569}) is  a solvable Lie group with a  left invariant
Riemannian structure which include all the rank one
symmetric spaces of non-compact type.

 We need citations for Damek-Ricci spaces. The Damek-Ricci space $NA$ is a semi-direct product of  $A\cong\mathbb{R}$ with a group of Heisenberg type $N$.
Let $\mathfrak{n}$ be a  Lie algebra of $N$,   $\mathfrak{z}$ be the centre of $\mathfrak{n}$ and $\mathfrak{h}$ its orthogonal complement. Denote by $Q=\frac{1}{2}\dim\mathfrak{h}+\dim\mathfrak{z}$  the homogeneous dimension of $N$.
We  parameterize the elements in $N = \exp \mathfrak{n}$ by $(X, Z)$, for $X\in\mathfrak{h}$  and $Z\in\mathfrak{z}$. Then  the group law
is given by
 \begin{equation*}
   (X,Z)(X',Z')=(X+X',Z+Z'+\frac{1}{2}[X,X']).
 \end{equation*}
 Thus the multiplication in $S=NA$ is given by
  \begin{equation*}
   (X,Z,a)(X',Z',a')=(X+a^{1/2}X',Z+aZ'+\frac{1}{2}a^{1/2}[X,X'],aa'), \;a,a'>0.
 \end{equation*}
Let $\Delta_{Z}$ denote the Euclidean Laplacian on the center of $N$ and let $\mcL_{0}$ denote the sub-Laplacian on $N$.
Let $\varrho$ denote the $A$-coordinate of a general point in $S$, and let $\p_{\varrho}$ denote the unit vector in the Lie algebra of $A$.
Then  the Laplace-Beltrami operator  $\Delta_{S}$  on $S$ is
given by
\[
  \Delta_{S} = 4\varrho \left[ \varrho(\p_{\varrho\varrho} + \Delta_{Z}) + \mcL_{0} -(Q-1)\p_{\varrho} \right]
\]
and that the bottom of the spectrum of $-\Delta_{S}$ is $Q^{2}$.

Firstly, we establish the factorization theorem on a Damek-Ricci space from which the factorization theorems on the quaternionic hyperbolic spaces and the Cayley hyperbolic plane follow naturally.  We state it as follows.
\begin{theorem}
  \label{lem:first-factorization}
  Let $a \in \R$ and $f \in C^{\oo}(\mcU)$.
  There holds
  \begin{align*}
    &\varrho^{\frac{k+Q+a}{2}}\prod_{j=1}^{k} \left[ \varrho \p_{\varrho\varrho} + a\p_{\varrho} + \varrho \Delta_{Z} + \mcL_{0} - i(k+1-2j )\sqrt{-\Delta_{Z}} \right]\left( \varrho^{\frac{k-Q-a}{2}}f \right)\\
    &= \prod_{j=1}^{k} \left\{ \varrho \left[ \varrho(\p_{\varrho\varrho} + \Delta_{Z}) + \mcL_{0} -(Q-1)\p_{\varrho} \right] + \frac{Q^{2}}{4} -\frac{(a-k+2j-2)^{2}}{4} \right\}f.
  \end{align*}
\end{theorem}

To state the factorization theorem on the ball model of $H_{\mathbb{Q}}^{m}$, we need to introduce some conventions.
First recall that the quaternionic space $\Q^{m}$ may be identified with $\C^{2m}$ by the identification
\[
  \Q^{m} \ni q = (q_{1},\ldots,q_{m}) \leftrightarrow \C^{2m} \ni z = (z_{1},\ldots,z_{2m}),
\]
where $q_{j} = z_{j} + z_{m+j} i_{2}$.
This allows us to write $\Delta$ in terms of the complex coordinates $z$:
\begin{align*}
  \Delta_{\mathbb{X}} f (z) &= 4\left( 1-|z|^{2} \right) \left\{ \sum_{i,j=1}^{m} \left( \left( \d_{ij} - z_{i}\bar{z}_{j} - \bar{z}_{m+i}z_{m+j} \right)\frac{\p^{2}f}{\p z_{i} \p \bar{z}_{j}}f  \right.\right.\\
  &+ \left( \bar{z}_{i}z_{m+j} - z_{m+i}\bar{z}_{j} \right)\frac{\p^{2} f}{\p z_{m+i} \p \bar{z}_{j}} + \left( \bar{z}_{m+i}z_{j} - z_{i}\bar{z}_{m+j} \right)\frac{\p^{2} f}{\p z_{i} \p \bar{z}_{m+j}}\\
  &\left.\left.+ \left( \d_{ij} - \bar{z}_{i}z_{j} - z_{m+i}\bar{z}_{m+j} \right)\frac{\p^{2} f}{\p z_{m+i} \p \bar{z}_{m+j}} \right) + R+\bar{R} \right\},
\end{align*}
where now
\[
  R = \sum_{j=1}^{2m} z_{j} \frac{\p}{\p z_{j}}\;\;\;\; \text{ and }\;\;\;\; \bar{R} = \sum_{j=1}^{2m} \bar{z}_{j} \frac{\p}{\p \bar{z}_{j}}.
\]
We introduce the following ``Quaternionic Geller operators'': given $\a \in \C$, define the quaternionic Geller operator
\begin{align*}
  \Delta_{\a} f (z) &= 4\left( 1-|z|^{2} \right) \left\{ \sum_{i,j=1}^{m} \left( \left( \d_{ij} - z_{i}\bar{z}_{j} - \bar{z}_{m+i}z_{m+j} \right)\frac{\p^{2}f}{\p z_{i} \p \bar{z}_{j}}f  \right.\right.\\
  &+ \left( \bar{z}_{i}z_{m+j} - z_{m+i}\bar{z}_{j} \right)\frac{\p^{2} f}{\p z_{m+i} \p \bar{z}_{j}} + \left( \bar{z}_{m+i}z_{j} - z_{i}\bar{z}_{m+j} \right)\frac{\p^{2} f}{\p z_{i} \p \bar{z}_{m+j}}\\
  &\left.\left.+ \left( \d_{ij} - \bar{z}_{i}z_{j} - z_{m+i}\bar{z}_{m+j} \right)\frac{\p^{2} f}{\p z_{m+i} \p \bar{z}_{m+j}} \right) + (1+\a)(R+\bar{R})-\a(\a+1) \right\}.
\end{align*}
In particular, $\Delta_{0} = \Delta_{\mathbb{X}}$, and if we set
\[
  \Delta_{\a}' = \frac{1}{4(1-|z|^{2})}\Delta_{\a},
\]
then
\[
  \Delta_{\a}' = \Delta_{0}' + \a (R +  \bar{R}) - \a(\a+1).
\]
We emphasize the analogy between $\Delta_{\a}$ and $D_{\a,\b}$ by pointing out the following intertwining relationships: for $u \in C^{\oo}(\mathbb{B}_{\C}^{n})$ and $s \in \R$, there holds
\[
  \Delta_{s-n,s-n} \left[ (1-|z|^{2})^{s-n}u \right] = 4^{-1}(1-|z|^{2})^{s-n}\left[ \Delta_{0,0} + 4s(n-s) \right]u\text{ on }\mathbb{B}_{\C}^{n}
\]
and, for $u \in C^{\oo}(\B_{\Q}^{m})$ and $s \in \R$, there holds
\[
  \Delta_{s-2m-1}\left[ (1-|z|^{2})^{s-2m-1}u \right] = \left( 1-|z|^{2} \right)^{s-2m-1}\left[ \Delta_{0} + 4s(2m+1-s) \right] \text{ on }\mathbb{B}_{\Q}^{m}.
\]
Recall that the spectral gaps of $-\Delta_{0,0}$ and $-\Delta_{0}$ are $(2m+1)^{2}$ and $n^{2}$, respectively.
Similarly, we can also define the  Geller's operators $\Delta_{\alpha}$ on $H_{\mathbb{O}}^{2}$ through the intertwining relationships in term of ball model
\[
\Delta_{\alpha}\left[ (1-|x|^{2})^{s-11}u \right] = \left( 1-|x|^{2} \right)^{s-11}\left[ \Delta_{\mathbb{X}} + 4s(11-s) \right],
\]
where $11$ is the spectral gaps of $-\Delta_{\mathbb{X}}$ on $H_{\mathbb{O}}^{2}$.
Now we can state the factorization theorem on the ball model of $H_{\mathbb{Q}}^{m}$.
\begin{theorem}
  Let $a\in\R$ and $k\in\N_{>0}$. Set $\Gamma=(R-\bar{R})^{2}-2D_{1}\bar{D_{1}}-2\bar{D_{1}}D_{1}$, where
  \begin{equation*}
    \begin{split}
      &D_{1}=\sum_{a=1}^{n} \{ \bar{z}_{a} \frac{\partial}{\partial {z}_{n+a}}- \bar{z}_{n+a}\frac{\partial}{\partial {z}_{a}} \},\\
      &\bar{D_{1}}=\sum_{a=1}^{n} \{ {z}_{a} \frac{\partial}{\partial \bar{z}_{n+a}}- z_{n+a}\frac{\partial}{\partial \bar{z}_{a}} \}.
    \end{split}
  \end{equation*}
  Then,  in the ball model, for all $f \in C^{\oo}(\B_{\Q}^{m})$, there holds
  \begin{align*}
    &4^{k}\left( 1-|z|^{2} \right)^{\frac{k+a+(2m+1)}{2}}\prod_{j=1}^{k}\left[ \Delta'_{\frac{1-a-(2m+1)}{2}} + \frac{(k+1-2j)^{2}}{4} -i\frac{k+1-2j}{2}\sqrt{\Gamma+1} \right]f\\
    &= \prod_{j=1}^{k} \left[ \Delta_{\mathbb{X}} + (2m+1)^{2} - (a-k+2j-2)^{2} \right]\left[ \left( 1-|z|^{2} \right)^{-\frac{k-a-(2m+1)}{2}}f \right].
  \end{align*}
  \label{thm:factorization-theorem2}
\end{theorem}

The factorization theorem on $H_{\mathbb{O}}^{2}$ in terms of  the ball model is more complex than that in $H_{\mathbb{Q}}^{m}$ and $H_{\mathbb{C}}^{n}$ and involve rather involved computations. We shall
address it in a forthcoming paper.

The second main result is the higher order  Poincar\'e-Sobolev  inequality.  Using precise Bessel-Green-Riesz and heat kernel estimates, we obtain the following:
\begin{theorem}
  Let $0<\g<3$, $0<\g'$, $2<p$ and  $0<\zeta$. Denote by $N=\dim\mathbb{X}$.
  If $0<\g'<N-\g$, suppose further that $2<p\leq\frac{2N}{N-(\g+\g')}$.
  Then there exists a constant $C>0$ such that, for all $u\in{C_{0}^{\oo}\left( \mathbb{X} \right)}$, there holds
  \begin{equation}
    \begin{aligned}
      \norm{u}_{p}\leq{C}\norm{\left( -\Delta_{\mathbb{X}}-\rho_{\mathbb{X}}^{2}+\zeta^{2} \right)^{\frac{\g'}{4}}\left( -\Delta-\rho_{\mathbb{X}}^{2} \right)^{\frac{\g}{4}}u}_{2}.
    \end{aligned}
  \end{equation}
  \label{thm:sobolev-inequality}
\end{theorem}

Using Theorem \ref{thm:sobolev-inequality} and the factorization Theorems  \ref{lem:first-factorization} and \ref{thm:factorization-theorem2}, we obtain the following Hardy-Sobolev-Maz'ya inequalities on $\mathbb{X}$. Here we state only for $H_{\Q}^{m}$.
\begin{theorem}
  Let $a\in\R$, $1\leq{k}<2m$, $2<p<\frac{4m}{2m-k}$ and $\la\leq\prod_{j=1}^{k}\frac{(a-k+2j-2)^{2}}{4}$.
  Then there exists a constant $C>0$ so that, for all $u\in{C_{0}^{\oo}(\mcU_{\mathbb{Q}}^{m})}$, there holds
  \begin{equation*}
    \begin{aligned}
      &\int\limits_{\H_{\Q}^{m-1}}\int\limits_{0}^{\oo}u\prod_{j=1}^{k}\left[ -\vr\p_{\vr\vr}-a\p_{\vr}-\vr\Delta_{Z}-\mathcal{L}_{0}-i(k+1-2j)\sqrt{-\Delta_{Z}} \right]u\frac{dxdzd\vr}{\vr^{1-a}}\\
      &-\la\int\limits_{ \H_{\Q}^{m-1} }\int\limits_{0}^{\oo}\frac{u^{2}}{\vr^{k+1-a}}dxdzd\vr\geq{C}\left( \int\limits_{\H_{\Q}^{m-1}}\int\limits_{0}^{\oo}|u|^{p}\vr^{\frac{\left( 2m+1-k+a \right)p}{2}-(2m-2)}dxdzd\vr \right)^{\frac{2}{p}},
    \end{aligned}
  \end{equation*}
  where $\mcU_{\mathbb{Q}}^{m}$ is the quaternionic  Siegel domain and $\H_{\Q}^{m-1}$ is the quaternionic Heisenberg group.
  In terms of the ball  model, for all $f\in{C_{0}^{\oo}(\B_{\Q}^{m})}$, there holds
  \begin{equation*}
    \begin{aligned}
      &\int\limits_{\B_{\Q}^{m}} f \prod_{j=1}^{k}\left[ \Delta'_{\frac{1-a-(2m+1)}{2}} + \frac{(k+1-2j)^{2}}{4} -i \frac{k+1-2j}{2}\sqrt{\Gamma+1} \right] f \frac{dz}{(1-|z|^{2})^{1-a}}\\
      &-\la \int\limits_{\B_{\Q}^{m}} \frac{f^{2}}{(1-|z|^{2})^{k+1-a}}dz\geq C \left( \int\limits_{\B_{\Q}^{m}}|f|^{p} \left( 1-|z|^{2} \right)^{\frac{(2m+1-k+a)}{2}-(2m-2)}dz \right)^{\frac{2}{p}}.
    \end{aligned}
  \end{equation*}
  \label{thm:hardy-sobolev-mazya-inequality}
\end{theorem}

In the limiting case,  we can establish the Adams inequalities on $\mathbb{X}$.

\begin{theorem}Let $0< \alpha<3$ and $\zeta>0$.
Then there exists a constant $C>0$ such that for all
$u\in C^{\infty}_{0}(\mathbb{X})$ with
\[
\|(-\Delta_{\mathbb{X}}-\rho_{\mathbb{X}}^{2}+\zeta^{2})^{(2n-\alpha)/4}(-\Delta_{\mathbb{X}}-\rho_{\mathbb{X}}^{2})^{\alpha/4}u\|_{2}\leq1,
\]
there holds
\[
\int_{\mathbb{X}}(e^{\beta_{0}\left(N/2,N\right) u^{2}}-1-\beta_{0}\left(N/2,N\right)
u^{2})dV\leq C.
\]
\label{thm:hardy-Adams-inequality}
\end{theorem}
As an application of  Theorem \ref{thm:hardy-Adams-inequality} and the factorization theorem,  we have
 the following Hardy-Adams
inequalities on $\mathbb{X}$. We also state only for $H_{\mathbb{Q}}^{m}$.


\begin{theorem}\label{thm:hardy-Adams-inequality2} Let $a\in\mathbb{R}$.
There exists a constant $C>0$ such that for all $u\in
C^{\infty}_{0}(\mathbb{B}_{\mathbb{Q}}^{m})$ with
\begin{equation*}
  \begin{split}
 &4^{2m}\int_{\mathbb{B}_{\mathbb{Q}}^{n}}u \prod^{2m}_{j=1}\left[\Delta'_{\frac{1-a-(2m+1)}{2}}+\frac{(2m+1-2j)^{2}}{4}-
i\frac{2m+1-2j}{2}\sqrt{\Gamma+1}\right]u\frac{dz}{(1-|z|^{2})^{1-a}}\\
  &-\prod^{2m}_{j=1}(a-2m+2j-2)^{2}\int_{\mathbb{B}_{\mathbb{Q}}^{n}}\frac{u^{2}}{(1-|z|^{2})^{2m+1-a}}dz\leq1,
  \end{split}
\end{equation*}
there holds
\[
\int_{\mathbb{B}_{\mathbb{Q}}^{n}} \frac{e^{\beta_{0}\left(2m,4m\right)
(1-|z|^{2})^{\frac{a+1}{2}}u^{2}}-1-\beta_{0}\left(2m,4m\right)(1-|z|^{2})^{\frac{a+1}{2}} u^{2}}{(1-|z|^{2})^{2m+2}}dz\leq C.
\]

In terms of  the Siegel domain model, we have for all $u\in
C^{\infty}_{0}(\mathcal{U}_{\mathbb{Q}}^{n})$ with
\begin{equation*}
  \begin{split}
 & 4^{2m}\int_{\mathbb{H}_{\mathbb{Q}}^{m-1}}\int^{\infty}_{0}u\prod^{n}_{j=1}\left[-\varrho\partial_{\varrho\varrho}-a\partial_{\varrho}-\varrho \Delta_{Z}-\mathcal{L}_{0} +i(k+1-2j)\sqrt{-\Delta_{Z}}\right]u\frac{dxdzd\varrho}{\varrho^{1-a}}\\
  &-\prod^{2m}_{j=1}(a-n+2j-2)^{2}\int_{\mathbb{H}_{\mathbb{Q}}^{-1}}\int^{\infty}_{0}\frac{u^{2}}{\varrho^{2m+1-a}}dxdzd\varrho\leq1,
    \end{split}
\end{equation*}
there holds
\[
 \int_{\mathbb{H}_{\mathbb{Q}}^{m-1}}\int^{\infty}_{0}\frac{e^{\beta_{0}\left(2m,4m\right)
\varrho^{a}u^{2}}-1-\beta_{0}\left(2m,4m\right)\varrho^{a} u^{2}}{\varrho^{2m+2}}dxdzd\varrho\leq C.
\]
\end{theorem}


Finally, we   set up some Adams type inequalities on Sobolev spaces $W^{\alpha,\frac{N}{\alpha}}(\mathbb{X})$ on $\mathbb{X}$ with dimension $N$ for arbitrary positive fractional
 order $\alpha<N$.
 More precisely, we have the following


\begin{theorem} \label{thm:Adams-inequality1}
Let $N\geq2$, $0<\alpha<N$ be an arbitrary real positive number, $p=N/\alpha$ and $\zeta$ satisfies $\zeta>0$ if $1<p<2$
and $\zeta>\rho_{\mathbb{X}}(\frac{1}{2}-\frac{1}{p})$ if $p\geq2$.
 Then for  measurable $E$
with finite Riemannian volume measure in $\mathbb{X}$, there exists $C=C(\zeta,\alpha,N,|E|)$ such that
\[
\frac{1}{|E|}\int_{E}\exp(\beta_{0}(\alpha,N)|u|^{p'})dV\leq C
\]
for  any $u\in W^{\alpha,p}(\mathbb{X})$ with $\int_{\mathbb{X}}
|(-\Delta_{\mathbb{X}}-\rho_{\mathbb{X}}^{2}+\zeta^{2})^{\frac{\alpha}{2}} u|^{p}dV\leq1$. Here $p'=\frac{p}{p-1}$.
Furthermore, this inequality is sharp  in the sense that if $\beta_{0}(\alpha,N)$ is replaced by any
$\beta>\beta_{0}(\alpha,N)$, then the above inequality can no longer hold with some $C$ independent of $u$.
\end{theorem}

\begin{theorem} \label{thm:Adams-inequality2}
Let $N\geq2$,  $0<\alpha<N$ be an arbitrary real positive number, $p=N/\alpha$ and $\zeta$ satisfies
 $\zeta>2\rho_{\mathbb{X}}\left|\frac{1}{2}-\frac{1}{p}\right|$ .
Then there exists $C=C(\zeta,\gamma,n)$ such that
\begin{equation}\label{inequality1}
\int_{\mathbb{X}}\Phi_{p}(\beta_{0}(\alpha,N)|u|^{p'})dV\leq C
\end{equation}
hold simultaneously
for  any $u\in W^{\alpha,p}(\mathbb{X})$ with $\int_{\mathbb{X}}|(-\Delta_{\mathbb{X}}-
\rho_{\mathbb{X}}^{2}+\zeta^{2})^{\frac{\alpha}{2}} u|^{p}dV\leq1$. Here
\[
\Phi_{p}(t)=e^{t}-\sum^{j_{p}-2}_{j=0}\frac{t^{j}}{j!},\;\; j_{p}=\min \{j\in \mathbb{N}: j\geq p\}.
\]
Furthermore, this inequality is sharp  in the sense that if $\beta_{0}(\alpha,N)$ is replaced by any
 $\beta>\beta(2n,\alpha)$, then the above inequality can no longer hold with some $C$ independent of $u$.
\end{theorem}

Notice that $|\frac{1}{2}-\frac{1}{p}|<\frac{1}{2}$ provided $p>1$. Choosing $\zeta=\rho_{\mathbb{X}}$ in Theorem \ref{thm:Adams-inequality2},  we have the following
\begin{corollary}
Let $N\geq2$,  $0<\alpha<N$ be an arbitrary real positive number and $p=N/\alpha$.
There exists $C=C(\alpha,n)$ such that
\[
\int_{\mathbb{X}}\Phi_{p}(\beta_{0}(\alpha,N)|u|^{p'})dV\leq C
\]
hold simultaneously
for  any $u\in W^{\alpha,p}(\mathbb{X})$ with $\int_{\mathbb{X}}|(-\Delta_{\mathbb{X}})^{\frac{\alpha}{2}} u|^{p}dV\leq1$.
\end{corollary}

 The organization of the paper is as follows:  In Section 2, we recall some necessary preliminary facts of quaternionic hyperbolic spaces and the Cayley hyperbolic plane.  We shall prove the factorization theorem, namely Theorem \ref{lem:first-factorization} and \ref{thm:factorization-theorem2},  in   Section 3.
 In section 4, we recall some necessary  facts of Funk-Hecke formulas for $Sp(m)\times Sp(1)$ and $Spin(9)$ and use them to compute  some  integrals in term of hypergeometric function.
 Sharp estimates of Bessel-Green-Riesz kernels   and their rearrangement estimates  are given  in Section 5 and Section 6, respectively. We shall prove the higher order Hardy-Sobolev-Maz'ya inequalities, namely Theorem \ref{thm:sobolev-inequality} and \ref{thm:hardy-sobolev-mazya-inequality},   in Section 7.  In Section 8, we prove  the Hardy-Adams inequality, namely Theorem \ref{thm:hardy-Adams-inequality}
and \ref{thm:hardy-Adams-inequality2}.
In Section 9, we show the Adams type inequality, namely Theorem \ref{thm:Adams-inequality1}
and \ref{thm:Adams-inequality2}.

\section{Preliminaries}

We begin by setting up notations and then recall proper definitions shortly after.

Let $\Q$ and $\C a$ denote, respectively, the quaternions and the Cayley algebra (i.e., octonions).
Let $H_{\Q}^{m}$ denote the quaternionic hyperbolic space of real dimension $4m$, and let $H_{\C a}$ denote the Cayley plane of real dimension 16.
In general, we will use $\mbF$ to denote any of the three normed division algebras $\left\{ \C, \Q, \C a \right\}$ and $H_{\mbF}^{m}$ to denote the corresponding hyperbolic space with $\mbF$-dimension $m$.
We recall that $H_{\mbF}^{m}$ is a Riemannian symmetric space and that, as homogeneous spaces, there hold $H_{\Q}^{m} = Sp(m,1) / Sp(m) \times Sp(1)$ and $H_{\C a} = F_{4} / \operatorname{Spin}(9)$.
Since there is only one Cayley plane, we shall often remove dimensional superscript and subscript decorations whenever specifying $\mbF = \C a$; e.g., $H_{\mbF}^{m}$ with $\mbF = \C a$ shall be written simply as $H_{\C a}$.

We will also use $\B_{\mbF}^{m} \subset \mbF^{m}$ and $\mcU_{\mbF}^{m}$ to denote $H_{\mbF}^{m}$ when realized, respectively, in the Beltrami-Klein ball model and Siegel domain model.
Let $S^{4m-1} = \p \B_{\Q}^{m}$ and $S^{15} = \p \B_{\C a}$ denote, respectively, the quaternionic and octonionic spheres, and let $d \sigma$ denote the round measure (i.e, the standard surface measure endowed from the ambient Euclidean space).
Note that $\B_{\C a} \subset \C a^{2} = \R^{16}$.

Next, we let $\H_{\mbF}^{n}$ denote the Heisenberg group over $\mbF \in \left\{ \C, \Q, \C a \right\}$ and let $Z = Z(H_{\mbF}^{n})$ denote the center of $\H_{\mbF}^{n}$.
We make the identifications $\H_{\C}^{n} = \R^{2n} \times \R$, $\H_{\Q}^{n} = \R^{4n} \times \R^{3}$ and $\H_{\C a} = \R^{8} \times \R^{7}$ and note $Z(\H_{\C}^{n}) = \R$, $Z(\H_{\Q}^{n}) = \R^{3}$ and $Z(\H_{\C a}) = \R^{7}$.
The homogeneous dimension of $\H_{\mbF}^{n}$ is given by $Q = \dim_{\R} \H_{\mbF}^{n} + \dim_{\R} \Im \mbF$.
In particular, the homogeneous dimensions for $\H_{\C}^{n}$, $\H_{\Q}^{n}$ and $\H_{\C a}$ are, respectively, $2n + 2$, $4n + 6$ and $22$.

Recalling that the boundary of $H_{\mbF}^{m}$ has a natural group structure given by $\H_{\mbF}^{m-1}$, we shall choose the normalization of the metric on $H_{\mbF}^{m}$ and sign convention on $\Delta_{\mathbb{X}}$ so that
\[
  \spec(-\Delta_{\mathbb{X}})=[\frac{Q^{2}}{4},\oo).
\]
We recall that $Q/2$ also has the interpretation as $\rho_{\mathbb{X}}$,  the half sum of positive roots of $H_{\mbF}^{m}$ counted with multiplicities.
In particular, on $H_{\C}^{m}$, $H_{\Q}^{m}$ and $H_{\C a}$ we evaluate $Q/2$ to be, respectively, $m$, $2m+1$, and $11$.

For the convenience of the reader, we include a short dictionary of the Laplacians considered in this paper:
\begin{center}
  \begin{tabular}{lll}
    $\Delta$&$\leftrightsquigarrow$&Laplace-Beltrami operator on $H_{\mbF}^{m}$ when $\mbF = \Q$ or $\C a$\\
    $\Delta_{H_{\R}^{n}}$&$\leftrightsquigarrow$&Laplace-Beltrami operator on $H_{\R}^{n}$ for a specified $n$\\
    $\Delta_{Z}$&$\leftrightsquigarrow$&Euclidean Laplacian on the center $Z=Z(\H_{\mbF}^{m-1})$\\
    $\Delta_{b}$&$\leftrightsquigarrow$&The sub-Laplacian on $\H_{\mbF}^{m-1}$.
  \end{tabular}
\end{center}

In the ball model, the Riemannian volume forms on $H_{\Q}^{m}$ and $H_{\C a}$ are given, respectively, by
\[
  dV = \frac{dz}{(1-|z|^{2})^{2m+2}},
\]
and
\[
  dV = 	\frac{dx}{(1-|x|^{2})^{12}},
\]
where $dz$ and $dx$ denote, respectively, the Lebesgue measure on $\C^{m}$ and $\R^{16}$.

\subsection{Automorphisms and Convolution}

In this section, we recall a family of automorphisms on $\B_{\Q}^{m}$ which are isometries and which are used to define convolution on $\B_{\Q}^{m}$.
Analogous automorphisms are also defined for $\B_{\C a}$ but require more notation and thus we direct the reader to \cite[pg. 56]{vv} for formal definitions.

Following \cite{vv}, we define for each $w\in\B_{\Q}^{m}$ the automorphism $\varphi_{w}:\B_{\Q}^{m}\to\B_{\Q}^{m}$ given by
\begin{align*}
  \varphi_{w}(z)=\left( 1-\gen{z,w}_{\Q} \right)^{-1}\left( w-P_{w}(z)-\sqrt{1-|w|^{2}}Q_{w}(z) \right)
\end{align*}
where
\begin{align*}
  P_{w}(z)&=
  \begin{cases}
    \gen{z,w}_{\Q}|w|^{-2}w&\text{if }w\neq0\\
    0&\text{if }w=0
  \end{cases}\\
  Q_{w}(z)&=z-P_{w}(z).
\end{align*}

We recall some properties of these automorphisms in the following proposition (see \cite{vv}).
Note that property (iv) is not present in \cite{vv}, but it is straightforward to prove.

\begin{lemmaalph}
  \label{lemalph:automorphism-properties}
  For each $w\in\B_{\Q}^{m}$, $\varphi_{w}$ satisfies the following properties:
  \begin{enumerate}[(i)]
    \item $\varphi_{w}(0)=w$ and $\varphi_{w}(w)=0$;
    \item for $z\in\overline{\B_{\Q}^{m}}$, there holds
      \[
	1-|\varphi_{w}(z)|^{2}=\frac{(1-|w|^{2})(1-|z|^{2})}{|1-\gen{z,w}_{\Q}|^{2}};
      \]
    \item $\varphi_{w}$ is an involutory isometry of $\B_{\Q}^{m}$;
    \item for $z\in\overline{\B_{\Q}^{m}}$, there holds
      \begin{align*}
	\sinh\left( \rho\left( \varphi_{w}(z) \right) \right)&=\frac{|\varphi_{w}(z)|}{\sqrt{1-|\varphi_{w}(z)|^{2}}}\\
	&=\left( \frac{|z-w|^{2}+|\gen{z,w}_{\Q}|^{2}-|z|^{2}|w|^{2}}{(1-|w|^{2})(1-|z|^{2})} \right)^{\frac{1}{2}}\\
	\cosh\left( \rho\left( \varphi_{w}(z) \right) \right)&=\frac{1}{\sqrt{1-|\varphi_{w}(z)|^{2}}}\\
	&=\frac{|1-\gen{z,w}_{\Q}|}{\sqrt{(1-|w|^{2})(1-|z|^{2})}}.
      \end{align*}
  \end{enumerate}
\end{lemmaalph}

We will use $\varphi_{w}$ to also denote the analogous automorphisms on $\B_{\C a}$.
We record in the following lemma the analogues to the properties recorded in the preceding lemma.
In preparation, if $z = (z_{1},z_{2}), w = (w_{1},w_{2}) \in \B_{\C a} \subset \C a^{2}$, then let
\[
  \Psi_{\C a}(z,w) =
  \begin{cases}
    |1 - (\bar z_{1} w_{2})(w_{2}^{-1} w_{1}) - z_{2} \bar w_{2} |^{2} & \text{ if } w_{2} \neq 0\\
    |1-\bar z_{1} w_{1} |^{2} & \text{ if }w_{2} =0
  \end{cases}.
\]
We also have $\Psi_{\C a}(z,w) = \Phi_{\C a}(z,w) - 2 \langle{z,w}\rangle_{\R} + 1$, where
\[
  \Phi_{\C a}(z,w) = |z_{1}|^{2}|w_{1}|^{2} + |z_{2}|^{2} |w_{2}|^{2} + 2 \Re\left( (z_{1}z_{2})(\overline{w_{1}w_{2}}) \right),
\]
and $\langle{\cdot,\cdot}\rangle_{\R}$ is the Euclidean inner product on $\R^{16}$.
We also remark that $\Phi_{\C a}(z,w)$ is an analogue of the form $|\langle{z,w}\rangle_{\mbF}|^{2}$, and $\Psi_{\C a}(z,w)$ is an analogue of the form $|1-\gen{z,w}_{\mbF}|^{2}$, where $\mbF \in \left\{ \R,\C,\Q \right\}$.
We point out that $\Psi_{\C a}(z,w) \leq |z|^{2} |w|^{2}$.

\begin{lemmaalph}
  \label{lemalph:automorphism-properties-octonionic}
  For each $w\in\B_{\C a}$, $\varphi_{w}$ satisfies the following properties:
  \begin{enumerate}[(i)]
    \item $\varphi_{w}(0)=w$ and $\varphi_{w}(w)=0$;
    \item for $z\in\overline{\B_{\C a}}$, there holds
      \[
	1-|\varphi_{w}(z)|^{2}=\frac{(1-|w|^{2})(1-|z|^{2})}{\Psi_{\C a}(z,w)};
      \]
    \item $\varphi_{w}$ is an involutory isometry of $\B_{\C a}$;
    \item for $z\in\overline{\B_{\C a}}$, there holds
      \begin{align*}
	\sinh\left( \rho\left( \varphi_{w}(z) \right) \right)&=\frac{|\varphi_{w}(z)|}{\sqrt{1-|\varphi_{w}(z)|^{2}}}\\
	&=\left( \frac{\Psi_{\C a}(z,w) - (1-|z|^{2})(1-|w|^{2})}{(1-|w|^{2})(1-|z|^{2})} \right)^{\frac{1}{2}}\\
	\cosh\left( \rho\left( \varphi_{w}(z) \right) \right)&=\frac{1}{\sqrt{1-|\varphi_{w}(z)|^{2}}}\\
	&=\frac{\sqrt{ \Psi_{\C a}(z,w)}}{\sqrt{(1-|w|^{2})(1-|z|^{2})}}.
      \end{align*}
  \end{enumerate}
\end{lemmaalph}

With these automorphisms defined, we may introduce the following convolution on $\B_{\mbF}^{m}$: for two functions $f,g$ on $\B_{\mbF}^{m}$, let
\[
  (f*g)(z)=\int\limits_{\B_{\mbF}^{m}}f\left( \varphi_{w}(z) \right)g(w)dV(w),
\]
whenever this is well-defined.
It is easy to see that, if $f$ is radial, then $f*g=g*f$, when defined.

\subsection{Helgason-Fourier Transform on Quaternionic Hyperbolic Spaces and Cayley Plane}

In this section, we recall the Helgason-Fourier transforms on the quaternionic hyperbolic spaces and Cayley plane, as well as the resulting Plancherel and inversion formulas. (see \cite{he, he2, te}.)
Given a function $f$ on $\B_{\Q}^{m}$, the Helgason-Fourier transform $\tilde{f}$ is defined by the formula
\[
  \hat{f}(\la,\varsigma)=\int\limits_{\B_{\Q}^{m}}f(z)e_{-\la,\varsigma}(z)dV,
\]
for $\la\in\R$ and $\varsigma\in S^{4m-1}$, provided this integral exists.
Here,
\[
  e_{\la,\varsigma}(z) = \left( \frac{1-|z|^{2}}{|1-\gen{z,\varsigma}_{\Q}|^{2}} \right)^{\frac{(2m+1)+i\la}{2}},
\]
defined for $z\in\B_{\Q}^{m}$, $\la\in\R$ and $\varsigma\in{S^{4m-1}}$, are eigenfunctions of $\Delta$ with eigenvalue $-(2m+1)^{2}-\la^{2}$.
Note that, for $z\in\B_{\Q}^{m}$ and $\varsigma\in S^{4m-1}$, the function
\[
  \left( \frac{1-|z|^{2}}{|1-\gen{z,\varsigma}_{\Q}} \right)^{2m+1},
\]
is the Poisson kernel on $\B_{\Q}^{m}$.

Analogously, if $f$ is a function on $\B_{\C a}$, then its Helgason-Fourier transform $\hat{f}$ is defined by the formula
\[
  \hat{f}(\la,\varsigma)=\int\limits_{\B_{\Q}^{m}}f(z)e_{-\la,\varsigma}(z)dV,
\]
for $\la\in\R$ and $\varsigma\in S^{4m-1}$, provided this integral exists, where now
\[
  e_{\la,\varsigma}(z) = \left( \frac{1-|z|^{2}}{\Psi_{\C a}(z,\varsigma)} \right)^{\frac{11+i\la}{2}},
\]
defined for $z\in\B_{\C a}$, $\la\in\R$ and $\varsigma\in{S^{15}}$, are eigenfunctions of $\Delta$ with eigenvalue $-121-\la^{2}$.
Note that, for $z\in\B_{\Q}^{m}$ and $\varsigma\in S^{4m-1}$, the function
\[
  \left( \frac{1-|z|^{2}}{\Psi_{\C a}(z,\varsigma)} \right)^{11},
\]
is the Poisson kernel on $\B_{\C a}$. 

The Helgason-Fourier transform enjoys the following properties:
\begin{enumerate}[(i)]
  \item For $f,g\in C_{0}^{\oo}(\B_{\mathbb{F}}^{m})$ and $g$ radial, there holds
    \[
      \widehat{f*g}=\hat{f} \cdot \hat{g}.
    \]
  \item For $f\in{C_{0}^{\oo}}(\B_{\mathbb{F}}^{m})$, there holds the inversion formula:
    \begin{equation}
      f(z)=C_{m}\int\limits_{-\oo}^{\oo} \int\limits_{S_{\mbF}} \hat{f}(\la,\varsigma) e_{\la,\varsigma}(z) |\mfc(\la)|^{-2} d\la d\sigma(\varsigma),
      \label{eq:helgason-fourier-inversion-fomrula}
    \end{equation}
    where $C_{m}$ is a positive constant and $\mfc(\la)$ is the Harish-Chandra $\mfc$-function; see \cite[pg. 436]{he} for an explicit formula.
  \item For $f\in{C_{0}^{\oo}}(\B_{\mbF}^{m})$, there holds the Plancherel formula:
    \begin{equation}
      \int\limits_{\B_{\mbF}^{m}}|f(z)|^{2} dV = C_{m} \int\limits_{-\oo}^{\oo} \int\limits_{S_{\mbF}} |\hat{f}(\la, \varsigma)|^{2} |\mfc(\la)|^{-2} d\la d\sigma(\varsigma).
      \label{eq:helgason-fourier-plancherel-formula}
    \end{equation}
  \item For $f\in{C_{0}^{\oo}}(\B_{\mbF}^{m})$, there holds
    \[
      \widehat{\Delta f}(\la,\varsigma) = -\left( \la^{2}+\frac{Q^{2}}{4} \right)\hat{f}(\la,\varsigma).
    \]
\end{enumerate}

\section{Factorization Theorems for the Operators on $\mathbb{X}$: proof of Theorem \ref{lem:first-factorization} and Theorem \ref{thm:factorization-theorem2}}
\subsection{The Factorization Theorem on Damek-Ricci space}
\begin{lemma}
  \label{lem:first-factorization-lemma}
  Let $a \in \R$ and $f \in C^{\oo}(\mcU)$.
  There holds
  \begin{align*}
    & \left[ \varrho \p_{\varrho\varrho} + a\p_{\varrho} + \varrho \Delta_{Z} + \mcL_{0} \right]\left( \varrho^{\frac{1-Q-a}{2}} f \right)\\
    & = \varrho^{-\frac{1+Q+a}{2}}\left\{ \varrho\left[ \varrho(\p_{\varrho\varrho} + \Delta_{Z}) + \mcL_{0} -(Q-1)\varrho \right] + \frac{Q^{2}}{4} - \frac{(a-1)^{2}}{4} \right\}f
  \end{align*}
\end{lemma}

\begin{proof}

  For reference, we provide explicit computations as follows.
  Observing that, for any $\b \in \R$, there holds
  \begin{align*}
    &\varrho^{\b+1} \left[ \varrho \p_{\varrho\varrho} + a \p_{\varrho} + \varrho \Delta_{Z} + \mcL_{0} \right](\varrho^{-\b}f)\\
    &= \varrho \left[ \varrho(\p_{\varrho\varrho} + \Delta_{Z}) + \mcL_{0} - (2\b-a)\p_{\varrho} \right]f + \b(\b+1-a)f,
  \end{align*}
  we may choose $\b=\frac{Q-1+a}{2}$ to obtain
  \begin{align*}
    &\varrho^{\frac{1+Q+a}{2}}\left[ \varrho \p_{\varrho\varrho} + a\p_{\varrho} + \varrho \Delta_{Z} + \mcL_{0} \right]\left( \varrho^{\frac{1-Q-a}{2}}f \right) \\
    &= \left\{ \varrho \left[ \varrho(\p_{\varrho\varrho} + \Delta_{Z}) + \mcL_{0} - (Q-1)\p_{\varrho} \right] + \frac{Q^{2}}{4} - \frac{(a-1)^{2}}{4} \right\}f.
  \end{align*}
  The desired result follows.
\end{proof}

\begin{lemma}
  \label{lem:second-factorization-lemma}
  Let $\b \in \R$.
  There holds
  \begin{align*}
    &\left[ \varrho\p_{\varrho\varrho} + (a+\b)\p_{\varrho} + \varrho\Delta_{Z} + \mcL_{0} \right]\left\{ \left[ \varrho \p_{\varrho\varrho} + (a-1)\p_{\varrho} +\varrho \Delta_{Z} + \mcL_{0} \right]^{2} + (\b-1)^{2}\Delta_{Z} \right\}\\
    &=\left\{ \left[ \varrho\p_{\varrho} + a\p_{\varrho} + \varrho\Delta_{Z} + \mcL_{0} \right]^{2} + \b^{2}\Delta_{Z} \right\}\left[ \varrho\p_{\varrho\varrho} + (a+\b-2)\p_{\varrho} + \varrho\Delta_{Z} + \mcL_{0} \right].
  \end{align*}

\end{lemma}

\begin{proof}

  Since
  \begin{align*}
    \p_{\varrho} \left[ \varrho \p_{\varrho\varrho} + (a-1)\p_{\varrho} + \varrho \Delta_{Z} + \mcL_{0} \right] = \left[ \varrho\p_{\varrho\varrho} + a\p_{\varrho} + \varrho \Delta_{Z} + \mcL_{0} \right]\p_{\varrho} + \Delta_{Z}.
  \end{align*}
  we have
  \begin{align*}
    &\p_{\varrho} \left[ \varrho\p_{\varrho\varrho} + (a-1)\p_{\varrho} + \varrho \Delta_{Z} + \mcL_{0} \right]^{2}\\
    &=\left[ \varrho \p_{\varrho\varrho} + a\p_{\varrho} + \varrho \Delta_{Z} + \mcL_{0} \right] \p_{\varrho} \left[ \varrho \p_{\varrho\varrho} + (a-1)\p_{\varrho} + \varrho \Delta_{Z} + \mcL_{0} \right]\\
    &+ \left[ \varrho \p_{\varrho\varrho} + (a-1)\p_{\varrho} + \varrho \Delta_{Z} + \mcL_{0} \right]\Delta_{Z}\\
    &=\left[ \varrho \p_{\varrho\varrho} + a\p_{\varrho} + \varrho \Delta_{Z} + \mcL_{0} \right]^{2} \p_{\varrho} + \left[ \varrho \p_{\varrho\varrho} + a \p_{\varrho} + \varrho \Delta_{Z} + \mcL_{0} \right]\Delta_{Z}\\
    &+\left[ \varrho \p_{\varrho\varrho} + (a-1)\p_{\varrho} + \varrho \Delta_{Z} + \mcL_{0} \right]\Delta_{Z}\\
    &=\left[ \varrho\p_{\varrho\varrho} + a\p_{\varrho} + \varrho \Delta_{Z} + \mcL_{0} \right]^{2} \p_{\varrho} + 2\left[ \varrho \p_{\varrho\varrho} + a \p_{\varrho} + \varrho \Delta_{Z} + \mcL_{0} \right]\Delta_{Z} - \Delta_{Z} \p_{\varrho}.
  \end{align*}

  Similarly,
  \begin{align*}
    &\left[ \varrho\p_{\varrho\varrho} + (a-1)\p_{\varrho} + \varrho \Delta_{Z} + \mcL_{0} \right]^{2} \\
    &= \left[ \varrho \p_{\varrho\varrho} + a\p_{\varrho} + \varrho \Delta_{Z} + \mcL_{0} \right]\left[ \varrho\p_{\varrho} + (a-1)\p_{\varrho} + \varrho \Delta_{Z} + \mcL_{0} \right]\\
    &- \p_{\varrho} \left[ \varrho \p_{\varrho\varrho} + (a-1)\p_{\varrho} + \varrho \Delta_{Z} + \mcL_{0} \right]\\
    &=\left[ \varrho\p_{\varrho\varrho} + a\p_{\varrho} + \varrho \Delta_{Z} + \mcL_{0} \right]^{2} - 2 \left[ \varrho\p_{\varrho\varrho} + a\p_{\varrho} + \varrho \Delta_{Z} + \mcL_{0} \right]\p_{\varrho} - \Delta_{Z}.
  \end{align*}

  Combining these two computations, we obtain
  \begin{align*}
    &\left[ \varrho \p_{\varrho\varrho} + (a+\b)\p_{\varrho} + \varrho \Delta_{Z} + \mcL_{0} \right]\left\{ \left[ \varrho \p_{\varrho\varrho} + (a-1)\p_{\varrho} + \varrho \Delta_{Z} + \mcL_{0} \right]^{2} + (\b-1)^{2} \Delta_{Z} \right\}\\
    &=\left[ \varrho \p_{\varrho\varrho} + a\p_{\varrho} + \varrho \Delta_{Z} + \mcL_{0} \right] \left\{ \left[ \varrho \p_{\varrho\varrho} + (a-1)\p_{\varrho} + \varrho\Delta_{Z} + \mcL_{0} \right]^{2} + (\b-1)^{2} \Delta_{Z} \right\}\\
    &+\b\p_{\varrho} \left\{ \left[ \varrho \p_{\varrho\varrho} + (a-1)\p_{\varrho} + \varrho \Delta_{Z} +\mcL_{0}  \right]^{2} + (\b-1)^{2} \Delta_{Z} \right\}\\
    &= \left[ \varrho \p_{\varrho\varrho} + a \p_{\varrho} + \varrho \Delta_{Z} + \mcL_{0} \right]\\
    &\cdot \left\{ \left[ \varrho \p_{\varrho\varrho} + a \p_{\varrho} + \varrho \Delta_{Z} + \mcL_{0} \right]^{2} - 2\left[ \varrho \p_{\varrho\varrho} + a\p_{\varrho} + \varrho \Delta_{Z} + \mcL_{0} \right]\p_{\varrho} + \b(\b-2)\Delta_{Z} \right\}\\
    &+\b\left\{ \left[ \varrho\p_{\varrho\varrho} + a\p_{\varrho} + \varrho \Delta_{Z} + \mcL_{0} \right]^{2} \p_{\varrho} + 2 \left[ \varrho\p_{\varrho\varrho} + a \p_{\varrho} + \varrho \Delta_{Z} + \mcL_{0} \right]\Delta_{Z} + \b(\b-2) \Delta_{Z} \p_{\varrho} \right\}\\
    &=\left\{ \left[ \varrho \p_{\varrho\varrho} + a \p_{\varrho} + \varrho \Delta_{Z} + \mcL_{0} \right]^{2} + \b^{2} \Delta_{Z} \right\}\left[ \varrho\p_{\varrho\varrho} + (a+\b-2)\p_{\varrho} + \varrho \Delta_{Z} + \mcL_{0} \right].
  \end{align*}

  This provides the desired identity.
\end{proof}

\begin{lemma}
  \label{lem:third-factorization-lemma}
  For $k \in \N\setminus\left\{ 0 \right\}$, there holds
  \begin{align*}
    &\left[ \varrho\p_{\varrho\varrho} + (a+2k)\p_{\varrho} + \varrho \Delta_{Z} + \mcL_{0} \right] \prod_{j=1}^{k} \left\{ \left[ \varrho \p_{\varrho\varrho} + (a-1) \p_{\varrho} + \varrho \Delta_{Z} + \mcL_{0} \right]^{2} + (2j-1)^{2} \Delta_{Z} \right\}\\
    &= \left( \varrho \p_{\varrho\varrho} + a\p_{\varrho} + \varrho \Delta_{Z} + \mcL_{0}  \right) \prod_{j=1}^{k}\left\{ \left[ \varrho\p_{\varrho\varrho} + (a-1)\p_{\varrho} + \varrho \Delta_{Z} + \mcL_{0} \right]^{2} + 4j^{2}\Delta_{Z} \right\},
  \end{align*}
  and
  \begin{align*}
    &\left[ \varrho \p_{\varrho\varrho} + (a+2k)\p_{\varrho} + \varrho \Delta_{Z} + \mcL_{0}\right]\\
    &\left\{ \left( \varrho \p_{\varrho\varrho} + a\p_{\varrho} + \varrho \Delta_{Z} + \mcL_{0} \right) \prod_{j=1}^{k-1}\left[ \left( \varrho \p_{\varrho\varrho} + (a-1)\p_{\varrho} + \varrho \Delta_{Z} + \mcL_{0} \right)^{2} + 4j^{2} \Delta_{Z} \right] \right\}\\
    &=\prod_{j=1}^{k} \left\{  \left[ \varrho\p_{\varrho\varrho} + (a-1)\p_{\varrho} + \varrho \Delta_{Z} + \mcL_{0} \right]^{2} + (2j-1)^{2} \Delta_{Z} \right\}.
  \end{align*}
\end{lemma}

\begin{proof}
  By Lemma \ref{lem:second-factorization-lemma}, we have
  \begin{align*}
    &\left[ \varrho\p_{\varrho\varrho} + (a+2k) \p_{\varrho} + \varrho \Delta_{Z} + \mcL_{0} \right] \prod_{j=1}^{k} \left\{ \left[ \varrho \p_{\varrho\varrho} + (a-1)\p_{\varrho} + \varrho \Delta_{Z} + \mcL_{0} \right]^{2} + (2j-1)^{2} \Delta_{Z} \right\}\\
    &=\left[ \varrho \p_{\varrho\varrho} + (a+2k)\p_{\varrho} + \varrho \Delta_{Z} + \mcL_{0} \right]\left\{ \left[ \varrho\p_{\varrho\varrho} + (a-1)\p_{\varrho} + \varrho \Delta_{Z} + \mcL_{0} \right]^{2} + (2k-1)^{2}\Delta_{Z} \right\}\\
    &\cdot \prod_{j=1}^{k-1} \left\{ \left[ \varrho\p_{\varrho\varrho} + (a-1)\p_{\varrho} + \varrho \Delta_{Z} + \mcL_{0} \right]^{2} + (2j-1)^{2} \Delta_{Z} \right\}\\
    &=\left\{ \left[ \varrho \p_{\varrho\varrho} + a \p_{\varrho} + \varrho \Delta_{Z} + \mcL_{0} \right]^{2} + 4k^{2} \Delta_{Z} \right\}\\
    &\cdot \left[ \varrho\p_{\varrho\varrho} + (a+2k-2) \p_{\varrho} + \varrho \Delta_{Z} + \mcL_{0} \right] \prod_{j=1}^{k-1}\left\{ \left[ \varrho\p_{\varrho\varrho} + (a-1)\p_{\varrho} + \varrho \Delta_{Z} + \mcL_{0} \right]^{2} + (2j-1)^{2}\Delta_{Z} \right\}.
  \end{align*}
  By repeating this process, we get the first identity in the lemma.
  The second identity is similarly obtained.
\end{proof}

\textbf{Proof of Theorem \ref{lem:first-factorization}}
  It is sufficient to show
  \begin{align*}
    &\prod_{j=1}^{k} \left[ \varrho \p_{\varrho\varrho} + a\p_{\varrho} + \varrho \Delta_{Z} + \mcL_{0} - i(k+1-2j )\sqrt{-\Delta_{Z}} \right]\left( \varrho^{\frac{k-Q-a}{2}}f \right)\\
    &=\varrho^{-\frac{k+Q+a}{2}} \prod_{j=1}^{k} \left\{ \varrho \left[ \varrho(\p_{\varrho\varrho} + \Delta_{Z}) + \mcL_{0} -(Q-1)\p_{\varrho} \right] + \frac{Q^{2}}{4} -\frac{(a-k+2j-2)^{2}}{4} \right\}f.
  \end{align*}
  We shall prove the lemma by induction.
  We have by Lemma \ref{lem:first-factorization-lemma}, the identity above is valid for $k=1$.
  Now assume it is valid for $k=l$, i.e.,
  \begin{align*}
    &\prod_{j=1}^{l} \left[ \varrho \p_{\varrho\varrho} + a\p_{\varrho} + \varrho \Delta_{Z} + \mcL_{0} - i(l+1-2j )\sqrt{-\Delta_{Z}} \right]\left( \varrho^{\frac{l-Q-a}{2}}f \right)\\
    &=\varrho^{-\frac{l+Q+a}{2}} \prod_{j=1}^{l} \left\{ \varrho \left[ \varrho(\p_{\varrho\varrho} + \Delta_{Z}) + \mcL_{0} -(Q-1)\p_{\varrho} \right] + \frac{Q^{2}}{4} -\frac{(a-l+2j-2)^{2}}{4} \right\}f.
  \end{align*}
  Making the substitution $a \to a-1$, we obtain
  \begin{align*}
    &\prod_{j=1}^{l} \left[ \varrho \p_{\varrho\varrho} + (a-1)\p_{\varrho} + \varrho \Delta_{Z} + \mcL_{0} - i(l+1-2j )\sqrt{-\Delta_{Z}} \right]\left( \varrho^{\frac{l-Q-a+1}{2}}f \right)\\
    &=\varrho^{-\frac{l+Q+a-1}{2}} \prod_{j=1}^{l} \left\{ \varrho \left[ \varrho(\p_{\varrho\varrho} + \Delta_{Z}) + \mcL_{0} -(Q-1)\p_{\varrho} \right] + \frac{Q^{2}}{4} -\frac{(a-1-l+2j-2)^{2}}{4} \right\}f.
  \end{align*}
  If $l$ is even, then Lemma \ref{lem:third-factorization-lemma} gives us
  \begin{align*}
    &\left[ \varrho \p_{\varrho} + (a+l)\p_{\varrho} + \varrho \Delta_{Z} + \mcL_{0} \right] \prod_{j=1}^{l} \left[ \varrho \p_{\varrho\varrho} + (a-1)\p_{\varrho} + \varrho \Delta_{Z} + \mcL_{0} - i(l+1-2j)\sqrt{-\Delta_{Z}} \right]\\
    &=\left[ \varrho\p_{\varrho\varrho} + (a+l) \p_{\varrho} + \varrho \Delta_{Z} + \mcL_{0} \right] \prod_{j=1}^{l/2}\left\{  \left[ \varrho\p_{\varrho\varrho} + (a-1)\p_{\varrho} + \varrho \Delta_{Z} + \mcL_{0} \right]^{2} + (2j-1)^{2} \Delta_{Z} \right\}\\
    &=\left( \varrho\p_{\varrho\varrho} +  a\p_{\varrho} + \varrho \Delta_{Z} + \mcL_{0} \right) \prod_{j=1}^{l/2} \left\{  \left[ \varrho\p_{\varrho\varrho} + (a-1)\p_{\varrho} + \varrho \Delta_{Z} + \mcL_{0} \right]^{2} + 4j^{2} \Delta_{Z} \right\}\\
    &= \prod_{j=1}^{l+1}\left[ \varrho \p_{\varrho\varrho} + (a-1)\p_{\varrho} + \varrho \Delta_{Z} -i(l+2-2j)\sqrt{-\Delta_{Z}} \right].
  \end{align*}
  Therefore, by Lemma \ref{lem:first-factorization-lemma}, there holds
  \begin{align*}
    &\left[ \varrho\p_{\varrho\varrho} + (a+l)\p_{\varrho} + \varrho \Delta_{Z} + \mcL_{0}  \right]\\
    &\left\{ \varrho^{-\frac{l+Q+a-1}{2}} \prod_{j=1}^{l}\left\{ \varrho\left[ \varrho (\p_{\varrho\varrho} + \Delta_{Z}) + \mcL_{0} -(Q-1)\p_{\varrho}  \right] + \frac{Q^{2}}{4} - \frac{(a-l+2j-3)^{2}}{4} \right\}f \right\}\\
    &= \varrho^{-\frac{l+Q+a+1}{2}} \prod_{j=1}^{l} \left\{ \varrho \left[ \varrho (\p_{\varrho\varrho} + \Delta_{Z}) + \mcL_{0} -(Q-1)\p_{\varrho} \right] + \frac{Q^{2}}{4} - \frac{(a-l+2j-3)^{2}}{4} \right\}f.
  \end{align*}
  The case for $l$ is odd is obtained by the second identity in Lemma \ref{lem:third-factorization-lemma}.

\subsection{The Factorization Theorem on the ball model of $H_{\mathbb{Q}}^{m}$}

Recall that
\begin{align*}
  \Delta_{\a} f (z) &= 4\left( 1-|z|^{2} \right) \left\{ \sum_{i,j=1}^{m} \left( \left( \d_{ij} - z_{i}\bar{z}_{j} - \bar{z}_{m+i}z_{m+j} \right)\frac{\p^{2}f}{\p z_{i} \p \bar{z}_{j}}f  \right.\right.\\
  &+ \left( \bar{z}_{i}z_{m+j} - z_{m+i}\bar{z}_{j} \right)\frac{\p^{2} f}{\p z_{m+i} \p \bar{z}_{j}} + \left( \bar{z}_{m+i}z_{j} - z_{i}\bar{z}_{m+j} \right)\frac{\p^{2} f}{\p z_{i} \p \bar{z}_{m+j}}\\
  &\left.\left.+ \left( \d_{ij} - \bar{z}_{i}z_{j} - z_{m+i}\bar{z}_{m+j} \right)\frac{\p^{2} f}{\p z_{m+i} \p \bar{z}_{m+j}} \right) + (1+\a)(R+\bar{R})-\a(\a+1) \right\}
\end{align*}
and $
  \Delta_{\a}' = \frac{1}{4(1-|z|^{2})}\Delta_{\a}.
$
It is easy to check
\begin{equation}
\Delta_{\a}'=\Delta_{\beta}'+(\alpha-\beta)(R+\bar{R})+(\beta-\alpha)(\beta+\alpha+1).
  \label{eq:6.1}
\end{equation}
Denote by $r=|z|$ and
$$\rho=\frac{1}{2}\ln\frac{1+r}{1-r}.$$
Then
\begin{equation}
  \cosh\rho=\frac{1}{\sqrt{1-r^{2}}},\;\;\sinh\rho=\frac{r}{\sqrt{1-r^{2}}},\;\;\partial_{\rho}=(1-r^{2})\partial_{r}.
  \label{eq:6.2}
\end{equation}
Furthermore, if $f=f(\rho)$, then
\begin{equation}
  \Delta f(\rho)=\partial_{\rho}^{2}f+\left((4m-1)\coth\rho+3\tanh\rho\right)\partial_{\rho}f.
  \label{eq:6.3}
\end{equation}
By using the identity $\Delta(fg)=g\Delta f+2\langle\nabla f,\nabla g\rangle+f\Delta g$ and (\ref{eq:6.3}), we have
\begin{equation*}
  \begin{split}
    \Delta[(\cosh\rho)^{a}f]=&f\Delta (\cosh\rho)^{a}+2\langle\nabla (\cosh\rho)^{a},\nabla f\rangle+(\cosh\rho)^{a}\Delta f\\
    =&\left[(4m+a+2)a(\cosh\rho)^{a}-a(a+2)(\cosh\rho)^{a-2}\right]f\\
    &+2a(\cosh\rho)^{a-1}\sinh\rho\partial_{\rho}f+(\cosh\rho)^{a}\Delta f\;\;\;(\because\langle\nabla\rho,\nabla f\rangle=\partial_{\rho}f).
  \end{split}
\end{equation*}
i.e.
\begin{equation}
  \begin{split}
    &[\Delta-(4m+a+2)a][(\cosh\rho)^{a}f]=[\Delta-(4m+a+2)a][(1-|z|^{2})^{\frac{a}{2}}f]\\
    =&(\cosh\rho)^{a-2}\left[(\cosh\rho)^{2}\Delta+2a\tanh\rho\partial_{\rho}-a(a+2)\right]f\\
    =&(\cosh\rho)^{a-2}\left[4\Delta'_{0}+2ar\partial_{r}-a(a+2)\right]f\\
    =&(\cosh\rho)^{a-2}\left[4\Delta'_{0}+2a(R+\bar{R})-a(a+2)\right]f\;\;\;(\because R+\bar{R}=r\partial_{r})\\
    =&4(\cosh\rho)^{a-2}\Delta'_{\frac{a}{2}}f.
  \end{split}
  \label{eq:6.4}
\end{equation}

We are now ready to give the
\textbf{Proof of Theorem \ref{thm:factorization-theorem2}}.  It suffices to show the following
  \begin{align*}
    &4^{k}\left( \cosh\rho \right)^{-k-a-(2m+1)}\prod_{j=1}^{k}\left[ \Delta'_{\frac{1-a-(2m+1)}{2}} + \frac{(k+1-2j)^{2}}{4} -i\frac{k+1-2j}{2}\sqrt{\Gamma+1} \right]f\\
    &= \prod_{j=1}^{k} \left[ \Delta + (2m+1)^{2} - (a-k+2j-2)^{2} \right]\left[ \left( \cosh\rho \right)^{k-a-(2m+1)}f \right].
  \end{align*}

  We shall prove it by induction.
  For $k=1$, we have, by (\ref{eq:6.4}),
  \begin{equation}
    \begin{split}
      &\left[ \Delta + (2m+1)^{2} - (a-1)^{2} \right]\left[ \left( \cosh\rho \right)^{1-a-(2m+1)}f \right]\\
      =&4 \left( \cosh\rho \right)^{-1-a-(2m+1)}\Delta'_{\frac{1-a-(2m+1)}{2}}f
    \end{split}
    \label{eq:6.5}
  \end{equation}

  Assume  it  holds for $k$, replacing $a$ by $a-1$, we have
  \begin{equation}
    \begin{split}
      &4^{k}\left( \cosh\rho \right)^{-k+1-a-(2m+1)}\prod_{j=1}^{k}\left[ \Delta'_{\frac{2-a-(2m+1)}{2}} + \frac{(k+1-2j)^{2}}{4} -i\frac{k+1-2j}{2}\sqrt{\Gamma+1} \right]f\\
      &= \prod_{j=1}^{k} \left[ \Delta + (2m+1)^{2} - (a-1-k+2j-2)^{2} \right]\left[ \left( \cosh\rho \right)^{k+1-a-(2m+1)}f \right].
    \end{split}
    \label{eq:6.6}
  \end{equation}
  Then for $k+1$, we have, by using (\ref{eq:6.5}) and (\ref{eq:6.6}) ,
  \begin{equation}
    \begin{split}
      & \prod_{j=1}^{k+1} \left[ \Delta + (2m+1)^{2} - (a-1-k+2j-2)^{2} \right]\left[ \left( \cosh\rho \right)^{k+1-a-(2m+1)}f \right]\\
      =&\left[ \Delta + (2m+1)^{2} - (a-1+k)^{2} \right]\\
      &\left\{4^{k}\left( \cosh\rho \right)^{-k+1-a-(2m+1)}\prod_{j=1}^{k}\left[ \Delta'_{\frac{2-a-(2m+1)}{2}} + \frac{(k+1-2j)^{2}}{4} -i\frac{k+1-2j}{2}\sqrt{\Gamma+1} \right]f\right\}\\
      =&4^{k+1}\left( \cosh\rho \right)^{-k-1-a-(2m+1)}\\
      &\Delta'_{\frac{1-k-a-(2m+1)}{2}}\prod_{j=1}^{k}\left[ \Delta'_{\frac{2-a-(2m+1)}{2}} + \frac{(k+1-2j)^{2}}{4} -i\frac{k+1-2j}{2}\sqrt{\Gamma+1} \right]f.
    \end{split}
    \label{eq:6.7}
  \end{equation}

  The rest of the proof is similar to that given in \cite{ly5} by using Lemma \ref{lemma:6.1} and we omit it.
The proof of Theorem \ref{thm:factorization-theorem2} is thereby completed.\
\\

Before the proof of Lemma \ref{lemma:6.1}, we need the following:
\begin{lemma}\label{lemma:a6.1}
There holds
  \begin{equation*}
    \begin{split}
      [\Delta'_{0}, \left[R+\bar{R} \right]]=\Delta'_{0}-\frac{1}{2}(R+\bar{R})+\frac{1}{4}(R+\bar{R})^{2}
      -\frac{1}{4}\Gamma.
    \end{split}
  \end{equation*}
\end{lemma}
\begin{proof}
We compute
\begin{align*}
  D_{1} \bar D_{1} &= \left( \bar z_{j} \p_{m+j} - \bar z_{m+j} \p_{j} \right)\left( z_{i} \bar\p_{m+i} - z_{m+i}\bar\p_{i} \right)\\
  &= z_{i} \bar z_{j} \bar\p_{m+i} \p_{m+j} - \bar z_{i} \bar\p_{i} -\bar z_{j} z_{m+i} \bar\p_{i} \p_{m+j} - \bar z_{m+i}\bar \p_{m+i} - \bar z_{m+j} z_{i} \p_{j} \bar\p_{m+i} + \bar z_{m+j}z_{m+i} \p_{j} \bar\p_{i}\\
  \bar D_{1} D_{1} &= \left( z_{j} \bar \p_{m+j} - z_{m+j} \bar\p_{j} \right)\left( \bar z_{i} \p_{m+i} - \bar z_{m+i} \p_{i} \right)\\
  &= z_{j} \bar z_{i} \bar\p_{m+j} \p_{m+i} - z_{i} \p_{i} - z_{j}\bar z_{m+i}\bar \p_{m+j}\p_{i} - z_{m+i}\p_{m+i} - z_{m+j}\bar z_{i} \bar\p_{j}\p_{m+i} + z_{m+j} \bar z_{m+i} \bar\p_{j} \p_{i}
\end{align*}
and so
\begin{align*}
  -2D_{1}\bar D_{1} - 2 \bar D_{1} D_{1} & = 2(R+\bar R) - 4 \sum_{j,i=1}^{m} \left( z_{i} \bar z_{j} \bar \p_{m+i} \p_{m+j} + z_{m+i}\bar z_{m+j}  \bar\p_{i} \p_{j} \right) \\
  &+  4\sum_{i,j=1}^{m} \left( z_{i} \bar z_{m+j} \bar \p_{m+i} \p_{j} + z_{m+i}\bar z_{j} \bar\p_{i} \p_{m+j} \right)
\end{align*}

A straightforward computation provides
\begin{align*}
  \frac{1}{2}[\Delta_{\a}', R + \bar R] &=  \Delta_{0}' - (R + \bar R) + R \bar R + \sum_{i,j=1}^{m} \bar z_{m+i} z_{m+j} \p_{i} \bar \p_{j} + \bar z_{i} z_{j} \p_{m+i}\bar \p_{m+j} \\
  &- \sum_{i,j=1}^{m} \bar z_{i} z_{m+j} \p_{m+i} \bar \p_{j} + \bar z_{m+i} z_{j} \p_{i} \bar \p_{m+j}\\
  &= \Delta_{0}' - (R+\bar R) + R\bar R + \frac{1}{4} \left( 2 D_{1} \bar D_{1} + 2 \bar D_{1} D_{1} + 2(R + \bar R) \right)\\
  &= \Delta_{0} ' - \frac{1}{2}(R + \bar R) + R \bar R + \frac{1}{2} ( D_{1} \bar D_{1} + \bar D_{1} D_{1})\\
  &= \Delta_{0}' + R \bar R - \frac{1}{2}(R + \bar R) + \frac{1}{4} \left( (R-\bar R)^{2} - \Gamma \right).
\end{align*}
The results follows.

\end{proof}

By Lemma \ref{lemma:a6.1}, it is easy to check
  \begin{equation} \label{eq:a6.7}
    \begin{split}
      [\Delta'_{\alpha},\Delta'_{\beta}]
      =(\alpha-\beta)[\left[R+\bar{R}, \Delta'_{0}\right]=2(\beta-\alpha)\left(\Delta'_{0}-\frac{1}{2}(R+\bar{R})+\frac{1}{4}(R+\bar{R})^{2}
      -\frac{1}{4}\Gamma\right).
    \end{split}
  \end{equation}
  We shall frequently use the fact
  \[
  [\Gamma,\Delta_{\alpha}']=\Gamma\Delta_{\alpha}'-\Delta_{\alpha}'\Gamma=0.
  \]

\begin{lemma}\label{lemma:6.1}
  There holds
  \begin{align*}
    &\Delta'_{\frac{1-k-a}{2}}\left\{\left[ \Delta'_{\frac{2-a}{2}} + \frac{(k-1)^{2} }{4} \right]^{2}-\frac{(k-1)^{2}}{4}\left\{\Gamma+1\right\} \right\}f\\
    =&\left\{\left[ \Delta'_{\frac{1-a}{2}} +\frac{ k^{2}}{4} \right]^{2}-\frac{k^{2}}{4} \left\{\Gamma+1\right\}\right\}
    \Delta'_{\frac{3-k-a}{2}}f.
  \end{align*}
\end{lemma}

\begin{proof} We compute, by using (\ref{eq:6.1}) and Lemma \ref{lemma:a6.1}
  \begin{align*}
    &\Delta'_{\frac{1-k-a}{2}}\left[ \Delta'_{\frac{2-a}{2}} + \frac{(k-1)^{2} }{4} \right]\\
=&\left(\Delta'_{\frac{1-a}{2}}+\frac{k^{2}}{4}-\frac{k}{2}(R+\bar{R})+\frac{k}{2}(2-a-k)\right)\left( \Delta'_{\frac{1-a}{2}} + \frac{k^{2} }{4} +\frac{1}{2}(R+\bar{R})+\frac{a-2-k}{2}\right)\\
=&\left(\Delta'_{\frac{1-a}{2}}+\frac{k^{2}}{4}\right)^{2}+\frac{1}{2}
\left(\Delta'_{\frac{1-a}{2}}+\frac{k^{2}}{4}\right)(R+\bar{R})-\frac{k}{2}(R+\bar{R})
\left(\Delta'_{\frac{1-a}{2}}+\frac{k^{2}}{4}\right)\\
&+\frac{-k^{2}+(1-a)k+a-2}{2}\left(\Delta'_{\frac{1-a}{2}}+\frac{k^{2}}{4}\right)+\frac{k(2-a)}{2}(R+\bar{R})+\frac{k(2-a-k)(a-2-k)}{4}\\
=&\left(\Delta'_{\frac{1-a}{2}}+\frac{k^{2}}{4}\right)^{2}+\frac{1-k}{2}
\left(\Delta'_{\frac{1-a}{2}}+\frac{k^{2}}{4}\right)(R+\bar{R})+\frac{k}{2}\left[\Delta'_{\frac{1-a}{2}} + \frac{k^{2} }{4} ,R+\bar{R}\right]
-\frac{k}{4}(R+\bar{R})^{2}\\
&+\frac{-k^{2}+(1-a)k+a-2}{2}\left(\Delta'_{\frac{1-a}{2}}+\frac{k^{2}}{4}\right)+\frac{k(2-a)}{2}(R+\bar{R})+\frac{k(2-a-k)(a-2-k)}{4}\\
=&\left(\Delta'_{\frac{1-a}{2}}+\frac{k^{2}}{4}\right)^{2}+\frac{1-k}{2}
\left(\Delta'_{\frac{1-a}{2}}+\frac{k^{2}}{4}\right)(R+\bar{R})+\frac{k}{2}\left[\Delta'_{0} ,R+\bar{R}\right]-\frac{k}{4}(R+\bar{R})^{2}\\
&+\frac{-k^{2}+(1-a)k+a-2}{2}\left(\Delta'_{\frac{1-a}{2}}+\frac{k^{2}}{4}\right)+\frac{k(2-a)}{2}(R+\bar{R})+\frac{k(2-a-k)(a-2-k)}{4}\\
=&\left(\Delta'_{\frac{1-a}{2}}+\frac{k^{2}}{4}\right)^{2}+\frac{1-k}{2}
\left(\Delta'_{\frac{1-a}{2}}+\frac{k^{2}}{4}\right)(R+\bar{R})+k\left(\Delta_{0}'-\frac{1}{2}(R+\bar{R})-\frac{1}{4}\Gamma\right)\\
&+\frac{-k^{2}+(1-a)k+a-2}{2}\left(\Delta'_{\frac{1-a}{2}}+\frac{k^{2}}{4}\right)+\frac{k(2-a)}{2}(R+\bar{R})+\frac{k(2-a-k)(a-2-k)}{4}\\
=&\left(\Delta'_{\frac{1-a}{2}}+\frac{k^{2}}{4}\right)^{2}+\frac{1-k}{2}
\left(\Delta'_{\frac{1-a}{2}}+\frac{k^{2}}{4}\right)(R+\bar{R})+k\left(\Delta_{\frac{1-a}{2}}'+\frac{k^{2}}{4}\right)-\frac{k}{4}(\Gamma+1)\\
&+\frac{-k^{2}+(1-a)k+a-2}{2}\left(\Delta'_{\frac{1-a}{2}}+\frac{k^{2}}{4}\right)\\
=&\left(\Delta'_{\frac{1-a}{2}}+\frac{k^{2}}{4}\right)^{2}+\frac{1-k}{2}
\left(\Delta'_{\frac{1-a}{2}}+\frac{k^{2}}{4}\right)(R+\bar{R})-\frac{k}{4}(\Gamma+1)\\
&+\frac{-k^{2}+(3-a)k+a-2}{2}\left(\Delta'_{\frac{1-a}{2}}+\frac{k^{2}}{4}\right)\\
=&\left(\Delta'_{\frac{1-a}{2}}+\frac{k^{2}}{4}\right)^{2}+\frac{1-k}{2}
\left(\Delta'_{\frac{1-a}{2}}+\frac{k^{2}}{4}\right)(R+\bar{R})-\frac{k}{4}(\Gamma+1)\\
&-\frac{(k-1)(k+a-2)}{2}\left(\Delta'_{\frac{1-a}{2}}+\frac{k^{2}}{4}\right)\\
=&\left(\Delta'_{\frac{1-a}{2}}+\frac{k^{2}}{4}\right)^{2}+\frac{1-k}{2}
\left(\Delta'_{\frac{1-a}{2}}+\frac{k^{2}}{4}\right)(R+\bar{R}+k+a-2)-\frac{k}{4}(\Gamma+1).
  \end{align*}
Therefore, we have
 \begin{align*}
&\Delta'_{\frac{1-k-a}{2}}\left[ \Delta'_{\frac{2-a}{2}} + \frac{(k-1)^{2} }{4} \right]^{2}-\left(\Delta'_{\frac{1-a}{2}}+\frac{k^{2}}{4}\right)^{2}
\Delta'_{\frac{3-k-a}{2}}\\
=&\left(\Delta'_{\frac{1-a}{2}}+\frac{k^{2}}{4}\right)^{2}\left( \Delta'_{\frac{2-a}{2}} + \frac{(k-1)^{2} }{4}-\Delta'_{\frac{3-k-a}{2}}\right)\\
&+\frac{1-k}{2}
\left(\Delta'_{\frac{1-a}{2}}+\frac{k^{2}}{4}\right)(R+\bar{R}+k+a-2)\left[ \Delta'_{\frac{2-a}{2}} + \frac{(k-1)^{2} }{4} \right]\\
&-\frac{k}{4}(\Gamma+1)\left[ \Delta'_{\frac{2-a}{2}} + \frac{(k-1)^{2} }{4} \right]\\
=&\frac{k-1}{2}\left(\Delta'_{\frac{1-a}{2}}+\frac{k^{2}}{4}\right)^{2}(R+\bar{R}+k+a-4)\\
&-
\frac{k-1}{2}
\left(\Delta'_{\frac{1-a}{2}}+\frac{k^{2}}{4}\right)(R+\bar{R}+k+a-2)\left[ \Delta'_{\frac{2-a}{2}} + \frac{(k-1)^{2} }{4} \right]
-\frac{k}{4}(\Gamma+1)\left[ \Delta'_{\frac{2-a}{2}} + \frac{(k-1)^{2} }{4} \right]\\
=&\frac{k-1}{2}\left(\Delta'_{\frac{1-a}{2}}+\frac{k^{2}}{4}\right)\left\{
\left(\Delta'_{\frac{1-a}{2}}+\frac{k^{2}}{4}\right)(R+\bar{R}+k+a-4)-\right.\\
&
\left.(R+\bar{R}+k+a-2)\left[ \Delta'_{\frac{2-a}{2}} + \frac{(k-1)^{2} }{4} \right]\right\}-\frac{k}{4}(\Gamma+1)\left[ \Delta'_{\frac{2-a}{2}} + \frac{(k-1)^{2} }{4} \right].
  \end{align*}
On the other hand,
 \begin{align*}
&\left(\Delta'_{\frac{1-a}{2}}+\frac{k^{2}}{4}\right)(R+\bar{R}+k+a-4)-(R+\bar{R}+k+a-2)\left[ \Delta'_{\frac{2-a}{2}} + \frac{(k-1)^{2} }{4} \right]\\
=&\left[\Delta'_{\frac{1-a}{2}}+\frac{k^{2}}{4}, R+\bar{R}+k+a-4\right]+(R+\bar{R}+k+a-4)
\left(\Delta'_{\frac{1-a}{2}}+\frac{k^{2}}{4}\right)\\
&-(R+\bar{R}+k+a-2)\left[ \Delta'_{\frac{2-a}{2}} + \frac{(k-1)^{2} }{4} \right]\\
=&\left[\Delta'_{0}, R+\bar{R}\right]+(R+\bar{R}+k+a-2)\left(\Delta'_{\frac{1-a}{2}}+\frac{k^{2}}{4}-
\Delta'_{\frac{2-a}{2}} - \frac{(k-1)^{2} }{4} \right)-2\left(\Delta'_{\frac{1-a}{2}}+\frac{k^{2}}{4}\right)\\
=&\left[\Delta'_{0}, R+\bar{R}\right]-2\Delta_{0}'-\frac{1}{2}(R+\bar{R})^{2}+(R+\bar{R})-\frac{1}{2}=-\frac{1}{2}(\Gamma+1).
\end{align*}
Combing both above inequalities yields
 \begin{align*}
&\Delta'_{\frac{1-k-a}{2}}\left[ \Delta'_{\frac{2-a}{2}} + \frac{(k-1)^{2} }{4} \right]^{2}-\left(\Delta'_{\frac{1-a}{2}}+\frac{k^{2}}{4}\right)^{2}
\Delta'_{\frac{3-k-a}{2}}\\
=&\frac{\Gamma+1}{4} \left[-(k-1)\left(\Delta'_{\frac{1-a}{2}}+\frac{k^{2}}{4}\right)-k\left( \Delta'_{\frac{2-a}{2}} + \frac{(k-1)^{2} }{4} \right)\right] \\
=&\frac{\Gamma+1}{4}\left[(k-1)^{2}\Delta'_{\frac{1-k-a}{2}}-k^{2}\Delta'_{\frac{3-k-a}{2}}\right].
\end{align*}
  The proof of Lemma \ref{lemma:6.1} is thereby completed.

  \end{proof}

\section{Funk-Hecke Formulas}

\subsection{The Funk-Hecke formula for the quaternionic hyperbolic space}

We note that the Funk-Hecke formula on the CR sphere was established by Frank and Lieb \cite{fr}.
The main source for the following is \color{red} \cite{jo1, jo, MR3449396,MR3424625} \color{black} where they extend Frank and Lieb's formula.
We begin by recalling the Funk-Hecke formulas for the quaternionic case.
We recall that $L^{2}(S^{4m-1})$ may be decomposed into the $U(2m)$-irreducibles decomposition
\[
  L^{2}(S^{4m-1})=\bigoplus_{j,k\geq0}\mcH_{j,k},
\]
where $\mcH_{j,k}$ consists of the Euclidean harmonic homogeneous polynomials in the complex variables $(z,\bar{z})$ and of bidegree $(j,k)$.
Recalling that $H_{\Q}^{m}=Sp(m,1)/Sp(m)\times{Sp(1)}$, the appropriate irreducible decomposition is into $Sp(m)\times{Sp(1)}$-irreducibles, and is given as follows:
\begin{equation}
  L^{2}(S^{4m-1})=\bigoplus_{j\geq{k}\geq0}V_{j,k},
  \label{eq:spherical-harmonic-decomposition}
\end{equation}
where $V_{j,k}\subset\mcH_{j,k}$ are the so-called $(j,k)-$bispherical harmonic spaces generated by the $Sp(m)\times{Sp(1)}$ action on a zonal harmonic polynomial (see \cite[Theorem 3.1 (4)]{jo}).

We recall the following quaternionic Funk-Hecke formula of Christ, Liu and Zhang (\cite[Lemma 5.4]{MR3424625}).
In the following, $P_{k}^{\a,\b}(t)$ denotes a Jacobi polynomial of degree $k$.

\begin{theoremalph}
  \label{thmalph:funk-hecke-formula}
  Let $K$ be an $L^{1}$ integrable function on the unit ball $\B_{\Q}^{1}$ in $\Q$.
  Then, any integral operator on $S^{4n+3}$ with kernel given by $K(\gen{\zeta, \bar\eta}_{\Q})$ is diagonal with respect to the decomposition \eqref{eq:spherical-harmonic-decomposition}, and the eigenvalue $\la_{j,k}(K)$ on $V_{j,k}$ is given by
  \begin{equation}
    \begin{aligned}
      \la_{j,k}(K)&=\frac{2\pi^{2n}k!}{(j-k+1)!(k+2n-1)!}\int\limits_{0}^{\frac{\pi}{2}}\left( \sin\theta \right)^{4n-1}\left( \cos\theta \right)^{j-k+3}P_{k}^{(2m-1,j-k+1)}\left( \cos2\theta \right)d\theta\\
      &\times\int\limits_{S^{3}}K\left( \cos\theta{u} \right)\frac{\sin\left( j-k+1 \right)\phi}{\sin\phi}du,
    \end{aligned}
    \label{eq:funk-hecke-formula}
  \end{equation}
  where $\Re{u}=\cos\phi$ ($\phi\in[0,\pi]$) and $du$ is the round measure on $S^{3}=\p{B_{\Q}^{1}}$.
\end{theoremalph}

Using Theorem \ref{thmalph:funk-hecke-formula} and taking inspiration from the proof of Lemma 5.5 of \cite{MR3424625}, we obtain the following integral formula which will be used later.

\begin{proposition}
  \label{prop:spherical-integral-eigenvalue}
  If $-\frac{1}{2}<\a<\oo$ and $0<r<1$, then
  \begin{equation}
    \int\limits_{S^{4n+3}}\frac{1}{|1-\gen{r\xi,\zeta}_{\Q}|^{2\a}}d\sigma(\eta)=\frac{2\pi^{2n+2}}{(2n+1)!}{}_{2}F_{1}\left( \a, \a-1; 2n+2; r^{2} \right).
    \label{eq:funk-hecke-application}
  \end{equation}
\end{proposition}

\begin{proof}

  Define the kernel $K_{r}(q) = |1-rq|^{-2\a}$ on $\B_{\Q}^{1}$ and observe that \eqref{eq:funk-hecke-application} may be understood as an integral operator on $S^{4n+3}$ with kernel $K_{r}(\gen{\zeta,\bar\eta}_{\Q})$ applied to the constant function $1 \in V_{0,0}$.
  Therefore, we may apply the Funk-Hecke formula \eqref{eq:funk-hecke-formula} to $K_{r}$ with $j=k=0$ to obtain
  \begin{align*}
    \la_{0,0}(K_{r}) & = \frac{8\pi^{2n+1}}{(2n-1)!} \int\limits_{0}^{\frac{\pi}{2}} \sin^{4n-1}\theta \cos^{3} \theta P_{0}^{(2n-1,1)}(\cos2\theta) d\theta\\
    &\times \int\limits_{0}^{\pi}\left( 1 + r^{2}\cos^{2}\theta - 2r \cos\phi \cos\theta \right)^{-\a} \sin^{2}\phi d\phi,
  \end{align*}
  where we have used
  \[
    |1-rq|^{2} = 1+r^{2}|q| - 2 \Re q \implies K(\cos\theta u) = |1+r^{2}\cos^{2}\theta - 2r \cos\theta \cos\phi|^{-\a}
  \]
  and
  \begin{align*}
    du &= \sin^{2} \phi \sin \phi' d\phi d\phi' d\phi'', \quad \phi,\phi' \in [0,\pi], \phi'' \in [0,2\pi]\\
    &\int\limits_{0}^{\pi} \int\limits_{0}^{2\pi} \sin\phi' d\phi' d\phi'' = 4\pi.
  \end{align*}
  Note also $P_{0}^{(2n-1,1)}\equiv1$.
  Using the cosine integral (see \cite{fr}, (5.11))
  \[
    \int\limits_{-\pi}^{\pi}\left( 1-2r\cos\phi + r^{2} \right)^{-\a}e^{i\ell\phi}d\phi = \frac{2\pi}{\Gamma^{2}(\a)}\sum_{\mu\geq0} r^{\ell+2\mu} \frac{\Gamma(\a+\mu)\Gamma(\a+\ell+\mu)}{\mu! (\ell+\mu)!}
  \]
  for $\ell \in \N$, that the integrand in even and that $\sin^{2}\phi = \frac{1}{2} (1-\cos2\phi)$, we have
  \begin{align*}
    &\int\limits_{0}^{\pi} \left( 1 + r^{2} \cos^{2}\theta - 2r \cos\theta \cos\phi \right)^{-\a} \sin^{2}\phi d\phi \\
    &=  \frac{\pi}{2\Gamma^{2}(\a)} \sum_{\mu\geq0} r^{2\mu} \cos^{2\mu}\theta \frac{\Gamma^{2}(\a)}{(\mu!)^{2}} - r^{2+2\mu} \cos^{2+2\mu}\theta \frac{\Gamma(\a)\Gamma(2+\a)}{\mu!(\mu+2)!}.
  \end{align*}
  Consequently, there holds
  \begin{align*}
    \la_{0,0}(K_{r}) & = \frac{4\pi^{2n+2}}{(2n-1)! \Gamma^{2}(\a)} \sum_{\mu\geq0} \frac{\Gamma(\mu+\a)}{\mu!}\left( r^{2\mu}\frac{\Gamma(\mu+\a)}{\mu!} \int\limits_{0}^{\frac{\pi}{2}} \sin^{4n-1}\theta \cos^{3+2\mu}\theta d\theta \right.\\
    &- \left. r^{2+2\mu} \frac{\Gamma(\mu+\a+2)}{(2+\mu)!} \int\limits_{0}^{\frac{\pi}{2}}\sin^{4n-1}\theta \cos^{5+2\mu}\theta d\theta \right).
  \end{align*}

  Letting $t=\cos 2\theta$ and observing $dt = -4 \sin\theta \cos\theta d\theta$, $\cos^{2}\theta = \frac{1}{2}(1+t)$ and $\sin^{2}\theta = \frac{1}{2}(1-t)$, we find
  \begin{align*}
    \int\limits_{0}^{\frac{\pi}{2}} \sin^{4n-1}\theta \cos^{\ell+3+2\mu}\theta d\theta &= \frac{1}{4} \int\limits_{-1}^{1}\left( \sin^{2}\theta \right)^{2n-1} \left( \cos^{2}\theta \right)^{\frac{\ell}{2}+1+\mu}dt\\
    &=2^{-2-2n-\mu-\frac{\ell}{2}}\int\limits_{-1}^{1}(1+t)^{\frac{\ell}{2}+2+\mu-1}(1-t)^{2n-1}dt\\
    &=\frac{1}{2}B\left( \frac{\ell}{2}+2+\mu,2n \right)\\
    &=\frac{\Gamma(\frac{\ell}{2}+2+\mu)\Gamma(2n)}{2\Gamma(\frac{\ell}{2}+2+\mu+2n)},
  \end{align*}
  where $B(x,y)$ is the beta function.

  It follows that
  \begin{align*}
    \la_{0,0}(K_{r})  &= \frac{4\pi^{2n+2}}{(2n-1)!\Gamma^{2}(\a)} \sum_{\mu\geq0}\frac{\Gamma(\mu+\a)}{\mu!}\left( r^{2\mu}\frac{\Gamma(\mu+\a)}{\mu!} \frac{\Gamma(2+\mu)\Gamma(2n)}{2\Gamma(2+\mu+2n)}\right.\\
    &- \left. r^{2\mu+2}\frac{\Gamma(\mu+\a+2)}{(2+\mu)!}\frac{\Gamma(3+\mu)\Gamma(2n)}{2\Gamma(3+\mu+2n)} \right)\\
    &=\frac{2\pi^{2n+2}}{\Gamma^{2}(\a)}\left( \sum_{\mu\geq1}\left[ \frac{\Gamma^{2}(\mu+\a) (\mu+1)!}{(\mu!)^{2} (\mu+1+2n)!} - \frac{\Gamma(\mu-1+\a) \Gamma(\mu+\a+1)}{(\mu-1)! (\mu+1+2n)!} \right]r^{2\mu} + \frac{\Gamma^{2}(\a)}{(2n+1)!} \right)\\
    &=(\a-1) \frac{2\pi^{2n+2}}{\Gamma^{2}(\a)} \sum_{\mu\geq0}\frac{\Gamma(\mu+\a) \Gamma(\mu-1+\a)}{\mu! (\mu+1+2n)!} r^{2\mu}\\
    &=\frac{2\pi^{2n+2}}{(2n+1)!} \sum_{\mu\geq0} \frac{(\a)_{\mu} (\a-1)_{\mu}}{(2n+2)_{\mu}} \frac{r^{2\mu}}{\mu!}\\
    &= \frac{2\pi^{2n+2}}{(2n+1)!}{}_{2}F_{1}(\a,\a-1;2n+2;r^{2}).
  \end{align*}
  This is the desired identity.
\end{proof}

\subsection{The Funk-Hecke formula for the Cayley hyperbolic plane}

We now discuss the Funk-Hecke formula for the octonionic case.
We recall that $L^{2}(S^{15})$ may be decomposed into the $\operatorname{Spin}(9)$-irreducible decomposition
\begin{equation}
  \label{eq:octonionic-spherical-harmonic-decomposition}
  L^{2}(S^{15}) = \bigoplus_{j\geq k \geq 0} W_{j,k}
\end{equation}
where $W_{j,k}$ is the so-called $(j,k)$-bispherical harmonic subspace, which is a finite dimensional space spanned by elements from the cyclic action of $\operatorname{Spin}(9)$ on zonal harmonics $Z_{j,k}(\zeta)$ (see \cite{jo} or \cite[eq. 2.12]{MR3449396} for precise formula).

We point out that the Funk-Hecke formula given in \cite{MR3449396} assumes the kernel function $K$ is of the form $K(\zeta \cdot \bar \eta)$, where, if $\zeta = (\zeta_{1},\zeta_{2}), \eta = (\eta_{1},\eta_{2}) \in \C a^{2}$, then $\zeta \cdot\bar \eta = \zeta_{1} \bar \eta_{1} + \zeta_{2} \bar \eta_{2}$.
Considering kernel functions of this form are due to their consideration of the natural distance function $|1-\zeta\cdot\bar\eta|$ on the sphere $S^{15}$.
 However, taking into consideration the geometry of the Cayley plane $H_{\C a}$ and the non-associativity of $\C a$, it is more appropriate for our purposes to consider kernels of the form $K( \Phi_{\C a}(\zeta,\eta))$ or $K(\Psi_{\C a}(\zeta,\eta))$ since $\Phi_{\C a}(\zeta,\eta)$ and $\Psi_{\C a}(\zeta,\eta)$ are octonionic analogues of $|\gen{\cdot,\cdot}_{\mbF}|^{2}$ and $|1-\gen{\cdot,\cdot}_{\mbF}|^{2}$, respectively.
As a result,  we will establish the following Funk-Hecke formulas which are more suitable for our purposes.

\begin{theorem}
  \label{thm:adapted-funk-hecke-formula}
  Suppose $K(\Phi_{\C a}(\zeta,\eta))$ is such that the following integral exists.
  Then the integral operator with kernel $K(\Phi_{\C a}(\zeta,\eta))$ is diagonal with respect to the bispherical decomposition harmonic decomposition \eqref{eq:octonionic-spherical-harmonic-decomposition}, and the eigenvalue on $W_{j,k}$ is given by
  \begin{align*}
    &\la_{j,k}(K) = \frac{15\pi^{4}k!}{(k+3)!} \int_{0}^{\frac{\pi}{2}} \cos ^{j-k+7} \theta \sin ^{7} \theta P_{k}^{(3,3+j-k)}(\cos 2\theta) d \theta \int_{S} K(\Psi_{\C a}\left( (1,0),(\bar u \cos \theta,0 ) \right))\\
    &\times \left( a_{j,k}^{0} \cos(j-k)\phi + a_{j,k}^{1} \cos(j-k+2)\phi + a_{j,k}^{2} \cos(j-k+4)\phi \right.\\
    &\left.+a_{j,k}^{3}\cos(j-k+6)\phi \right) du,
  \end{align*}
  where $\Re u = \cos \phi$ ($\phi \in [0,\pi)$), $du$ is the standard surface measure on $S$, the unit sphere in $\C a$, $P_{k}^{(3,3+j-k)}(z)$ is the Jacobi polynomial of order $k$ associated to the weight $(1-z)^{3}(1+z)^{3+j-k}$ and
    \begin{align*}
      a_{j,k}^{0} &= \frac{1}{8}\frac{1}{j-k+3} - \frac{1}{4}\frac{1}{j-k+2} + \frac{1}{8}\frac{1}{j-k+1}\\
      a_{j,k}^{1} &= \frac{3}{8}\frac{1}{j-k+3} - \frac{1}{4}\frac{1}{j-k+4} - \frac{1}{8}\frac{1}{j-k+1}\\
      a_{j,k}^{2} &= -\frac{3}{8}\frac{1}{j-k+3} + \frac{1}{4}\frac{1}{j-k+2} + \frac{1}{8}\frac{1}{j-k+5}\\
      a_{j,k}^{3} &= -\frac{1}{8}\frac{1}{j-k+3} + \frac{1}{4}\frac{1}{j-k+4} - \frac{1}{8}\frac{1}{j-k+5}.
    \end{align*}
\end{theorem}

\begin{proof}
  Since the latter half of the proof is the same as the proof of Lemma 3.3 in \cite{MR3449396}, we shall only point out the needed adaptation.

  We have from Schur's lemma and the irreducibility of the $W_{j,k}$ that the integral operator with kernel $K(\Psi_{\C a}(\zeta,\eta))$ is diagonal.
  Let $\la_{j,k}$ denote the eigenvalue corresponding to the subspace $W_{j,k}$.
  Letting $Y_{j,k}^{\mu}$, $1\leq \mu \leq \dim W_{j,k}$, be a normalized orthogonal basis of $W_{j,k}$, we then have
  \[
    \int_{S^{15}} K(\Psi_{\C a}(\zeta,\eta)) Y_{j,k}^{\mu}(\eta) d\sigma = \la_{j,k} Y_{j,k}^{\mu}(\zeta).
  \]
  Letting
  \[
    Z_{j,k} (\zeta,\eta) = Z_{j,k}(\zeta \cdot\bar\eta) = \sum_{\mu =1}^{\dim W_{j,k}} Y_{j,k}^{\mu}(\zeta) \overline{Y_{j,k}^{\mu}(\eta)}
  \]
  be the reproducing kernel of the projection onto $W_{j,k}$, we have
  \[
    \int_{S^{15}} K(\Psi_{\C a}(\zeta, \eta)) Z_{j,k}(\eta \cdot \bar\zeta) d\eta= \la_{j,k} Z_{j,k}(1).
  \]
  Here $Z_{j,k}(1)$ denotes the aforementioned zonal harmonic $Z_{j,k}(\zeta)$ evaluated at $\zeta = 1$.
  All that is needed now is to observe that $K(\Psi_{\C a}(\zeta,\eta))$ and $Z_{j,k}(\eta , \zeta)$ are invariant under the action of $Spin(9)$.
  Indeed, if this were the case, then we would obtain
  \begin{align*}
    \la_{j,k} &= Z_{j,k}(1)^{-1} \int_{S^{15}} K(\Psi_{\C a}(\zeta, \eta)) Z_{j,k}(\eta \cdot \bar\zeta) d\sigma \\
    &= Z_{j,k}(1)^{-1} \int_{S^{15}} K\left( \Psi_{\C a}\left( (1,0),\eta \right) \right) Z_{j,k}\left( (1,0),\eta \right) d\eta,
  \end{align*}
  The remainder of the proof follows as the proof of Lemma 3.3 in \cite{MR3449396}.

  That $K(\Psi_{\C}(\zeta,\eta))$ is $Spin(9)$-invariant follows from the $Spin(9)$-invariance of $\Phi_{\C a}(\zeta,\eta)$.
    Therefore,
  \[
    \int_{S^{15}} Z_{j,k}(A \zeta, A \eta) {Y_{j,k}^{\mu}(\eta)} d\eta = \int_{S^{15}}Z_{j,k}(A \zeta, \eta) {Y_{j,k}^{\mu}(A^{-1} \eta)} d\eta = Y_{j,k}^{\mu}(\zeta),
  \]
  which shows that $Z_{j,k}(A \zeta, A \eta ) = Z_{j,k}(\zeta,\eta)$ by the uniqueness of the representation of a linear functional.
\end{proof}

Lastly we state and prove the octonionic analogue of Proposition \ref{prop:spherical-integral-eigenvalue}.

\begin{proposition}
  \label{prop:spherical-integral-eigenvalue-octonionic}
  If $-\frac{1}{2}<\a<\oo$ and $0<r<1$, then
  \begin{equation}
    \label{eq:funk-hecke-application-octonionic}
    \int\limits_{S^{4n+3}}\frac{1}{\Psi_{\C a}(r\xi,\zeta)^{\a}}d\sigma(\eta)= \frac{2\pi^{8}}{7!}{}_{2}F_{1}\left( \a,\a-3;8;r^{2} \right).
  \end{equation}
\end{proposition}

\begin{proof}
  The proof follows similarly to the proof of Proposition \ref{prop:spherical-integral-eigenvalue} by applying Theorem \ref{thm:adapted-funk-hecke-formula} to the kernel $\Psi_{\C a}(r\xi,\eta)^{-\a}$.
  It should be pointed out that
  \[
    \Psi_{\C a}\left( (r,0),(\bar u \cos\theta,0) \right) = 1 - 2 r \cos \phi \cos \theta + r^{2} \cos^{2} \theta
  \]
  since $\Re u = \cos \phi$.
\end{proof}

  \section{Kernel Estimates}
  \label{sec:kernel-estimates}

  We recall that the heat kernel $e^{t\Delta}$ on $H_{\Q}^{m}$ is given by the following formula:
\begin{align*}
  e^{t\Delta}&=c_{m}t^{-\frac{1}{2}}e^{-(2m+1)^{2}t}\int\limits_{\rho}^{\oo}\frac{\sinh2r}{\sqrt{\cosh2r-\cosh2\rho}}\left( -\frac{1}{\sinh2r}\frac{\p}{\p{r}} \right)^{2}\left( -\frac{1}{\sinh r}\frac{\p}{\p{r}} \right)^{2m-2}e^{-\frac{r^{2}}{4t}}dr,
\end{align*}
where $c_{m}=2^{-2m+\frac{3}{2}}\pi^{-2m-\frac{1}{2}}$.
The heat kernel $e^{t\Delta}$ on $H_{\C a}$ is given by
\begin{align*}
  e^{t\Delta}&= c_{o} t^{-\frac{1}{2}} e^{-11^{2}t} \int_{\rho}^{\oo} \frac{\sinh 2r}{\sqrt{\cosh 2r - \cosh 2\rho}} \left( -\frac{1}{\sinh 2r} \frac{\p}{\p r} \right)^{4} \left( - \frac{1}{\sinh r} \frac{\p}{\p r} \right)^{4} e^{-\frac{r^{2}}{4t}}dr,
\end{align*}
where $c_{o}= 2^{-\frac{9}{2}} \pi^{-\frac{17}{2}}$.
Letting $h_{t}(\rho,2\tilde{m}+1)$ denote the heat kernel on the odd dimensional real hyperbolic space $H_{\R}^{2\tilde{m}+1}$, we recall also that
\begin{equation}
  h_{t}(\rho,2\tilde{m}+1)=b_{\tilde{m}}t^{-\frac{1}{2}}e^{-\tilde{m}^{2}t}\left( -\frac{1}{\sinh\rho}\frac{\p}{\p\rho} \right)^{\tilde{m}}e^{-\frac{\rho^{2}}{4t}},
  \label{eq:heat-kernel-on-real-hyperbolic-space}
\end{equation}
where $b_{\tilde{m}}=2^{-\tilde{m}-1}\pi^{-\tilde{m}-\frac{1}{2}}$.
See for example \cite{MR2018351} and \cite{lo} for these formulas.

It will be useful to write $e^{t\Delta}$ in terms of $h_{t}$, and this can be done as follows.
We consider $H_{\Q}^{m}$ first.
Observe that, if $\tilde{m}=2m-2$, then
\[
  e^{-(2m+1)^{2}t}=e^{(-12m+3)t}e^{-\tilde{m}^{2}t},
\]
and so
\begin{equation}
  \begin{aligned}
    e^{t\Delta}&=\frac{c_{m}}{b_{2m-2}}\int\limits_{\rho}^{\oo}\frac{\sinh2r}{\sqrt{\cosh2r-\cosh2\rho}}\left( -\frac{1}{\sinh2r}\frac{\p}{\p{r}} \right)^{2}e^{(-12m+3)t}h_{t}(r,4m-3) dr.
  \end{aligned}
  \label{eq:quaternionic-heat-kernel-in-terms-of-hyperbolic-heat-kernel}
\end{equation}
Similarly, on $H_{\C a}$, there holds (by setting $\tilde{m} = 4$)
\begin{equation}
  \begin{aligned}
    e^{t\Delta}&=\frac{c_{o}}{b_{4}}\int\limits_{\rho}^{\oo}\frac{\sinh2r}{\sqrt{\cosh2r-\cosh2\rho}}\left( -\frac{1}{\sinh2r}\frac{\p}{\p{r}} \right)^{4}e^{-105t}h_{t}(r,9) dr.
  \end{aligned}
  \label{eq:octonionic-heat-kernel-in-terms-of-hyperbolic-heat-kernel}
\end{equation}

We now recall the Bessel-Green-Riesz functions.
For the sake of notational convenience, we write
\begin{center}
  \begin{tabular}{ll}
    $k_{\zeta,\g}=\left( -\Delta-\frac{Q^{2}}{4}+\zeta^{2} \right)^{-\frac{\g}{2}}$&for $0<\g<\dim_{\R} H_{\mbF}^{m}$ and $\zeta>0$\\
    $k_{\g}=\left( -\Delta-\frac{Q^{2}}{4} \right)^{-\frac{\g}{2}}$&for $0<\g<3$
  \end{tabular}.
\end{center}
In \cite[page 1083, (iii)]{anj}, Anker and Ji established the following asymptotics for $k_{\zeta,\g}$ and $k_{\g}$:

\begin{equation}
  \begin{tabular}{ll}
    $k_{\zeta,\g}\sim\rho^{\frac{\g-2}{2}}e^{-\zeta\rho-\frac{Q}{2}\rho}$&for $\rho\geq1$\\
    $k_{\g}\sim\rho^{\g-2}e^{-\frac{Q}{2}\rho}$&for $\rho\geq1$.
  \end{tabular}.
  \label{eq:anker-ji-large-distance-asymptotics}
\end{equation}

We will need several technical lemmas to obtain small distance estimates of $k_{\zeta}$.
We state them now.
The first estimate is a small distance estimate for the Bessel-Green-Riesz kernel on the real hyperbolic space $H_{\R}^{k}$ (see \cite[Lemma 3.2]{LiLuy2}).

\begin{lemmaalph}
  \label{lemalph:liluyang-small-bessel-green-riesz-kernel-estimate}
  Let $k\geq3$ and $0<\g<3$\color{red} . \color{black}
  If $0<\rho<1$, then
  \begin{align*}
    \left( -\Delta_{H_{\R}^{k}}-\left( \frac{k-1}{2} \right)^{2} \right)^{-\frac{\g}{2}}=\frac{1}{\g_{k}(\g)}\frac{1}{\rho^{k-\g}}+O\left( \frac{1}{\rho^{k-\g-1}} \right),
  \end{align*}
  where
  \[
    \g_{k}(\g)=\frac{\pi^{\frac{k}{2}}2^{\g}\Gamma\left( \frac{\g}{2} \right)}{\Gamma\left( \frac{k-\g}{2} \right)}.
  \]
\end{lemmaalph}

The second lemma is an exact evaluation of a hyperbolic trigonometric integral (see \cite[Lemma 3.2]{ly5}).

\begin{lemmaalph}
  \label{lemalph:luyang-hypertrig-integral}
  Let $\b>0$ and $\rho>0$.
  Then
  \[
    \int_{\rho}^{\oo}\frac{\cosh{r}}{\left( \sinh{r} \right)^{\b}}\frac{1}{\sqrt{\cosh2r-\cosh2\rho}}dt=\frac{\Gamma\left( \frac{1}{2} \right)\Gamma\left( \frac{\b}{2} \right)}{2\sqrt{2}\Gamma\left( \frac{1+\b}{2} \right)}\frac{1}{\left( \sinh\rho \right)^{\b}}.
  \]
\end{lemmaalph}

The last lemma pertains to controlling higher order derivatives of $\frac{r^{\b-2}}{\sinh{r}}$ for large $r$ (see also \cite[Lemma 3.1]{LiLuy2} and \cite[Corollary 5.14]{MR1469569}).

\begin{lemma}
  \label{lem:higher-order-derivatives-on-trig}
  Let $p,q\in\N_{\geq0}$ and $0<\g<3$.
  If $0<r$, then
  \[
    \left( -\frac{1}{\sinh2r}\frac{\p}{\p{r}} \right)^{q}\left( -\frac{1}{\sinh{r}} \frac{\p}{\p{r}} \right)^{q}\frac{r^{\b-2}}{\sinh{r}}\lesssim{r^{\b-2}}e^{-\left( p+2q+1 \right)r}.
  \]
\end{lemma}

\begin{proof}
  Using
  \[
    \frac{1}{\sinh r} = \frac{2e^{-r}}{1 - e^{-2r}} = 2 \sum_{j=0}^{\oo} e^{-(2j+1)r},
  \]
  it is easy to see that
  \[
    \left( -\frac{1}{\sinh r} \frac{\p}{\p r} \right)^{p} \frac{r^{\b-2}}{\sinh r} \sim r^{\b-2} \left[ e^{-(p+1)r} + e^{-(p+3)r} + \cdots \right] ,
  \]
  and, similarly, that
  \[
    \left( -\frac{1}{\sinh 2r} \frac{\p}{\p r} \right)^{q} \left( -\frac{1}{\sinh r} \frac{\p}{\p r} \right)^{p} \frac{r^{\b-2}}{\sinh r} \lesssim r^{\b-2} e^{-(2q+p+1)r},
  \]
  as desired.
\end{proof}

  In the following subsections, we will prove various kernel estimates for $k_{\g}$, $k_{\zeta,\g'}$, $k_{\g}*k_{\zeta,\g'}$ and $k_{\g}*k_{\zeta,\g'}*f$ for smooth compactly supported function on $\B_{\Q}^{m}$ and $\B_{\C a}$.
  Along with the Fourier analysis on symmetric spaces (i.e., the Plancherel theorem and Kunze-Stein phenomenon) and factorization, these estimates form the ingredients of the proofs of the Poincar\'e-Sobolev and Hardy-Sobolev-Maz'ya inequalities on $H_{\Q}^{m}$ and $H_{\C a}$.

  \subsection{Convolution Estimates}

  In order to prove the kernel estimates, we will need asymptotics of certain convolutions.
  This is contained in Lemmas \ref{lem:convolution-estimates}, \ref{lem:convolution-estimates-octonionic}, \ref{lem:convolution-estimates-large-distance} and \ref{lem:convolution-estimates-large-distance-octonionic} below.
  Due to the appearance of $\Psi_{\C a}$ in the automorphisms on $\B_{\C a}$, separate considerations are needed for $\B_{\C a}$ and so we state the convolution estimates for $\B_{\Q}^{m}$ and $\B_{\C a}$ separately.
  We mention that, when compared to the complex hyperbolic setting, the hypothesis $\la_{1}+\la_{2}>\g+\g'-4m+2$ differs from the reasonably expected $\la_{1}+\la_{2}>\g+\g'-4m$, and this has to do with the higher dimensional center of $\H_{\Q}^{m-1}$.
  This is similar for the corresponding hypothesis in Lemma \ref{lem:convolution-estimates-octonionic} for $H_{\C a}$.

  We will need the following convolution integral on Euclidean space (see \cite{s}).
  \begin{lemmaalph}
    \label{lemalph:euclidean-convolution-integral}
    For $0<\g,\g'<k$ and $0<\g+\g'<k$, there holds
    \begin{equation}
      \begin{aligned}
	\int\limits_{\R^{k}}|x|^{\g-k}|y-x|^{\g'-k}dx=\frac{\g_{k}(\g)\g_{k}(\g')}{\g_{k}(\g+\g')}|y|^{\g+\g'-k}.
      \end{aligned}
    \end{equation}
    where
    \[
      \g_{k}(\g)=\frac{\pi^{\frac{k}{2}}2^{\g}\Gamma\left( \frac{\g}{2} \right)}{\Gamma\left( \frac{k-\g}{2} \right)}.
    \]
  \end{lemmaalph}

  We may now state the main convolution estimate lemma for small distances.
  \begin{lemma}
    Let $0<\g<4m$, $0<\g'<4m$, and $\la_{1}+\la_{2}>\g+\g'-4m+2$.
    If $0<\g+\g'<4m-1$ and $0<\rho<1$, then on $\B_{\Q}^{m}$ there holds
    \begin{equation}
      \begin{aligned}
	\frac{1}{\left( \sinh\rho \right)^{4m-\g}\left( \cosh\rho \right)^{\la_{1}}}*\frac{1}{\left( \sinh\rho \right)^{4m-\g'}\left( \cosh\rho \right)^{\la_{2}}}\leq\frac{\g_{4m}(\g)\g_{4m}(\g')}{\g_{4m}(\g+\g')}\frac{1}{\rho^{4m-\g-\g'}}+O\left( \frac{1}{\rho^{4m-\g-\g'-1}} \right)
      \end{aligned}.
    \end{equation}
    If $4m-1\leq\g+\g'<4m$, $0<\e<4m-\g-\g'$ and $0<\rho<1$, then on $\B_{\Q}^{m}$ there holds
    \begin{equation}
      \begin{aligned}
	\frac{1}{\left( \sinh\rho \right)^{4m-\g}\left( \cosh\rho \right)^{\la_{1}}}*\frac{1}{\left( \sinh\rho \right)^{4m-\g'}\left( \cosh\rho \right)^{\la_{2}}}\leq\frac{\g_{4m}(\g)\g_{4m}(\g')}{\g_{4m}(\g+\g')}\frac{1}{\rho^{4m-\g-\g'}}+O\left( \frac{1}{\rho^{4m-\g-\g'- \e}} \right)
      \end{aligned}.
    \end{equation}
    \label{lem:convolution-estimates}
  \end{lemma}

  \begin{proof}

    By Lemma \ref{lemalph:automorphism-properties} item (iv), and by $dV=\frac{dz}{(1-|z|^{2})^{2m+2}}$, we compute
    \begin{align*}
      &\frac{1}{\left( \sinh\rho \right)^{4m-\g}\left( \cosh\rho \right)^{\la_{1}}}*\frac{1}{\left( \sinh\rho \right)^{4m-\g'}\left( \cosh\rho \right)^{\la_{2}}}\\
      &=\int\limits_{\B_{\Q}^{m}}\left( \frac{\sqrt{1-|z|^{2}}}{|z|} \right)^{4m-\g} \left( 1-|z|^{2} \right)^{\frac{\la_{1}}{2}} \left( \frac{(1-|w|^{2})(1-|z|^{2})}{|z-w|^{2}+|\gen{z,w}_{\Q}|^{2}-|z|^{2}|w|^{2}} \right)^{\frac{4m-\g'}{2}}\\
      &\times\left( \frac{(1-|w|^{2})(1-|z|^{2})}{|1-\gen{z,w}_{\Q}|^{2}} \right)^{\frac{\la_{2}}{2}}\frac{dz}{(1-|z|^{2})^{2m+2}}\\
      &=\left( 1-|w|^{2} \right)^{\frac{4m-\g'+\la_{2}}{2}}\int\limits_{\B_{\Q}^{m}}\frac{1}{|z|^{4m-\g}} \left( \frac{1}{|z-w|^{2}+|\gen{z,w}_{\Q}|^{2}-|z|^{2}|w|^{2}} \right)^{\frac{4m-\g'}{2}}\\
      &\times\frac{1}{|1-\gen{z,w}_{\Q}|^{\la_{2}}}\frac{1}{(1-|z|^{2})^{\frac{4+\la+\la'-4m-\la_{1}-\la_{2}}{2}}}dz\\
      &=\left( \cosh\rho(w) \right)^{-\left( 4m-\g'+\la_{2} \right)}\left( A_{5}+A_{6} \right),
    \end{align*}
    where
    \[
      A_{5}=\int\limits_{ \left\{ |z|<\frac{1}{2} \right\}}\cdots\text{ and }A_{6}=\int\limits_{ \left\{ \frac{1}{2}\leq|z|<1 \right\}}\cdots.
    \]
    Note that, when $\rho(w)<1$ and $|z|\leq\frac{1}{2}$, there holds
    \begin{align*}
      |1-\gen{z,w}_{\Q}|^{\la_{2}}\left( 1-|z|^{2} \right)^{\frac{4+\g+\g'-4m-\la_{1}-\la_{2}}{2}}=1+O\left( |z| \right).
    \end{align*}
    On the other hand, there holds
    \begin{align*}
      |\gen{z,w}_{\Q}|^{2}-|z|^{2}|w|^{2}&=||z|^{2}+\gen{z,w-z}_{\Q}|^{2}-|z|^{2}|w-z+z|^{2}\\
      &=|\gen{z,w-z}_{\Q}|^{2}-|z|^{2}|w-z|^{2}\\
      &=|z|^{2}|w-z|^{2}\left[ \left|\gen{\frac{z}{|z|},\frac{w-z}{|w-z|}}_{\Q}\right|-1 \right],
    \end{align*}
    and so
    \begin{align*}
      |z-w|^{2}+|\gen{z,w}_{\Q}|^{2}-|z|^{2}|w|^{2}&=|z-w|^{2}\left[ 1+|z|^{2}\left[ \left| \gen{\frac{z}{|z|},\frac{w-z}{|w-z|}}_{\Q} \right|^{2}-1 \right] \right]\\
      &=|z-w|^{2}\left( 1+O\left( |z|^{2} \right) \right).
    \end{align*}

    Since $0<\g+\g'<4m-1$, we may use Lemma \ref{lemalph:euclidean-convolution-integral} to compute
    \begin{align*}
      A_{5}&=\int\limits_{ \left\{ |z|\leq\frac{1}{2} \right\}}\frac{1}{|z|^{4m-\g}}\frac{1}{|z-w|^{4m-\g'}}\left( 1+O\left( |z| \right) \right)dz\\
      &\leq\int_{\R^{4m}}\frac{1}{|z|^{4m-\g}}\frac{1}{|z-w|^{4m-\g'}}dz+O\left( \int_{\R^{M}}\frac{1}{|z|^{4m-\g-1}}\frac{1}{|z-w|^{4m-\g'}}dz \right)\\
      &=\frac{\g_{4m}(\g)\g_{4m}(\g')}{\g_{4m}(\g+\g')}\frac{1}{|w|^{4m-\g-\g'}}+O\left( \frac{1}{|w|^{4m-\g-\g'-1}} \right).
    \end{align*}
    Similarly, if $0<\e<4m-\g-\g'$, we obtain
    \[
      A_{5}=\frac{\g_{4m}(\g)\g_{4m}(\g')}{\g_{4m}(\g+\g')}\frac{1}{|w|^{4m-\g-\g'}}+O\left( \frac{1}{|w|^{4m-\g-\g'-\e}} \right).
    \]

    We are left with estimating $A_{6}$: since
    \[
      \frac{4+\g+\g'-4m-\la_{1}-\la_{2}}{2}<1
    \]
    is equivalent to
    \[
      \g+\g'-4m+2<\la_{1}+\la_{2}
    \]
    we find
    \begin{align*}
      A_{6}&=\int\limits_{ \left\{ \frac{1}{2}\leq|z|\leq1 \right\} } \frac{1}{|z|^{4m-\g}} \left( \frac{1}{|z-w|^{2}+|\gen{z,w}_{\Q}|^{2}-|z|^{2}|w|^{2}} \right)^{\frac{4m-\g'}{2}}\\
      &\times\frac{1}{|1-\gen{z,w}_{\Q}|^{\la_{2}}} \frac{1}{\left( 1-|z|^{2} \right)^{\frac{4+\g+\g'-4m-\la_{1}-\la_{2}}{2}}}dz\\
      &\sim \int\limits_{ \left\{ \frac{1}{2}\leq|z|\leq1 \right\} } \frac{1}{\left( 1-|z|^{2} \right)^{\frac{4+\g+\g'-4m-\la_{1}-\la_{2}}{2}}}dz\\
      &\sim \int\limits_{\frac{1}{2}}^{1} \frac{r}{\left( 1-r^{2} \right)^{\frac{4+\g+\g'-4m-\la_{1}-\la_{2}}{2}}} dr\\
      &<\oo.
    \end{align*}

    In conclusion, since $\cosh{r}\sim1$ in as $r\to0$, we find
    \begin{align*}
      \frac{1}{\left( \sinh\rho \right)^{4m-\g}\left( \cosh\rho \right)^{\la_{1}}}&*\frac{1}{\left( \sinh\rho \right)^{4m-\g'}\left( \cosh\rho \right)^{\la_{2}}}\\
      &\leq\frac{\g_{4m}(\g)\g_{4m}(\g')}{\g_{4m}(\g+\g')}\frac{1}{|w|^{4m-\g-\g'}}+O\left( \frac{1}{|w|^{4m-\g-\g'-1}} \right),
    \end{align*}
    and the result follows since
    \[
      \rho(w)=\frac{1}{2}\log\frac{1+|w|}{1-|w|}=|w|+O\left( |w|^{3} \right)
    \]
    as $|w|\to0$.
  \end{proof}

  \begin{lemma}
    Let $0<\g<16$, $0<\g'<16$, and $\la_{1}+\la_{2}>\g+\g'-10$.
    If $0<\g+\g'<15$ and $0<\rho<1$, then on $\B_{\C a}$ there holds
    \begin{equation}
      \begin{aligned}
	\frac{1}{\left( \sinh\rho \right)^{16-\g}\left( \cosh\rho \right)^{\la_{1}}}*\frac{1}{\left( \sinh\rho \right)^{16-\g'}\left( \cosh\rho \right)^{\la_{2}}}\leq\frac{\g_{16}(\g)\g_{16}(\g')}{\g_{16}(\g+\g')}\frac{1}{\rho^{16-\g-\g'}}+O\left( \frac{1}{\rho^{15-\g-\g'}} \right)
      \end{aligned}.
    \end{equation}
    If $15\leq\g+\g'<16$, $0<\e<16-\g-\g'$ and $0<\rho<1$, then on $\B_{\C a}$ there holds
    \begin{equation}
      \begin{aligned}
	\frac{1}{\left( \sinh\rho \right)^{16-\g}\left( \cosh\rho \right)^{\la_{1}}}*\frac{1}{\left( \sinh\rho \right)^{16-\g'}\left( \cosh\rho \right)^{\la_{2}}}\leq\frac{\g_{16}(\g)\g_{16}(\g')}{\g_{16}(\g+\g')}\frac{1}{\rho^{16-\g-\g'}}+O\left( \frac{1}{\rho^{16-\g-\g'- \e}} \right)
      \end{aligned}.
    \end{equation}
    \label{lem:convolution-estimates-octonionic}
  \end{lemma}

  \begin{proof}
    By Lemma \ref{lemalph:automorphism-properties-octonionic} item (iv), and by $dV=\frac{dz}{(1-|z|^{2})^{12}}$, we compute
    \begin{align*}
      &\frac{1}{\left( \sinh\rho \right)^{16-\g}\left( \cosh\rho \right)^{\la_{1}}}*\frac{1}{\left( \sinh\rho \right)^{16-\g'}\left( \cosh\rho \right)^{\la_{2}}}\\
      &=\int\limits_{\B_{\Q}^{m}}\left( \frac{\sqrt{1-|z|^{2}}}{|z|} \right)^{16-\g} \left( 1-|z|^{2} \right)^{\frac{\la_{1}}{2}} \left( \frac{(1-|w|^{2})(1-|z|^{2})}{\Psi_{\C a}(z,w) - (1-|z|^{2})(1-|w|^{2})} \right)^{\frac{16-\g'}{2}}\\
      &\times\left( \frac{(1-|w|^{2})(1-|z|^{2})}{\Psi_{\C a}(z,w)} \right)^{\frac{\la_{2}}{2}}\frac{dz}{(1-|z|^{2})^{12}}\\
      &=\left( 1-|w|^{2} \right)^{\frac{16-\g'+\la_{2}}{2}}\int\limits_{\B_{\Q}^{m}}\frac{1}{|z|^{16-\g}} \left( \frac{1}{\Psi_{\C a}(z,w) - (1-|z|^{2})(1-|w|^{2})} \right)^{\frac{16-\g'}{2}}\\
      &\times\frac{1}{\Psi_{\C a}(z,w)^{\frac{\la_{2}}{2}}}\frac{1}{(1-|z|^{2})^{\frac{ \g + \g' -\la_{1} - \la_{2} - 8}{2}}}dz\\
      &=\left( \cosh\rho(w) \right)^{-\left( 16-\g'+\la_{2} \right)}\left( A_{5}'+A_{6}' \right),
    \end{align*}
    where
    \[
      A_{5}'=\int\limits_{ \left\{ |z|<\frac{1}{2} \right\}}\cdots\text{ and }A_{6}'=\int\limits_{ \left\{ \frac{1}{2}\leq|z|<1 \right\}}\cdots.
    \]
    Note that, when $\rho(w)<1$ and $|z|\leq\frac{1}{2}$, there holds
    \begin{align*}
      \Psi_{\C a}(z,w)^{\frac{\la_{2}}{2}}\left( 1-|z|^{2} \right)^{\frac{\g+\g'-\la_{1}-\la_{2}-8}{2}}=1+O\left( |z| \right).
    \end{align*}
    Next, we have
    \begin{align*}
      \Psi_{\C a}(z,a) - (1-|z|^{2})(1-|w|^{2}) &= \Phi_{\C a}(z,a) - 2 \langle{z,a}\rangle_{\R} + |z|^{2} + |a|^{2} - |z|^{2}|a|^{2}\\
      &= \Phi_{\C a}(z,a) + |z-a|^{2} - |z|^{2} |a|^{2} \\
      &= |z-a|^{2} \left( 1 + \frac{\Phi_{\C a}(z,a) - |z|^{2}|a|^{2}}{|z-a|^{2}} \right).
    \end{align*}
    Moreover, it is not hard to see that
    \[
      \frac{\Phi_{\C a}(z,w) - |z|^{2}|w|^{2}}{|z-w|^{2}} = O(|z|^{2}).
    \]
    Indeed, using invariance of distance $\rho$, we can assume $w = (w_{1},w_{2})$ with $\Re w_{1} = c \in \R$ and all other components are zero.
    Then $\Phi_{\C a}(z,a) - |z|^{2}|a|^{2} = -c^{2}|z_{2}|^{2}$ and clearly $|z_{2}|^{2} / |z-a|^{2}$ is bounded as $z\to a$.
    Therefore, using also that $\Phi_{\C a}(z,w) \leq |z|^{2} |w|^{2}$, we  obtain
    \begin{align*}
      \Psi_{\C a}(z,a) - (1-|z|^{2})(1-|w|^{2}) &= |z-a|^{2} \left( 1 + O(|z|^{2}) \right).
    \end{align*}
    The remainder of the proof is analogous to the proof of Lemma \ref{lem:convolution-estimates} and is thus omitted.

  \end{proof}

  Next, we will state and prove the main convolution lemma for large distances.
  In preparation, we recall some properties and definitions of certain special functions.
  First, recall the generalized hypergeometric function
  \[
    {}_{p}F_{q}\left( a_{1},\ldots,a_{p};b_{1},\ldots,b_{q};z \right) = \sum_{k=0}^{\oo}\frac{(a_{1})_{k} \cdots (a_{p})_{k}}{(b_{1})_{k} \cdots (n_{q})_{k}}\frac{z^{k}}{k!}.
  \]
  Second, we also recall the following hypergeometric integral (see \cite[Equation 7.512.5]{gr}): supposing the complex parameters $\a,\b,\g,\rho,\sigma$ satisfy
  \[
    \Re\rho>0, \quad \Re\sigma>0, \quad \Re(\g+\sigma-\a-\b)>0,
  \]
  there holds
  \begin{equation}
    \int_{0}^{1} x^{\rho-1} (1-x)^{\sigma-1} {}_{2}F_{1}(\a,\b;\g;x)dx = \frac{\Gamma(\rho)\Gamma(\sigma)}{\Gamma(\rho+\sigma)} {}_{3}F_{2}(\a,\b,\rho;\g,\rho+\sigma;1).
    \label{eq:hypergeometric-integral}
  \end{equation}

  \begin{lemma}
    Let $0<\g<4m$, $0<\g'<4m$, and $\la_{1}+\la_{2}>\g+\g'-4m+2$.
    If $\la_{2}-\g'<\la_{1}-\g$ and $1\leq\rho$, then on $\B_{\Q}^{m}$ there holds
    \[
      \frac{1}{\left( \sinh\rho \right)^{4m-\g} \left( \cosh\rho \right)^{\la_{1}}} * \frac{1}{\left( \sinh\rho \right)^{4m-\la'} \left( \cosh\rho \right)^{\la_{2}}} \sim e^{-(4m-\g'+\la_{2})\rho}.
    \]
    \label{lem:convolution-estimates-large-distance}
  \end{lemma}

  \begin{proof}
    By the proof of Lemma \ref{lem:convolution-estimates}, we have
    \begin{align*}
      &\frac{1}{\left( \sinh\rho \right)^{4m-\g}\left( \cosh\rho \right)^{\la_{1}}}*\frac{1}{\left( \sinh\rho \right)^{4m-\g'}\left( \cosh\rho \right)^{\la_{2}}}\\
      &=\left( \cosh\rho(w) \right)^{-\left( 4m-\g'+\la_{2} \right)}\int\limits_{\B_{\Q}^{m}}\frac{1}{|z|^{4m-\g}} \left( \frac{1}{|z-w|^{2}+|\gen{z,w}_{\Q}|^{2}-|z|^{2}|w|^{2}} \right)^{\frac{4m-\g'}{2}}\\
      &\times\frac{1}{|1-\gen{z,w}_{\Q}|^{\la_{2}}}\frac{1}{(1-|z|^{2})^{\frac{4+\la+\la'-4m-\la_{1}-\la_{2}}{2}}}dz.
    \end{align*}

    Setting
    \[
      F(w)=\int_{S^{4m-1}}\left( \frac{1}{|z-w|^{2}+|\gen{z,w}_{\Q}|^{2}-|z|^{2}|w|^{2}} \right)^{\frac{4m-\g'}{2}}\frac{1}{|1-\gen{z,w}_{\Q}|^{\la_{2}}}d\sigma,
    \]
    we see that $F(w)=F(|w|)$.
    Moreover, by Proposition \ref{prop:spherical-integral-eigenvalue}, we find
    \begin{align*}
      \lim\limits_{|w|\to1^{-}}F(w)&=\lim\limits_{|w|\to1^{-}}\int\limits_{S^{4m-1}} |1-\gen{z,w}_{\Q}|^{-(4m-\g'+\la_{2})}d\sigma\\
      &=\frac{2\pi^{2m}}{\Gamma(2m)}{}_{2}F_{1}\left( \frac{4m-\g'+\la_{2}}{2},\frac{4m-\g'+\la_{2}-2}{2};2m;|z|^{2} \right).
    \end{align*}
    Consequently, there holds
    \begin{align*}
      \lim\limits_{|w|\to1^{-}}\int\limits_{\B_{\Q}^{m}}\frac{1}{|z|^{4m-\g}} & \left( \frac{1}{|z-w|^{2}+|\gen{z,w}_{\Q}|^{2}-|z|^{2}|w|^{2}} \right)^{\frac{4m-\g'}{2}}\\
      &\times\frac{1}{|1-\gen{z,w}_{\Q}|^{\la_{2}}}\frac{1}{(1-|z|^{2})^{\frac{4+\la+\la'-4m-\la_{1}-\la_{2}}{2}}}dz\\
      &=\frac{2\pi^{2m}}{\Gamma(2m)}\int_{0}^{1}r^{\g-1}(1-r^{2})^{-\left( \frac{4+\g+\g'-4m-\la_{1}-\la_{2}}{2} \right)}\\
      &\times{}_{2}F_{1}\left( \frac{4m-\g'+\la_{2}}{2},\frac{4m-\g'+\la_{2}-2}{2};2m;r^{2} \right)dr\\
      &=\frac{2\pi^{2m}}{\Gamma(2m)}\int_{0}^{1}t^{\frac{\g}{2}-1}(1-t)^{-\left( \frac{4+\g+\g'-4m-\la_{1}-\la_{2}}{2} \right)}\\
      &\times{}_{2}F_{1}\left( \frac{4m-\g'+\la_{2}}{2},\frac{4m-\g'+\la_{2}-2}{2};2m;t \right)dt,
    \end{align*}
    where the change of variable $r^{2}=t$ was used in the last equality.
    Now, using \eqref{eq:hypergeometric-integral}, we have
    \begin{align*}
      \lim\limits_{|w|\to1^{-}}\int\limits_{\B_{\Q}^{m}}\frac{1}{|z|^{4m-\g}} & \left( \frac{1}{|z-w|^{2}+|\gen{z,w}_{\Q}|^{2}-|z|^{2}|w|^{2}} \right)^{\frac{4m-\g'}{2}}\\
      &\times\frac{1}{|1-\gen{z,w}_{\Q}|^{\la_{2}}}\frac{1}{(1-|z|^{2})^{\frac{4+\la+\la'-4m-\la_{1}-\la_{2}}{2}}}dz\\
      &=\frac{\pi^{2m}}{\Gamma(2m)}\frac{\Gamma\left( \frac{\g}{2} \right) \Gamma\left( \frac{4m+\la_{1}+\la_{2}-\g-\g'-2}{2} \right)}{\Gamma\left( \frac{4m+\la_{1}+\la_{2}-\g'-2}{2} \right)}\\
      &\times{}_{3}F_{2} \left( \frac{4m-\g'+\la_{2}}{2} , \frac{4m-\g'+\la_{2}-2}{2}, \frac{\g}{2} ; 2m , \frac{4m+\la_{1}+\la_{2}-\g'-2}{2} ; 1 \right).
    \end{align*}
    At last, using $\cosh{r} \sim e^{r}$ for $1\leq{r}$, we have proved
    \begin{align*}
      \frac{1}{\left( \sinh\rho \right)^{4m-\g}\left( \cosh\rho \right)^{\la_{1}}}*\frac{1}{\left( \sinh\rho \right)^{4m-\g'}\left( \cosh\rho \right)^{\la_{2}}}&\sim \left( \cosh\rho \right)^{-(4m-\g'+\la_{2})}\\
      &\sim e^{-(4m-\g'+\la_{2})\rho}.
    \end{align*}
  \end{proof}

  \begin{lemma}
    Let $0<\g<16$, $0<\g'<16$, and $\la_{1}+\la_{2}>\g+\g'-10$.
    If $\la_{2}-\g'<\la_{1}-\g$ and $1\leq\rho$, then
    \[
      \frac{1}{\left( \sinh\rho \right)^{16-\g} \left( \cosh\rho \right)^{\la_{1}}} * \frac{1}{\left( \sinh\rho \right)^{16-\la'} \left( \cosh\rho \right)^{\la_{2}}} \sim e^{-(16-\g'+\la_{2})\rho}.
    \]
    \label{lem:convolution-estimates-large-distance-octonionic}
  \end{lemma}

  \begin{proof}
    By the proof of Lemma \ref{lem:convolution-estimates-octonionic}, we have
    \begin{align*}
      &\frac{1}{\left( \sinh\rho \right)^{16-\g}\left( \cosh\rho \right)^{\la_{1}}}*\frac{1}{\left( \sinh\rho \right)^{16-\g'}\left( \cosh\rho \right)^{\la_{2}}}\\
      &=\left( \cosh\rho(w) \right)^{-\left( 16-\g'+\la_{2} \right)}\int\limits_{\B_{\C a}}\frac{1}{|z|^{16-\g}} \left( \frac{1}{\Psi_{\C a}(z,w) - (1-|z|^{2})(1-|w|^{2})} \right)^{\frac{16-\g'}{2}}\\
      &\times\frac{1}{\Psi_{\C a}(z,w)^{\frac{\la_{2}}{2}}}\frac{1}{(1-|z|^{2})^{\frac{ \g + \g' -\la_{1} - \la_{2} - 8}{2}}}dz.
    \end{align*}

    Setting
    \[
      F(w)=\int_{S^{4m-1}}\left( \frac{1}{\Psi_{\C a}(z,w) - (1-|z|^{2})(1-|w|^{2})} \right)^{\frac{16-\g'}{2}} \frac{1}{\Psi_{\C a}(z,w)^{\frac{\la_{2}}{2}}}d\sigma,
    \]
    we see that $F(w)=F(|w|)$.
    Moreover, by Proposition \ref{prop:spherical-integral-eigenvalue-octonionic}, we find
    \begin{align*}
      \lim\limits_{|w|\to1^{-}}F(w)&=\lim\limits_{|w|\to1^{-}}\int\limits_{S^{4m-1}} \Psi_{\C a}(z,w)^{-\frac{16-\g'+\la_{2}}{2}}d\sigma\\
      &=\frac{2\pi^{8}}{7!}{}_{2}F_{1}\left( \frac{16-\g'+\la_{2}}{2}, \frac{10-\g'+\la_{2}}{2}; 8;r^{2} \right).
    \end{align*}
    Consequently, there holds
    \begin{align*}
      \lim\limits_{|w|\to1^{-}}\int\limits_{\B_{\C a}}\frac{1}{|z|^{16-\g}}& \left( \frac{1}{\Psi_{\C a}(z,w) - (1-|z|^{2})(1-|w|^{2})} \right)^{\frac{16-\g'}{2}}\\
      &\times\frac{1}{\Psi_{\C a}(z,w)^{\frac{\la_{2}}{2}}}\frac{1}{(1-|z|^{2})^{\frac{ \g + \g' -\la_{1} - \la_{2} - 8}{2}}}dz\\
      &=\frac{2\pi^{8}}{7!}\int_{0}^{1}r^{\g-1}(1-r^{2})^{-\left( \frac{\g+\g'-\la_{1}-\la_{2}-8}{2} \right)}\\
      &\times{}_{2}F_{1}\left( \frac{16-\g'+\la_{2}}{2}, \frac{16-\g'+\la_{2}-6}{2}; 8;r^{2} \right) dr\\
      &=\frac{2\pi^{8}}{7!}\int_{0}^{1}t^{\frac{\g}{2}-1}(1-t)^{-\left( \frac{\g+\g'-\la_{1}-\la_{2}-8}{2} \right)}\\
      &\times{}_{2}F_{1}\left( \frac{16-\g'+\la_{2}}{2}, \frac{10-\g'+\la_{2}}{2}; 8;r^{2} \right) dt\\
    \end{align*}
    where the change of variable $r^{2}=t$ was used in the last equality.
    Now, using \eqref{eq:hypergeometric-integral}, we have
    \begin{align*}
      \lim\limits_{|w|\to1^{-}}\int\limits_{\B_{\C a}}\frac{1}{|z|^{16-\g}}& \left( \frac{1}{\Psi_{\C a}(z,w) - (1-|z|^{2})(1-|w|^{2})} \right)^{\frac{16-\g'}{2}}\\
      &\times\frac{1}{\Psi_{\C a}(z,w)^{\frac{\la_{2}}{2}}}\frac{1}{(1-|z|^{2})^{\frac{ \g + \g' -\la_{1} - \la_{2} - 8}{2}}}dz\\
      &=\frac{\pi^{8}}{7!}\frac{\Gamma\left( \frac{\g}{2} \right) \Gamma\left( \frac{10+\la_{1}+\la_{2} - \g-\g'}{2} \right)}{\Gamma\left( \frac{10+\la_{1}+\la_{2}-\g'}{2} \right)}\\
      &\times{}_{3}F_{2} \left( \frac{16-\g'+\la_{2}}{2}, \frac{10-\g'+\la_{2}}{2}; \frac{\g}{2}; 7, \frac{10+\la_{1}+\la_{2}-\g'}{2} ; 1 \right).
    \end{align*}
    At last, using $\cosh{r} \sim e^{r}$ for $1\leq{r}$, we have proved
    \begin{align*}
      \frac{1}{\left( \sinh\rho \right)^{16-\g}\left( \cosh\rho \right)^{\la_{1}}}*\frac{1}{\left( \sinh\rho \right)^{16-\g'}\left( \cosh\rho \right)^{\la_{2}}}&\sim \left( \cosh\rho \right)^{-(16-\g'+\la_{2})}\\
      &\sim e^{-(16-\g'+\la_{2})\rho}.
    \end{align*}

  \end{proof}

  \subsection{Estimates for $k_{\g}$}

  In this subsection, we obtain the asymptotics for $k_{\g}$.
  Note that the large distance asymptotics ($1\leq\rho$) are already contained in \eqref{eq:anker-ji-large-distance-asymptotics}.

  \begin{lemma}
    Let $0<\g<3$ and let $N = \dim_{\R} H_{\mbF}^{m}$.
    If $0<\rho<1$, then
    \[
      k_{\g}\leq\frac{1}{\g_{N}(\g)}\frac{1}{\rho^{N-\g}}+O\left( \frac{1}{\rho^{N-\g-1}} \right).
    \]
    If $1\leq\rho$, then
    \[
      k_{\g}\sim\rho^{\g-2}e^{-\frac{Q}{2}\rho}.
    \]
    \label{lem:k-gamma-small-distance}
  \end{lemma}

  \begin{proof}
    As mentioned above, we only need to prove the estimate for $0<\rho<1$.

    By using \eqref{eq:quaternionic-heat-kernel-in-terms-of-hyperbolic-heat-kernel}, we will write $k_{\g}$ in terms of a Bessel-Green-Riesz kernel on $H_{\R}^{n}$, where $n=2\tilde{m}+1$ and $\tilde m = 2m-2$ if $\mbF = \Q$ and $\tilde m = 4$ if $\mbF = \C a$.
    Also recall that $h_{t}(\rho,n)$ denotes the heat kernel on $H_{\R}^{n}$ (see \eqref{eq:heat-kernel-on-real-hyperbolic-space}).
    Lastly let $c$ denote $c_{m}$ (resp. $c_{o}$) from \eqref{eq:quaternionic-heat-kernel-in-terms-of-hyperbolic-heat-kernel} (resp. \eqref{eq:octonionic-heat-kernel-in-terms-of-hyperbolic-heat-kernel}) and let $\mu = 2$ (resp. $\mu=4$) when $\mbF = \Q$ (resp. $\mbF = \C a$).

    Then, by the Mellin transform and \eqref{eq:quaternionic-heat-kernel-in-terms-of-hyperbolic-heat-kernel} and \eqref{eq:octonionic-heat-kernel-in-terms-of-hyperbolic-heat-kernel}, we have
    \begin{align*}
      k_{\g}(\rho)&=\frac{1}{\Gamma\left( \frac{\g}{2} \right)} \int_{0}^{\oo} t^{\frac{\g}{2}-1} e^{t\left( \Delta+\frac{Q^{2}}{4} \right)} dt\\
      &= \frac{c}{b_{\tilde{m}}} \int_{\rho}^{\oo} \frac{\sinh2r}{\sqrt{\cosh2r-\cosh2\rho}} \left( -\frac{1}{\sinh2r}\frac{\p}{\p{r}} \right)^{\mu}\\
      &\times \frac{1}{\Gamma\left( \frac{\g}{2} \right)} \int_{0}^{\oo} t^{\frac{\g}{2}-1} e^{\frac{Q^{2}}{4}t} e^{-\tilde{m}^{2}t} h_{t}(r,n) dtdr\\
      &=\frac{c}{b_{\tilde m}} \int_{\rho}^{\oo} \frac{\sinh2r}{\sqrt{\cosh2r-\cosh2\rho}} \left( -\frac{1}{\sinh2r}\frac{\p}{\p{r}} \right)^{\mu} \left( -\Delta_{H_{\R}^{n}} - \left( \frac{n-1}{2} \right)^{2} \right)^{-\frac{\g}{2}} dr\\
      &=A_{1}+A_{2},
    \end{align*}
    where
    \[
      A_{1}=\int\limits_{\rho}^{1}\cdots\text{ and }A_{2}=\int\limits_{1}^{\oo}\cdots.
    \]

    We begin by estimating $A_{1}$.
    Using Lemma \ref{lemalph:liluyang-small-bessel-green-riesz-kernel-estimate}, it is easy to see that, for $0<r<1$, there holds
    \begin{align*}
      \left( -\frac{1}{\sinh2r}\frac{\p}{\p{r}} \right)^{2}\left( -\Delta_{H_{\R}^{n}}-\left( \frac{n-1}{2} \right)^{2} \right)^{-\frac{\g}{2}}&=\left( -\frac{1}{\sinh2r}\frac{\p}{\p{r}} \right)^{2}\left( \frac{1}{\g_{n}(\g)}\frac{1}{r^{n-\g}} + O\left( \frac{1}{r^{n-\g-1}} \right) \right)\\
      &=\frac{1}{\g_{n}(\g)}\frac{(n-\g)(n+2-\g)}{4}\frac{1}{r^{n+4-\g}}+O\left( \frac{1}{r^{n+3-\g}} \right),
    \end{align*}
    and similarly
    \begin{align*}
      &\left( -\frac{1}{\sinh2r}\frac{\p}{\p{r}} \right)^{4}\left( -\Delta_{H_{\R}^{n}}-\left( \frac{n-1}{2} \right)^{4} \right)^{-\frac{\g}{2}}\\
      &=\frac{1}{\g_{n}(\g)}\frac{(n-\g)(n+2-\g)(n+4-\g)(n+6-\g)}{16}\frac{1}{r^{n+8-\g}}+O\left( \frac{1}{r^{n+7-\g}} \right).
    \end{align*}
    Consequently, in the quaternionic case, there holds
    \begin{align*}
      \sinh2r\left( -\frac{1}{\sinh2r} \frac{\p}{\p{r}} \right)^{2}\left( -\Delta_{H_{\R}^{n}} - \left( \frac{n-1}{2} \right)^{2} \right)&=\frac{1}{\g_{n}(\g)}\frac{(n-\g)(n+2-\g)}{2}\frac{1}{r^{n+3-\g}}+O\left( \frac{1}{r^{n+2-\g}} \right)\end{align*}
    and, in the octonionic case, there holds
    \begin{align*}
      &\sinh2r\left( -\frac{1}{\sinh2r} \frac{\p}{\p{r}} \right)^{4}\left( -\Delta_{H_{\R}^{n}} - \left( \frac{n-1}{2} \right)^{2} \right)\\
      &=\frac{1}{\g_{n}(\g)}\frac{(n-\g)(n+2-\g)(n+4-\g)(n+6-\g)}{8}\frac{1}{r^{n+7-\g}}+O\left( \frac{1}{r^{n+6-\g}} \right).
    \end{align*}

    Now, using Lemma \ref{lemalph:luyang-hypertrig-integral}, we compute in the quaternionic case that
    \begin{align*}
      A_{1}&=\frac{c_{m}}{b_{2m-2}}\frac{(n-\g)(n+2-\g)}{2\g_{n}(\g)}\int\limits_{\rho}^{1}\frac{1}{\sqrt{\cosh2r-\cosh2\rho}} \left[ \frac{1}{r^{n+3-\g}}+O\left( \frac{1}{r^{n+2-\g}} \right) \right]dr\\
      &\leq\frac{c_{m}}{b_{2m-2}}\frac{(n-\g)(n+2-\g)}{2\g_{n}(\g)}\int\limits_{\rho}^{1}\frac{\cosh{r}}{\sqrt{\cosh2r-\cosh2\rho}} \left[ \frac{1}{\left( \sinh{r} \right)^{n+3-\g}}+O\left( \frac{1}{\left( \sinh{r} \right)^{n+2-\g}} \right) \right]dr\\
      &=\frac{c_{m}(n-\g)(n+2-\g)}{2\g_{n}(\g)b_{2m-2}}\frac{\Gamma\left( \frac{1}{2} \right) \Gamma\left( \frac{n+3-\g}{2} \right)}{2\sqrt{2}\Gamma\left( \frac{n+4-\g}{2} \right)}\frac{1}{\left( \sinh\rho \right)^{n+3-\g}}+O\left( \frac{1}{\left( \sinh\rho \right)^{n+2-\g}} \right)\\
      &=\frac{1}{\g_{4m}(\g)}\frac{1}{\left( \sinh\rho \right)^{4m-\g}}+O\left( \frac{1}{\left( \sinh\rho \right)^{4m-\g-1}} \right),
    \end{align*}
    where we have computed
    \[
      \frac{c_{m}(n-\g)(n+2-\g)}{2\g_{n}(\g)b_{2m-2}}\frac{\Gamma\left( \frac{1}{2} \right) \Gamma\left( \frac{n+3-\g}{2} \right)}{2\sqrt{2}\Gamma\left( \frac{n+4-\g}{2} \right)}=\frac{1}{\g_{4m}(\g)}.
    \]
    Similarly, we have in the octonionic case that
    \[
      A_{1}=\frac{1}{\g_{16}(\g)}\frac{1}{\left( \sinh\rho \right)^{16-\g}}+O\left( \frac{1}{\left( \sinh\rho \right)^{15-\g}} \right).
    \]
    Concerning estimating $A_{2}$, it is clear from Lemma \ref{lem:higher-order-derivatives-on-trig} that $A_{2}\lesssim1$ for both the quaternionic and octonionic cases, and so
    \begin{align*}
      k_{\g}(\rho)&=A_{1}+A_{2}\\
      &\leq\frac{1}{\g_{N}(\g)}\frac{1}{\rho^{N-\g}}+O\left( \frac{1}{\rho^{N-\g-1}} \right),
    \end{align*}
    as desired.

  \end{proof}

  \subsection{Estimate for $k_{\zeta,\g}$}

  In this subsection, we obtain the asymptotics for $k_{\zeta,\g}$ for $0<\g<4m$ and $0<\zeta$.
  Note that the large distance asymptotics ($1\leq\rho$) are already contained \eqref{eq:anker-ji-large-distance-asymptotics}.

  \begin{lemma}
    Let $N = \dim_{\R} H_{\mbF}^{m}$, $0<\g<N$, $0<\zeta$ and $0<\e<\min\left\{ 1,N-\g \right\}$.
    If $0<\rho<1$, then
    \[
      k_{\zeta,\g}\leq\frac{1}{\g_{N}(\g)}\frac{1}{\rho^{N-\g}}+O\left( \frac{1}{\rho^{N-\g-\e}} \right).
    \]
    If $1\leq\rho$, then
    \[
      k_{\zeta,\g}\sim\rho^{\frac{\g-2}{2}}e^{-\zeta\rho-\frac{Q}{2}\rho}.
    \]
    \label{lem:k-zeta-gamma-small-distance}
  \end{lemma}

  \begin{proof}
    As mentioned above, we only need to prove the estimate for $0<\rho<1$.

    As before,  let $n=2\tilde{m}+1$ with $\tilde{m}$ as above, and choose $\tilde\g$ and $\ell$ such that $0<\tilde\g<3$, $0\leq\ell<n-1$ and $\g=\tilde\g+\ell$.
    Then
    \[
      k_{\zeta,\g}=k_{\zeta,\tilde\g}*k_{\zeta,\ell}.
    \]
    Using Lemmas \ref{lem:convolution-estimates} and \ref{lem:convolution-estimates-octonionic}, it will be sufficient to estimate $k_{\zeta,\tilde\g}$ and $k_{\zeta,\ell}$ separately.

    To estimate $k_{\zeta,\tilde\g}$, note that, by Lemma \ref{lem:k-gamma-small-distance}, there holds
    \begin{align*}
      k_{\zeta,\tilde\g}&=\left( -\Delta - \frac{Q^{2}}{4} + \zeta^{2} \right)^{-\frac{\tilde\g}{2}}\\
      &=\frac{1}{\Gamma\left( \frac{\tilde\g}{2} \right)}\int_{0}^{\oo}t^{\frac{\tilde\g}{2}-1}e^{t\left( \Delta + \frac{Q^{2}}{4} - \zeta^{2} \right)} dt\\
      &\leq \frac{1}{\Gamma\left( \frac{\tilde\g}{2} \right)} \int\limits_{0}^{\oo}t^{\frac{\tilde\g}{2}-1}e^{t\left( \Delta + \frac{Q^{2}}{4} \right)} dt\\
      &=\left( -\Delta - \frac{Q^{2}}{4} \right)^{-\frac{\tilde\g}{2}}\\
      &=k_{\tilde\g}\\
      &\leq\frac{1}{\g_{N}(\tilde\g)}\frac{1}{\rho^{N-\tilde\g}}+O\left( \frac{1}{\rho^{N-\tilde\g-1}} \right).
    \end{align*}
    We see that, if $\ell=0$, then we are done, and so we assume without loss of generality that $0<\ell$.

    We now estimate $k_{\zeta,\ell}$.
    As in the previous proof, let $\mu = 2$ for the quaternionic case and $\mu = 4$ for the octonionic case, and let $c$ denote $c_{m}$ or $c_{o}$ in the respective cases.
    We compute
    \begin{align*}
      k_{\zeta,\ell}&= \frac{1}{\Gamma\left( \frac{\ell}{2} \right)} \int\limits_{0}^{\oo}t^{\frac{\ell}{2}-1} e^{t\left( \Delta+\frac{Q^{2}}{4} - \zeta^{2} \right)} dt\\
      &=\frac{c}{b_{\tilde{m}}} \int\limits_{\rho}^{\oo}\frac{\sinh2r}{\sqrt{\cosh2r-\cosh2\rho}} \left( -\frac{1}{\sinh2r}\frac{\p}{\p{r}} \right)^{2}\left( -\Delta_{H_{\R}^{n}}-\left( \frac{n-1}{2} \right)^{2} + \zeta^{2} \right)^{-\frac{\ell}{2}} dr\\
      &=A_{7}+A_{8},
    \end{align*}
    where
    \[
      A_{7}=\int\limits_{\rho}^{1}\cdots\text{ and }A_{8}=\int\limits_{1}^{\oo}\cdots.
    \]

    From \cite[Proposition 2.5]{LiLuy1}, we have that
    \[
      \left( -\Delta - \left( \frac{n-1}{2} \right)^{2} + \zeta^{2} \right)^{-\frac{\ell}{2}} = \frac{1}{\g_{n}(\g)}\frac{1}{\rho^{n-\ell}} + O\left( \frac{1}{\rho^{n-\ell-1}} \right),
    \]
    and by similar computations to those given in the proof of Lemma \ref{lem:k-gamma-small-distance}, we have
    \begin{align*}
      \sinh2r\left( -\frac{1}{\sinh2r} \frac{\p}{\p{r}} \right)^{2}\left( -\Delta_{H_{\R}^{n}} - \left( \frac{n-1}{2} \right)^{2} \right)&=\frac{(n-\ell)(n+2-\ell)}{2\g_{n}(\g)}\frac{1}{r^{n+3-\ell}} + O\left( \frac{1}{r^{n+2-\ell}} \right).
    \end{align*}
    and
    \begin{align*}
      &\sinh2r\left( -\frac{1}{\sinh2r} \frac{\p}{\p{r}} \right)^{4}\left( -\Delta_{H_{\R}^{n}} - \left( \frac{n-1}{2} \right)^{2} \right)\\
      &=\frac{(n-\ell)(n+2-\ell)(n+4-\ell)(n+6-\ell)}{8\g_{n}(\g)}\frac{1}{r^{n+7-\ell}} + O\left( \frac{1}{r^{n+6-\ell}} \right).
    \end{align*}
    Consequently, using $1\leq\cosh{r}$ and Lemma \ref{lemalph:luyang-hypertrig-integral}, we find for the quaternionic case that
    \begin{align*}
      A_{7}&\leq\frac{c_{m}}{b_{2m-2}}\frac{(n-\ell)(n+2-\ell)}{2\g_{n}(\g)}\int\limits_{\rho}^{\oo} \cosh{r} \frac{\sinh2r}{\sqrt{\cosh2r-\cosh2\rho}} \left[ \frac{1}{\left( \sinh{r} \right)^{n+3-\ell}} + O\left( \frac{1}{\left( \sinh{r} \right)^{n+2-\ell}} \right) \right] dr\\
      &=\frac{1}{\g_{4m}(\ell)} \frac{1}{\left( \sinh\rho \right)^{4m-\ell}} + O\left( \frac{1}{\left( \sinh\rho \right)^{4m-\ell-1}} \right)
    \end{align*}
    where we have computed
    \[
      \frac{c_{m}(n-\ell)(n+2-\ell)}{2\g_{n}(\ell)b_{2m-2}}\frac{\Gamma\left( \frac{1}{2} \right) \Gamma\left( \frac{n+3-\ell}{2} \right)}{2\sqrt{2}\Gamma\left( \frac{n+4-\ell}{2} \right)}=\frac{1}{\g_{4m}(\ell)}.
    \]
    Similarly, we have for that octonionic case that
    \[
      A_{7} \leq \frac{1}{\gamma(\ell)(\g)}\frac{1}{(\sinh\rho)^{16-\ell}} + O\left( \frac{1}{(\sinh \rho)^{15 - \ell}} \right).
    \]
    Again, using Lemma \ref{lem:higher-order-derivatives-on-trig}, we have that $A_{8}\lesssim1$, and so we have proved to two estimates
    \begin{align*}
      k_{\zeta,\ell}&\leq\frac{1}{\g_{N}(\ell)}\frac{1}{\left( \sinh\rho \right)^{N-\ell}} + O\left( \frac{1}{\left( \sinh\rho \right)^{N-\ell-1}} \right)\\
      k_{\zeta,\tilde\g}&\leq\frac{1}{\g_{N}(\tilde\g)}\frac{1}{\left( \sinh\rho \right)^{N-\tilde\g}}+O\left( \frac{1}{\left( \sinh\rho \right)^{N-\tilde\g-1}} \right).
    \end{align*}

    Now, using \eqref{eq:anker-ji-large-distance-asymptotics}, we have, for any $0<\zeta'<\zeta$, $0<\a$ and $1\leq\rho$, there holds
    \begin{align*}
      k_{\zeta,\a} \sim \rho^{\frac{\a-2}{2}} e^{-\zeta\rho - \frac{Q}{2}} \lesssim_{\a} e^{-\zeta'\rho - \frac{Q}{2}\rho}.
    \end{align*}
    Therefore, using this estimate and that $\cosh{r}\sim{e^{r}}$ and $\sinh{r}\sim{e^{r}}$ for $r>1$, and $\sinh{r} \sim r$ and $\cosh{r} \sim 1$ for $0<r<1$, we obtain the following global estimates (i.e, for $0<\rho$):
    \begin{align*}
      k_{\zeta,\ell}&\leq\frac{1}{\g_{N}(\ell)} \frac{\left( \cosh\rho \right)^{N-\frac{Q}{2}-\ell-\zeta'}}{\left( \sinh\rho \right)^{N-\ell}} + O\left( \frac{\left( \cosh\rho \right)^{N-\frac{Q}{2}-\ell-\zeta'-1}}{\left( \sinh\rho \right)^{N-\ell-1}} \right)\\
      k_{\zeta,\tilde\g}&\leq\frac{1}{\g_{N}(\tilde\g)} \frac{\left( \cosh\rho \right)^{N-\frac{Q}{2}-\tilde\g-\zeta'}}{\left( \sinh\rho \right)^{N-\tilde\g}} + O\left( \frac{\left( \cosh\rho \right)^{N-\frac{Q}{2}-\tilde\g-\zeta'-1}}{\left( \sinh\rho \right)^{N-\tilde\g-1}} \right).
    \end{align*}

    At last, using Lemmas \ref{lem:convolution-estimates} and \ref{lem:convolution-estimates-octonionic} and letting $0<\e<\min\left\{ 1, N-\g \right\}$, we obtain
    \begin{align*}
      k_{\zeta,\g}&=k_{\zeta,\ell}*k_{\zeta,\tilde\g}\\
      &\leq\frac{1}{\g_{N}(\ell+\tilde\g)}\frac{1}{\rho^{N - \tilde\g - \ell}} + O\left( \frac{1}{\rho^{N - \tilde\g - \ell - \e}} \right),
    \end{align*}
    which gives the desired estimate since $\g=\ell+\tilde\g$.
  \end{proof}

  \subsection{Estimates for $k_{\g}*k_{\zeta,\g'}$}

  In this subsection, we obtain the asymptotics for $k_{\g}*k_{\zeta,\g}$ for $0<\g<3$, $0<\g'<N-\g$ and $0<\zeta$.

  \begin{lemma}
    Let $0<\g<3$, $0<\g'<N-\g$, $0<\zeta$ and $0<\e<\min\left\{ 1,N-\g-\g',\frac{\zeta}{2} \right\}$.
    If $0<\rho<1$, then
    \[
      k_{\g}*k_{\zeta,\g'}\leq\frac{1}{\g_{N}(\g+\g')}\frac{1}{\rho^{N-\g-\g'}}+O\left( \frac{1}{\rho^{N-\g-\g'-\e}} \right).
    \]
    If $1\leq\rho$, then
    \[
      k_{\g}*k_{\zeta,g'}\lesssim{e^{\left( \e-\frac{Q}{2} \right)\rho}}.
    \]
    \label{lem:k-gamma-ast-k-zeta-gamma}
  \end{lemma}

  \begin{proof}
    By Lemma \ref{lem:k-gamma-small-distance}, we have for $0<\rho<1$ the estimate
    \[
      k_{\g}\leq\frac{1}{\g_{N}(\g)}\frac{1}{\left( \sinh\rho \right)^{N-\g}}+O\left( \frac{1}{\left( \sinh\rho \right)^{N-\g-1}} \right),
    \]
    and, by \eqref{eq:anker-ji-large-distance-asymptotics}, we have for any $0<\e$ and $1\leq\rho$ the estimate
    \[
      k_{\g} \sim \rho^{\g-2} e^{-\frac{Q}{2}\rho} \lesssim_{\g} e^{(\e-\frac{Q}{2})\rho}.
    \]
    Consequently, we obtain the global estimates (i.e., for $0<\rho$):
    \[
      k_{\g} \leq \frac{1}{\g_{N}(\g)} \frac{\left( \cosh \rho \right)^{N - \frac{Q}{2}-\g+\e}}{(\sinh \rho)^{N-\g}} + O\left( \frac{\left( \cosh \rho \right)^{N-\frac{Q}{2}-\g+\e-1}}{(\sinh \rho)^{N-\g-1}} \right).
    \]
    Similarly, we have for $0<\rho$ the global estimates
    \[
      k_{\zeta,\g'}\leq\frac{1}{\g_{N}(\g')} \frac{\left( \cosh\rho \right)^{N-\frac{Q}{2}-\g'-\zeta+\e}}{\left( \sinh\rho \right)^{N-\g'}} + O\left( \frac{\left( \cosh\rho \right)^{N-\frac{Q}{2}+\e-\g'-\zeta-1}}{\left( \sinh\rho \right)^{N-\g'-1}} \right).
    \]
    Therefore, by Lemmas \ref{lem:convolution-estimates} and \ref{lem:convolution-estimates-octonionic}, there holds
    \[
      k_{\g}*k_{\zeta,\g'}\leq\frac{1}{\g_{N}(\g+\g')} \frac{1}{\rho^{N-\g-\g'}} + O\left( \frac{1}{\rho^{N-\g-\g'-\e}} \right)
    \]
    for $0<\rho<1$.

    Similarly, using Lemmas \ref{lem:convolution-estimates-large-distance} and \ref{lem:convolution-estimates-large-distance-octonionic} we have
    \[
      k_{\g}*k_{\zeta,\g'}\lesssim e^{(\e-\frac{Q}{2})\rho}.
    \]
  \end{proof}

  \begin{lemma}
    Let $0<\g<3$, $0<\g'<N-\g$, $0<\zeta$ and $0<\zeta'<\zeta$.
    If $1\leq\rho$, then
    \[
      k_{\g}*k_{\zeta,\g'}\lesssim{e^{-\left( \zeta'+\frac{Q}{2} \right)\rho}}+\rho^{\g-2}e^{-\frac{Q}{2}\rho}*k_{\zeta,\g'}.
    \]
    \label{lem:k-gamma-ast-k-zeta-refined-estimate}
  \end{lemma}

  \begin{proof}
    Using \eqref{eq:anker-ji-large-distance-asymptotics}, we have
    \begin{align*}
      k_{\g}*k_{\zeta,\g'}&=\int\limits_{ \left\{ z\in\B_{\mbF}^{m}:\rho(z)<\frac{1}{2} \right\}} k_{\g}\left( \rho(z) \right) k_{\zeta,\g'}\left( \rho(z,w) \right) dV(z)\\
      &+\int\limits_{ \left\{ z\in\B_{\mbF}^{m}:\frac{1}{2}\leq\rho(z)<1 \right\}}k_{\g}(\rho(z)) k_{\zeta,\g'}\left( \rho(z,w) \right) dV(z)\\
      &\lesssim \int\limits_{ \left\{ z\in\B_{\mbF}^{m}:\rho(z)<\frac{1}{2} \right\}}k_{\g}\left( \rho(z) \right) k_{\zeta,\g'}\left( \rho(z,w) \right) dV(z)\\
      &+\int\limits_{ \left\{ z\in\B_{\mbF}^{m}:\frac{1}{2}\leq\rho(z)<1 \right\}} \rho(z)^{\g-2} e^{-\frac{Q}{2}\rho(z)} k_{\zeta,\g'}\left( \rho(z,w) \right) dV(z)\\
      &\leq \int\limits_{ \left\{ z\in\B_{\mbF}^{m}:\rho(z)<\frac{1}{2} \right\}}k_{\g}\left( \rho(z) \right) k_{\zeta,\g'}\left( \rho(z,w) \right) dV(z)\\
      &+\rho^{\g-2} e^{-\frac{Q}{2}\g} * k_{\zeta,\g'}.
    \end{align*}

    Thus we need only show that, for $1\leq\rho$, there holds
    \[
      \int\limits_{ \left\{ z\in\B_{\mbF}^{m}:\rho(z)<\frac{1}{2} \right\}}k_{\g}\left( \rho(z) \right) k_{\zeta,\g'}\left( \rho(z,w) \right) dV(z)\lesssim e^{-(\zeta'+\frac{Q}{2})\rho}.
    \]

    By Lemma \ref{lem:k-gamma-small-distance}, we have that, for $\rho(z)<\frac{1}{2}$, there holds
    \[
      k_{\g}(\rho(z))\lesssim \frac{1}{\rho(z)^{N-\g}} \sim \frac{1}{|z|^{N-\g}}.
    \]
    Next, observing that $1\leq\rho(w)$ and $\rho(z)<\frac{1}{2}$ imply $\frac{1}{2}\leq\rho(w) - \rho(z) \leq \rho(z,w)$, we have by \eqref{eq:anker-ji-large-distance-asymptotics} that, for $0<\zeta'<\zeta$, there holds
    \[
      k_{\zeta,\g'}\left( \rho(z,w) \right) \lesssim e^{-\zeta' \rho(z,w) - \frac{Q}{2} \rho(z,w)} \sim \left( \cosh\rho(z,w) \right)^{-(\zeta'+\frac{Q}{2})}.
    \]
    Combining these estimates with Lemma \ref{lemalph:automorphism-properties}, we compute
    \begin{align*}
      &\int\limits_{ \left\{ z\in\B_{\Q}^{m}:\rho(z)<\frac{1}{2} \right\}}k_{\g}\left( \rho(z) \right) k_{\zeta,\g'}\left( \rho(z,w) \right) dV(z) \\
      &\lesssim \int\limits_{ \left\{ z\in\B_{\Q}^{m} : \rho(z)<\frac{1}{2} \right\}} \frac{1}{|z|^{4m-\g}} \left( \frac{\sqrt{(1-|w|^{2}) (1-|z|^{2})}}{|1 - \gen{z,w}_{\Q}|} \right)^{2m+1+\zeta'} \left( \frac{1}{1-|z|^{2}} \right)^{2m+2} dz\\
      &\sim (1-|w|^{2})^{\frac{2m+1+\zeta'}{2}} \int_{ \left\{ z\in\B_{\Q}^{m} : \rho(z)<\frac{1}{2} \right\}} \frac{1}{|z|^{4m-\g}} dz\\
      &\sim \left( \cosh\rho \right)^{-(2m+1+\zeta')}\\
      &\sim e^{-(\zeta'+2m+1)\rho}.
    \end{align*}
    Similarly,
    \begin{align*}
      &\int\limits_{ \left\{ z\in\B_{\C a}:\rho(z)<\frac{1}{2} \right\}}k_{\g}\left( \rho(z) \right) k_{\zeta,\g'}\left( \rho(z,w) \right) dV(z) \\
      &\lesssim \int\limits_{ \left\{ z\in\B_{\C a} : \rho(z)<\frac{1}{2} \right\}} \frac{1}{|z|^{16-\g}} \left( \frac{(1-|w|^{2}) (1-|z|^{2})}{\Psi_{\C a}(z,w)} \right)^{\frac{11+\zeta'}{2}} \left( \frac{1}{1-|z|^{2}} \right)^{12} dz\\
      &\sim (1-|w|^{2})^{\frac{11+\zeta'}{2}} \int_{ \left\{ z\in\B_{\C a} : \rho(z)<\frac{1}{2} \right\}} \frac{1}{|z|^{16-\g}} dz\\
      &\sim \left( \cosh\rho \right)^{-(11+\zeta')}\\
      &\sim e^{-(\zeta'+11)\rho}.
    \end{align*}
  \end{proof}

  \section{Rearrangement Estimates}
  We firstly collect known results about nonincreasing rearrangements and Lorentz spaces on the hyperbolic spaces $\mathbb{X}$.
These results will be used to prove estimates on $k_{\g}*k_{\zeta,\g'}*f$ for $f\in{C_{0}^{\oo}(\mathbb{X})}$.

To begin, let $f:\mathbb{X}\to\R$, and define
\begin{align*}
  f^{*}(t)&=\inf\left\{ s>0:\la_{f}(s)\leq{t} \right\}\\
  \la_{f}(s)&=|\left\{ z\in\mathbb{X}:|f(z)|>s \right\}|\\
  &=\int\limits_{z\in\mathbb{X}:|f(z)|>s}dV.
\end{align*}
Next, for a domain $\Omega\subset\mathbb{X}$, we recall the Lorentz spaces $L^{p,q}(\Omega)$ consist of functions for which the following norm is finite:
\[
  \norm{f}_{L^{p,q}(\Omega)}=
  \begin{cases}
    \norm{t^{\frac{1}{p}-\frac{1}{q}}f^{*}(t)}_{L^{q}\left( 0,|\Omega| \right)}&1\leq{q}<\oo\\
    \sup\limits_{t>0}t^{\frac{1}{p}}f^{*}(t)&q=\oo
  \end{cases}.
\]
Define next $f^{**}(t)=\frac{1}{t}\int\limits_{0}^{t}f^{*}(s)ds$ and
\[
  \norm{f}_{L^{p,q}(\Omega)}^{*}=
  \begin{cases}
    \norm{t^{\frac{1}{p}-\frac{1}{q}}f^{**}(t)}_{L^{q}(0,|\Omega)|}&1\leq{q}<\oo\\
    \sup\limits_{t>0}t^{\frac{1}{p}}f^{**}(t)&q=\oo.
  \end{cases}
\]

Let $1<r,p_{1},p_{2}<\oo$ and $1\leq{s,q_{1},q_{2}}\leq\oo$ satisfy
\[
  \frac{1}{p_{1}}+\frac{1}{p_{2}}-1=\frac{1}{r},\quad\frac{1}{q_{1}}+\frac{1}{q_{2}}\geq\frac{1}{s},
\]
and assume $f\in{L^{p_{1},q_{1}}(\mathbb{X})}$ and $g\in{L^{p_{2},q_{2}}(\mathbb{X})}$.
The generalized Young's inequality (see \cite[Theorem 2.6]{on})
\[
  \norm{f*g}_{L^{r,s}}\leq{C}\norm{f}_{L^{p_{1},q_{1}}}\norm{g}_{L^{p_{2},q_{2}}},
\]
and norm equivalence (see \cite{on} for $1\leq{r}<\oo$ and \cite[Theorem 3.4]{yap} for $0<r<1$)
\[
  \norm{f*g}_{L^{q,r}}\leq\norm{f*g}_{L^{q,r}}^{*}\leq\frac{q}{q-1}\norm{f*g}_{L^{q,r}}
\]
give the following lemma.

\begin{lemmaalph}
  Let $1<r,p_{1},p_{2}<\oo$ and $1\leq{s,q_{1},q_{2}}\leq\oo$.
  If
  \[
    \frac{1}{p_{1}}+\frac{1}{p_{2}}-1=\frac{1}{r},\quad\frac{1}{q_{1}}+\frac{1}{q_{2}}\geq\frac{1}{s},
  \]
  $f\in{L^{p_{1},q_{1}}(\mathbb{X})}$ and $g\in{L^{p_{2},q_{2}}(\mathbb{X})}$, then
  \[
    \norm{f*g}_{L^{r,s}}\leq{C}\norm{f}_{L^{p_{1},q_{1}}}\norm{g}_{L^{p_{2},q_{2}}}.
  \]
  \label{lemalph:generalized-youngs-inequality}
\end{lemmaalph}

  In this section, we collect the kernel estimates obtained above and state the corresponding estimates for their nonincreasing rearrangements.
  We also prove that the square integrability of the rearrangement $[k_{\zeta}*k_{\zeta,\g'}]^{*}$ on any interval of the form $(c,\oo)$, $0<c$.

  In preparation of obtaining the rearrangement estimates, we first estimate the volume of the geodesic ball $B_{\rho}$ centered at the origin and with radius $\rho$.
  For $H_{\Q}^{m}$, we may use
  \[
    |B_{\rho}|=\omega_{4m-1}\int_{0}^{\rho}\left( \sinh{r} \right)^{4m-1} \left( \cosh\rho \right)^{3} dr,
  \]
  to obtain
  \[
    |B_{\rho}|=\frac{\omega_{4m-1}}{4m}\rho^{4m}+O\left( \rho^{4m+2} \right)\text{ if }0<\rho<1
  \]
  and
  \[
    |B_{\rho}| \thicksim e^{(4m+2)\rho}\text{ if }1\leq\rho.
  \]
  Similarly, for $H_{\C a}$, we may use
  \[
    |B_{\rho}|=\omega_{15}\int_{0}^{\rho}\left( \sinh{r} \right)^{15} \left( \cosh\rho \right)^{7} dr,
  \]
  to obtain
  \[
    |B_{\rho}|=\frac{\omega_{15}}{16}\rho^{16}+O\left( \rho^{18} \right)\text{ if }0<\rho<1
  \]
  and
  \[
    |B_{\rho}| \thicksim e^{22\rho}\text{ if }1\leq\rho.
  \]

  Next, we collect the kernel estimates established above.
  On $H_{\mbF}^{m}$ with $N = \dim_{\R} H_{\mbF}^{m}$, there holds
  \begin{itemize}
    \item Let $0<\zeta$.
      If $0<\g<N$, $0<\e<\min\left\{ 1,N-\g \right\}$ and $0<\rho<1$, then
      \[
	k_{\zeta,\g}\leq\frac{1}{\g_{N(\g)}}\frac{1}{\rho^{N-\g}} + O\left( \frac{1}{\rho^{N-\g-\e}} \right).
      \]
      If $0<\g$ and $1\leq\rho$, then
      \[
	k_{\zeta,\g}\sim \rho^{\frac{\g-2}{2}} e^{-(\zeta+\frac{Q}{2})\rho}.
      \]
    \item Let $\zeta=0$.
      If $0<\g<3$ and $0<\rho<1$, then
      \[
	k_{\g}\leq\frac{1}{\g_{N}(\g)} \frac{1}{\rho^{N-\g}} + O\left( \frac{1}{\rho^{N-\g-1}} \right).
      \]
      If $0<\g<3$ and $1\leq\rho$, then
      \[
	k_{\g}\sim \rho^{\g-2}e^{-\frac{Q}{2}\rho}.
      \]
    \item Let $0<\zeta$.
      If $0<\g<3$, $0<\g'<N-\g$, $0<\e<\min\left\{ 1,N-\g-\g',\frac{\zeta}{2} \right\}$ and $0<\rho<1$, then
      \[
	k_{\g}*k_{\zeta,\g'}\leq\frac{1}{\g_{N}(\g+\g')}\frac{1}{\rho^{N-\g-\g'}} + O\left( \frac{1}{\rho^{N-\g-\g'-\e}} \right).
      \]
      If $1\leq\rho$, then
      \[
	k_{\g}*k_{\zeta,g'}\lesssim{e^{\left( \e-\frac{Q}{2} \right)\rho}}.
      \]
      If $0<\zeta'<\zeta$ and $1\leq\rho$, then
      \[
	k_{\g}*k_{\zeta,\g'}\lesssim e^{-(\zeta'+\frac{Q}{2})\rho} + \rho^{\g-2} e^{-\frac{Q}{2}\rho} * k_{\zeta,\g'}.
      \]
  \end{itemize}

  The corresponding estimates for their rearrangements are listed now.

  \begin{itemize}
    \item Let $0<\zeta$.
      If $0<\g<N$, $0<\e<\min\left\{ 1,N-\g \right\}$ and $0<t<2$, then
      \[
	[k_{\zeta,\g}]^{*}\leq\frac{1}{\g_{N(\g)}} \left( \frac{N}{\omega_{N-1}}t \right)^{\frac{\g-N}{N}} + O\left( t^{\frac{\g+\e-N}{N}} \right).
      \]
      If $0<\g$ and $2\leq t$, then
      \[
	[k_{\zeta,\g}]^{*} \sim t^{-\frac{1}{2}-\frac{1}{N}\zeta} \left( \ln t \right)^{\frac{\g-2}{2}}
      \]
    \item Let $\zeta=0$.
      If $0<\g<3$ and $0<t<2$, then
      \[
	[k_{\g}]^{*} \leq \frac{1}{\g_{N}(\g)} \left( \frac{N}{\omega_{N-1}}t \right)^{\frac{\g-N}{N}} + O\left( t^{\frac{\g+1-N}{N}} \right).
      \]
      If $0<\g<3$ and $2\leq t$, then
      \[
	[k_{\g}]^{*} \sim t^{-\frac{1}{2}} \left( \ln t \right)^{\g-2}.
      \]
    \item Let $0<\zeta$.
      If $0<\g<3$, $0<\g'<N-\g$, $0<\e<\min\left\{ 1,N-\g-\g',\frac{\zeta}{2} \right\}$ and $0<t<2$, then
      \begin{equation}
	[k_{\g}*k_{\zeta,\g'}]^{*} \leq \frac{1}{\g_{N}(\g+\g')} \left( \frac{N}{\omega_{N-1}} t \right)^{\frac{\g+\g'-N}{N}} + O\left( t^{\frac{\g+\g'+\e-N}{N}} \right).
	\label{eq:rearrangement-k-g-ast-k-z-g-small-dist1}
      \end{equation}
      If $2\leq t$, then
      \begin{equation}
	[k_{\g}*k_{\zeta,g'}]^{*} \lesssim t^{\frac{\e-\frac{Q}{2}}{N}}.
	\label{eq:rearrangement-k-g-ast-k-z-g-large-dist}
      \end{equation}
      Moreover,  using Lemma \ref{lem:k-gamma-ast-k-zeta-refined-estimate},  we have, for $c>0$,
        \begin{equation}
	\int^{\infty}_{c}|[k_{\alpha}\ast k_{\zeta,\beta}]^{\ast}(t)|^{2}dt<\infty.
	\label{eq:rearrangement-k-g-ast-k-z-g-large-dist2}
      \end{equation}
The proof of (\ref{eq:rearrangement-k-g-ast-k-z-g-large-dist2}) is similar to that given in \cite{LiLuy2}, Lemma 4.1 and we omit it.
  \end{itemize}

  \subsection{Estimates for $k_{\g}*k_{\zeta,\g'}*f$}

  In this section, we prove and $L^{p}-L^{p'}$ inequality for $k_{\g}*k_{\zeta,\g'}*f$, which is dual to the Poincar\'e-Sobolev inequality.
  We will need to make use of the Kunze-Stein phenomenon. Kunze-Stein phenomenon is important in harmonic analysis (see \cite{co}, \cite{co1}, \cite{co2}, \cite{io}, \cite{Stein}) and is closely related to geometric and functional inequalities (see Beckner \cite{be1}). We  begin by recalling relevant results.
  The proofs of Lemmas \ref{lemalph:luyang-convolution-lemma-1} and \ref{lemalph:luyang-convolution-lemma-2} may be found in \cite{ly5}.

  We begin by recalling that Cowling,  Giulini and  Meda(see \cite{co}, \cite{co1}, \cite{co2}) established the following sharp version on Lorentz space (\cite{hu}, \cite{on})  of the Kunze-Stein phenomenon for connected real simple groups $G$ of real rank one with finite center:
  \[
    L^{p,q_{1}}(G)*L^{p,q_{2}}\subset{L^{p,q_{3}}(G)}
  \]
  provided $1<p<2$, $1\leq{q_{1},q_{2},q_{3}}\leq\oo$ and $1+\frac{1}{q_{3}}\leq\frac{1}{q_{1}}+\frac{1}{q_{2}}$.
  In particular, this applies to $Sp(m,1)$ and $F_{4}$, and by following \cite{ly5}, we can obtain similar phenomenon on $H_{\Q}^{m}$ and $H_{\mathbb{Q}}^{2}$.
  To be more precise, let $L^{p}(G)$ and $L^{p,q}(G)$ denote the usual Lebesgue and Lorentz spaces, respectively, and let $L^{p,q}(G/K)$, $L^{p,q}(K\backslash{G})$ and $L^{p,q}(K\backslash{G}/K)$ denote the closed subspaces of $L^{p,q}(G)$ of the right-$K$-invariant, left-$K$-invariant and $K$-bi-invariant functions, respectively.
  Following \cite{ly5}, we can show

  \begin{lemmaalph}
    For $p\in(1,2)$, there holds
    \[
      L^{p}\left( K\backslash{G} \right)*L^{p}\left( G/K \right)\subset{L^{p,\oo}(K\backslash{G}/K)}.
    \]
    \label{lemalph:luyang-convolution-lemma-1}
  \end{lemmaalph}

  \begin{lemmaalph}
    For $p\in(1,2)$ and $p'=\frac{p}{p-1}$, there holds
    \begin{align*}
      L^{p',1}\left( K\backslash{G}/K \right)*L^{p}(G/K)\subset{L^{p'}(G/K)}
    \end{align*}
    and, if $f\in{L^{p,1}(K\backslash{G}/K)}$ and $h\in{L^{p}(G/K)}$, then there exists a constant $C>0$ such that
    \begin{align*}
      \norm{f*h}_{L^{p'}(G/K)}\leq{C}\norm{f}_{L^{p',1}(K\backslash{G}/K)}\norm{h}_{L^{p}(G/K)}.
    \end{align*}
    \label{lemalph:luyang-convolution-lemma-2}
  \end{lemmaalph}

  Using Lemma \ref{lemalph:luyang-convolution-lemma-2}, we prove the following estimate on $k_{\g}*k_{\zeta,\g'}*f$.

  \begin{lemma}
    Let $0<\g<3$, $0<\g'<N-\g$, $0<\zeta$ and $\frac{2N}{N+\g+\g'}\leq{p}<2$.
    Then, for $f\in{C_{0}^{\oo}(\B_{\mbF}^{m})}$, there holds
    \[
      \norm{k_{\g}*k_{\zeta,\g'}*f}_{p'}\leq{C\norm{f}_{p}}.
    \]
    \label{lem:k-gamma-k-zeta-f-estimate}
  \end{lemma}

  \begin{proof}
    Define the cut off functions
    \begin{align*}
      \eta_{1}(\rho)&=
      \begin{cases}
	k_{\g}*k_{\zeta,\g'}&0<\rho<1\\
	0&1\leq\rho
      \end{cases}\\
      \eta_{2}(\rho)&=k_{\g}*k_{\zeta,\g'} - \eta_{1}(\rho).
    \end{align*}

    By \eqref{eq:rearrangement-k-g-ast-k-z-g-small-dist}, there exists a $t_{0}>0$ such that, for $0<t \leq t_{0}$, there holds
    \[
      \eta_{1}^{*}(t) \lesssim t^{\frac{\g+\g'-N}{N}},
    \]
    and, for $t_{0}<t$, there holds $\eta_{1}^{*}(t)=0$.
    Next, by Lemma \ref{lemalph:generalized-youngs-inequality}, there holds
    \[
      \norm{ \eta_{1} * f}_{L^{p'}} = \norm{ \eta_{1} * f}_{L^{p',p'}} \leq C \norm{ \eta_{1} }_{L^{\frac{p'}{2},\oo}} \norm{ f }_{L^{p}}.
    \]
    But
    \[
      \norm{ \eta_{1} }_{L^{\frac{p'}{2},\oo}} = \sup\limits_{0<t<\oo} t^{\frac{2}{p'}} \eta_{1}^{*}(t) \lesssim \sup\limits_{0<t<t_{0}} t^{\frac{2}{p'}+\frac{\g+\g'-N}{N}} < \oo,
    \]
    provided
    \[
      \frac{2}{p'}+\frac{\g+\g'-N}{N}>0,
    \]
    which is equivalent to
    \[
      p>\frac{2N}{\g+\g'+N},
    \]
    as is assumed.
    Consequently, there holds
    \[
      \norm{ \eta_{1} * f}_{L^{p'}} \lesssim \norm{f}_{L^{p}}.
    \]

    Next, by \eqref{eq:rearrangement-k-g-ast-k-z-g-large-dist}, there exists a $0<t_{0}$ such that, for $0 < t \leq t_{0}$, there holds
    \[
      \eta_{2}(t) \lesssim 1,
    \]
    and, for $t_{0} < t$ and $0 < \e < \min\left\{ 1 , N-\g-\g' , \frac{\zeta}{2} \right\}$, there holds
    \[
      \eta_{2}^{*}(t) \lesssim t^{\frac{\e-\frac{Q}{2}}{N}}.
    \]
    Consequently, we find, for $0<\e<\frac{Q}{2}+\frac{N}{p}$, there holds
    \[
      \norm{\eta_{2}}_{L^{p',1}} = \int_{0}^{\oo} t^{\frac{1}{p'}-1} \eta_{2}^{*}(t) dt < \oo.
    \]

    At last, Lemma \ref{lemalph:luyang-convolution-lemma-2}, we obtain
    \[
      \norm{ \eta_{2} * f}_{L^{p'}} \leq C \norm{f}_{L^{p}},
    \]
    and therefore
    \[
      \norm{ k_{\g} * k_{\zeta,\g'} * f}_{L^{p}} \leq \norm{\eta_{1} * f}_{L^{p'}} + \norm{ \eta_{2} * f}_{L^{p}} \leq C \norm{f}_{L^{p}},
    \]
    as desired.
  \end{proof}

  \section{Proofs of Theorem \ref{thm:sobolev-inequality} and  \ref{thm:hardy-sobolev-mazya-inequality} }

    With all of the kernel estimates proved in Section \ref{sec:kernel-estimates}, we are ready to prove the Poincar\'e-Sobolev inequality (Theorem \ref{thm:sobolev-inequality}) and Hardy-Sobolev-Maz'ya inequality (Theorem \ref{thm:hardy-sobolev-mazya-inequality}).
  For the reader's convenience, we restate these theorems before their respective proofs.

  \begin{manualtheorem}{\ref{thm:sobolev-inequality}}
    Let $0<\g<3$, $0<\g'$, $2<p$ and  $0<\zeta$. Denote by $N=\dim\mathbb{X}$.
  If $0<\g'<N-\g$, suppose further that $2<p\leq\frac{2N}{N-(\g+\g')}$.
  Then there exists a constant $C>0$ such that, for all $u\in{C_{0}^{\oo}\left( \mathbb{X} \right)}$, there holds
  \begin{equation}
    \begin{aligned}
      \norm{u}_{p}\leq{C}\norm{\left( -\Delta_{\mathbb{X}}-\rho_{\mathbb{X}}^{2}+\zeta^{2} \right)^{\frac{\g'}{4}}\left( -\Delta-\rho_{\mathbb{X}}^{2} \right)^{\frac{\g}{4}}u}_{2}.
    \end{aligned}
  \end{equation}
  \end{manualtheorem}

  \begin{proof}
    By Lemma \ref{lem:k-gamma-k-zeta-f-estimate}, we have
    \begin{equation}
      \norm{\left( -\Delta_{\mathbb{X}}-\rho_{\mathbb{X}}^{2}+\zeta^{2} \right)^{\frac{-\g'}{4}}\left( -\Delta-\rho_{\mathbb{X}}^{2} \right)^{-\frac{\g}{4}}u}_{L^{p'}} \leq C \norm{u}_{L^{p}}.
      \label{eq:sobolev-inequality-proof-eq1}
    \end{equation}
    Consulting \cite[Appendx Lemma]{be}, we have that that \eqref{eq:sobolev-inequality-proof-eq1} is equivalent to
    \[
      \norm{u}_{L^{p}} \leq C \norm{ \left( -\Delta_{\mathbb{X}}-\rho_{\mathbb{X}}^{2}+\zeta^{2} \right)^{\frac{\g'}{4}}\left( -\Delta-\rho_{\mathbb{X}}^{2} \right)^{\frac{\g}{4}}u}_{L^{2}},
    \]
        thereby proving the theorem.
  \end{proof}
\
\\

\textbf{Proof of Theorem \ref{thm:hardy-sobolev-mazya-inequality}}.
    We need only prove the inequality in case

    \[
      \la=\prod_{j=1}^{k}\frac{(a-k+2j-2)^{2}}{4}.
    \]
    We will use the factorization Theorem (Theorem \ref{lem:first-factorization}), and so, we set
    \[
      u=\vr^{\frac{k-(2m+1)-a}{2}}f,
    \]
    and obtain
    \begin{align*}
      &4^{k}\int\limits_{\H_{\Q}^{m-1}} \int\limits_{0}^{\oo} u \prod_{j=1}^{k}\left[ -\vr\p_{\vr\vr}-a\p_{\vr}-\vr\Delta_{Z}-\mathcal{L}_{0}+i(k+1-2j)\sqrt{-\Delta_{Z}} \right]u \frac{dxdzd\vr}{\vr^{1-a}}\\
      &= \int\limits_{\H_{\Q}^{m-1}} \int\limits_{0}^{\oo} f\prod_{j=1}^{k}\left[ -\Delta-(2m+1)^{2}+(a-k+2j-2)^{2} \right]f \frac{dxdzd\vr}{\vr^{2m+2}}\\
      &=4 \int_{\mcU^{m}}f\prod_{j=1}^{k}\left[ -\Delta-(2m+1)^{2}+(a-k+2j-2)^{2} \right]f dV.
    \end{align*}

    Next, using that $\spec\left( -\Delta \right) = [(2m+1)^{2},\oo)$, we have the following sharp inequality 
    \begin{align*}
      \int_{\mcU^{m}}f\prod_{j=1}^{k}\left[ -\Delta-(2m+1)^{2}+(a-k+2j-2)^{2} \right]f dV\geq \prod_{j=1}^{k}(a-k+2j-2)^{2} \int\limits_{\mcU^{m}} f^{2} dV.
    \end{align*}

    Applying Plancherel's theorem, there holds
    \begin{align*}
      &\int\limits_{\mcU^{m}} f \prod_{j=1}^{k}\left[ -\Delta - (2m+1)^{2} + (a-k+2j-2)^{2} \right] f dV- \prod_{j=1}^{k} (a-k+2j-2)^{2} \int\limits_{\mcU^{m}} f^{2} dV\\
      &= C_{m} \int\limits_{-\oo}^{\oo} \int\limits_{S^{4m-1}} \left[ \prod_{j=1}^{k} \left( \la^{2} + (a-k+2j-2)^{2} \right) - \prod_{j=1}^{k} (a-k+2j-2)^{2} \right] |\hat{f}(\la,\varsigma)|^{2} |\mfc(\la)|^{-2} d\la d\sigma(\varsigma).
    \end{align*}

    Choosing $0<\d$ so that
    \[
      \prod_{j=1}^{k}(\la^{2} + (a-k+2j-2)^{2}) - \prod_{j=1}^{k} (a-k+2j-2)^{2} \geq \la^{2} (\la^{2} + \d)^{k-1},
    \]
    applying Theorem \ref{thm:sobolev-inequality}, and applying the Plancherel theorem, we obtain
    \begin{align*}
      &\int\limits_{\mcU^{m}} f \prod_{j=1}^{k}\left[ -\Delta - (2m+1)^{2} + (a-k+2j-2)^{2} \right] f dV- \prod_{j=1}^{k} (a-k+2j-2)^{2} \int\limits_{\mcU^{m}} f^{2} dV\\
      &\geq C_{m} \int\limits_{-\oo}^{\oo} \int\limits_{S^{4m-1}} \la^{2} (\la^{2} + \d)^{k-1} |\hat{f}(\la,\varsigma)|^{2} |\mfc(\la)|^{-2} d\la d\sigma(\varsigma)\\
      &= \int\limits_{\mcU^{m}} f \left( -\Delta - (2m+1)^{2} \right)\left( -\Delta - (2m+1)^{2} + \d \right)^{k-1} f dV\\
      &\geq C \norm{f}_{L^{p}}^{2}.
          \end{align*}
          This proves the first inequality.
 The proof of the second inequality is similar and we omit it.

\section{Proofs of Theorem \ref{thm:hardy-Adams-inequality} and \ref{thm:hardy-Adams-inequality2}}
\textbf{Proof of Theorem \ref{thm:hardy-Adams-inequality}}
 Set $\Omega(u)=\{x\in\mathbb{B}_{\mathbb{C}}^{n}:|u(x)|\geq1\}$. Then by Theorem \ref{thm:sobolev-inequality},
 we have, for $p>2$,
\begin{equation*}
\begin{split}
|\Omega(u)|=&\int_{\Omega(u)}dV\leq\int_{\mathbb{X}}|u|^{p}dV\lesssim 1.
\end{split}
\end{equation*}
Therefore, $|\Omega(u)|\leq \Omega_{0}$ for some constant $\Omega_{0}$ independent of $u$.
We write
\begin{equation}\label{6.1}
\begin{split}
&\int_{\mathbb{B}_{\mathbb{C}}^{n}}(e^{\beta_{0}(N/2,N) u^{2}}-1-\beta_{0}(N/2,N) u^{2})dV\\
=&\int_{\Omega(u)}(e^{\beta_{0}(N/2,N) u^{2}}-1-\beta_{0}(N/2,N) u^{2})dV+
\int_{\mathbb{X}\setminus\Omega(u)}(e^{\beta_{0}(N/2,N) u^{2}}-1-\beta_{0}(N/2,N) u^{2})dV\\
\leq&\int_{\Omega(u)}e^{\beta_{0}(N/2,N) u^{2}}dV+
\int_{\mathbb{X}\setminus\Omega(u)}(e^{\beta_{0}(n,2n)
u^{2}}-1-\beta_{0}(N/2,N) u^{2})dV.
\end{split}
\end{equation}
The second part of right hand of (\ref{6.1}) is bounded. In fact, we have
\begin{equation*}
\begin{split}
&\int_{\mathbb{X}\setminus\Omega(u)}(e^{\beta_{0}(N/2,N) u^{2}}-1-\beta_{0}(N/2,N) u^{2})dV\\
=&\int_{\mathbb{X}\setminus\Omega(u)}\sum^{\infty}_{n=2}\frac{(\beta_{0}(N/2,N) u^{2})^{n}}{n!}dV\\
\leq&\int_{\mathbb{X}\setminus\Omega(u)}\sum^{\infty}_{n=2}\frac{(\beta_{0}(N/2,N) )^{n}u^{4}}{n!}dV\\
\leq&\sum^{\infty}_{n=2}\frac{(\beta_{0}(N/2,N) )^{n}}{n!}\int_{\mathbb{X}}|u(x)|^{4}dV\leq C.
\end{split}
\end{equation*}
Here we use the fact $|u(z)|<1$, $z\in\mathbb{X}\setminus\Omega(u)$ and
Theorem \ref{thm:sobolev-inequality}.

Next we shall show that
$\int_{\Omega(u)}e^{\beta_{0}(N/2,N) u^{2}}dV$ is also bounded by some universal
constant.
Set
\[
v=(-\Delta_{\mathbb{X}}-\rho_{\mathbb{X}}^{2}+\zeta^{2})^{(2n-\alpha)/4}(-\Delta_{\mathbb{X}}-\rho_{\mathbb{X}}^{2})^{\alpha/4}
u.
\]
Then
\begin{equation}\label{5.5}
\begin{split}
\int_{\mathbb{X}}|v|^{2}dV\leq1
\end{split}
\end{equation}
and  we can write $u$ as a potential
\begin{equation}\label{5.6}
\begin{split}
u=(-\Delta_{\mathbb{X}}-\rho_{\mathbb{X}}^{2}+\zeta^{2})^{-(2n-\alpha)/4}(-\Delta_{\mathbb{X}}-\rho_{\mathbb{X}}^{2})^{-\alpha/4}v=v\ast\phi,
\end{split}
\end{equation}
where $\phi=(-\Delta_{\mathbb{X}}-\rho_{\mathbb{X}}^{2}+\zeta^{2})^{-(2n-\alpha)/4}(-\Delta_{\mathbb{X}}-\rho_{\mathbb{X}}^{2})^{-\alpha/4}=k_{\zeta,(N-\alpha)/2}\ast k_{\alpha/2}$.
By \ref{eq:rearrangement-k-g-ast-k-z-g-small-dist1} and \ref{eq:rearrangement-k-g-ast-k-z-g-large-dist2},
\begin{equation*}
  \phi^{\ast}(t)\leq \frac{1}{\gamma_{N}(N/2)}\cdot\left(\frac{Nt}{\omega_{N-1}}\right)^{-\frac{1}{2}}
+O\left(t^{\frac{\epsilon-n}{2n}}\right),\;\;\;\;0<t<2\;\;\;\; \textrm{and}\; \int_{c}^{\infty}|\phi^{\ast}(t)|^{2}dt<\infty,\;\forall c>0.
\end{equation*}
Closely  following the   proof of  Theorem 1.7  in \cite{LiLuy1},  we have that there exists a constant $C$ which is independent of  $u$ and $\Omega(u)$ such that
 \begin{equation*}
\begin{split}
&\int_{\Omega(u)}e^{\beta_{0}(N/2,N) u^{2}}dV=\int^{|\Omega(u)|}_{0}\exp(\beta_{0}(N/2,N)|u^{\ast}(t)|^{2})dt\leq\int^{\Omega_{0}}_{0}\exp(\beta_{0}(N/2,N)|u^{\ast}(t)|^{2})dt\leq C.
\end{split}
\end{equation*}
 The proof of Theorem \ref{thm:hardy-Adams-inequality} is thereby completed.
\
\\

\textbf{Proof of  Theorem  \ref{thm:hardy-Adams-inequality2}}
It is enough to show that in term of ball model,  for some $\zeta>0$, there holds
\begin{equation*}
\begin{split}
&\|(-\Delta_{\mathbb{X}}-\rho_{\mathbb{X}}^{2}+\zeta^{2})^{(2m-1)/2}(-\Delta_{\mathbb{X}}-\rho_{\mathbb{X}}^{2})^{1/2}[(1-|z|^{2})^{\frac{a+1}{2}}u]\|_{2}\\
\leq&  4^{2m}\int_{\mathbb{B}_{\mathbb{Q}}^{n}}u \prod^{2m}_{j=1}\left[\Delta'_{\frac{1-a-(2m+1)}{2}}+\frac{(2m+1-2j)^{2}}{4}-
i\frac{2m+1-2j}{2}\sqrt{\Gamma+1}\right]u\frac{dz}{(1-|z|^{2})^{1-a}}\\
  &-\prod^{2m}_{j=1}(a-2m+2j-2)^{2}\int_{\mathbb{B}_{\mathbb{Q}}^{n}}\frac{u^{2}}{(1-|z|^{2})^{2m+1-a}}dz
\end{split}
\end{equation*}
and in term of Siegel domain,
  \begin{equation*}
    \begin{aligned}
    &\|(-\Delta_{\mathbb{X}}-\rho_{\mathbb{X}}^{2}+\zeta^{2})^{(2m-1)/2}(-\Delta_{\mathbb{X}}-\rho_{\mathbb{X}}^{2})^{1/2}[\varrho^{\frac{a+1}{2}}u]\|_{2}\\
    \leq  & 4^{2m}\int_{\mathbb{H}_{\mathbb{Q}}^{m-1}}\int^{\infty}_{0}u\prod^{n}_{j=1}\left[-\varrho\partial_{\varrho\varrho}-a\partial_{\varrho}-\varrho \Delta_{Z}+\mathcal{L}_{0} +i(k+1-2j)\sqrt{-\Delta_{Z}}\right]u\frac{dxdzd\varrho}{\varrho^{1-a}}\\
  &-\prod^{2m}_{j=1}(a-n+2j-2)^{2}\int_{\mathbb{H}_{\mathbb{Q}}^{-1}}\int^{\infty}_{0}\frac{u^{2}}{\varrho^{2m+1-a}}dxdzd\varrho.
    \end{aligned}
  \end{equation*}
The proof is similar to that given in the proof of \ref{thm:hardy-sobolev-mazya-inequality} via Plancherel formula and we omit it.

\section{Appendix: Proofs of Theorem \ref{thm:Adams-inequality1} and \ref{thm:Adams-inequality2}}

In this section, we will outline the proofs of Adams inequalities, namely  Theorem \ref{thm:Adams-inequality1} and \ref{thm:Adams-inequality2} for the convenience of the reader.
We refer the interested reader to \cite{LiLuy1}, \cite{LiLuy2}, \cite{ly2}, \cite{ly5} for all the details.

\textbf{Proof of Theorem \ref{thm:Adams-inequality1}}
Let $f=(-\Delta_{\mathbb{X}}-\rho_{\mathbb{X}}^{2}+\zeta^{2})^{\frac{\alpha}{2}}u$. Then $\|f\|_{p}\leq1$ and
$$u=(-\Delta_{\mathbb{X}}-\rho_{\mathbb{X}}^{2}+\zeta^{2})^{-\frac{\alpha}{2}}f=f\ast k_{\zeta,\alpha}$$
Using O'Neil's lemma (\cite{on}), we have for $t>0$,
\[
u^{\ast}(t)\leq\frac{1}{t}\int^{t}_{0}f^{\ast}(s)ds\int^{t}_{0} k_{\zeta,\alpha}^{\ast}(s)ds+
\int^{\infty}_{t}f^{\ast}(s) k_{\zeta,\alpha}^{\ast}(s)ds.
\]
Using the rearrangement estimates of $k_{\zeta,\alpha}]^{\ast}$, it is easy to check
\begin{equation*}
\begin{split}
&[k_{\zeta,\alpha}]^{\ast}(t)\leq\frac{1}{\gamma_{N}(\alpha)}\left(\frac{Nt}{\omega_{2n-1}}\right)^{\frac{\alpha-N}{N}}
+O\left(t^{\frac{\alpha+\epsilon-N}{N}}\right),   \;0<t<2;\\
&\int^{\infty}_{c}|[k_{\zeta,\alpha}]^{\ast}(t)|^{p'}dt<\infty,\;\;\forall c>0.
\end{split}
\end{equation*}
Closely  following the   proof of  Theorem 1.13  in \cite{LiLuy2},  we have that there exists a constant $C$ which is independent of  $u$  such that
 \begin{equation*}
\begin{split}
&\frac{1}{|E|}\int_{E}\exp(\beta_{0}(\alpha,N)|u|^{p'})dV\leq
\frac{1}{|E|}\int^{|E|}_{0}\exp(\beta_{0}(\alpha,N)|u^{\ast}(t)|^{p'})dt\\
\leq&\frac{1}{|E|}\int^{|E|}_{0}\exp\left(\beta_{0}(\alpha,N)\left|\frac{1}{t}\int^{t}_{0}f^{\ast}(s)ds\int^{t}_{0} k_{\zeta,\alpha}^{\ast}(s)ds+
\int^{\infty}_{t}f^{\ast}(s) k_{\zeta,\alpha}^{\ast}(s)ds\right|^{p'}\right)dt\leq C.
\end{split}
\end{equation*}
The sharpness of the constant $\beta_{0}(\alpha,N)$ can be verified by the process similar to that in \cite{ad,ksw}  and thus the
 proof of Theorem \ref{thm:Adams-inequality1} is completed.\
\\

Using the symmetrization-free argument from the local inequalities to global ones developed by Lam and the second author \cite{ll, ll2}, we can conclude the

\textbf{Proof of Theorem \ref{thm:Adams-inequality2}}   Let $u\in W^{\alpha,p}(\mathbb{X})$ with $\int_{\mathbb{X}}|(-\Delta_{\mathbb{X}}-
\rho_{\mathbb{X}}^{2}+\zeta^{2})^{\frac{\alpha}{2}} u|^{p}dV\leq1$. By H\"ormander-Mikhlin type multiplier theorem (see \cite{an}), we have
\[
\int_{\mathbb{X}}|u|^{p}dV\lesssim \int_{\mathbb{X}}|(-\Delta_{\mathbb{X}}-
\rho_{\mathbb{X}}^{2}+\zeta^{2})^{\frac{\alpha}{2}} u|^{p}dV\leq1
\]
provided $\zeta>2\rho_{\mathbb{X}}|\frac{1}{2}-\frac{1}{p}|$.
 Set $\Omega(u)=\{z\in \mathbb{X}:|u(z)|\geq1\}$. Then
 we have
\begin{equation*}
\begin{split}
|\Omega(u)|=&\int_{\Omega(u)}dV\leq\int_{\mathbb{X}}|u|^{p}dV\leq  \Omega_{0},
\end{split}
\end{equation*}
where $\Omega_{0}$ is a constant  independent of $u$.
We write
\begin{equation*}
\int_{\mathbb{X}}\Phi_{p}(\beta_{0}(\alpha,N)|u|^{p'})dV
=\int_{\Omega(u)}\Phi_{p}(\beta_{0}(\alpha,N)|u|^{p'})dV+\int_{\mathbb{X}\setminus\Omega(u)}\Phi_{p}(\beta_{0}(\alpha,N)|u|^{p'})dV.
\end{equation*}
Notice that on the domain $\mathbb{X}\setminus\Omega(u)$, we
have $|u(z)|<1$. Thus,
\begin{equation}\label{5.13}
\begin{split}
\int_{\mathbb{X}\setminus\Omega(u)}\Phi_{p}(\beta_{0}(\alpha,N)|u|^{p'})dV\leq
&\sum^{\infty}_{k=j_{p}-1}\frac{\beta_{0}(\alpha,N)^{k}}{k!}\int_{\mathbb{X}\setminus\Omega(u)}\sum^{\infty}_{n=2}|u|^{p'k}dV\\
\leq&\sum^{\infty}_{k=j_{p}-1}\frac{\beta_{0}(\alpha,N)^{k}}{k!}\int_{\mathbb{X}\setminus\Omega(u)}\sum^{\infty}_{n=2}|u|^{p}dV\\
\leq&\sum^{\infty}_{k=j_{p}-1}\frac{\beta_{0}(\alpha,N)^{k}}{k!}\|u\|^{p}_{p}\leq C.
\end{split}
\end{equation}
Moreover, by Theorem \ref{thm:Adams-inequality1}, if $\zeta$ satisfies $\zeta>0$ if $1<p<2$
and $\zeta>2n\left|\frac{1}{p}-\frac{1}{2}\right|$ if $p\geq2$, then
\begin{equation}\label{5.14}
\begin{split}
\int_{\Omega(u)}\Phi_{p}(\beta_{0}(\alpha,N)|u|^{p'})dV\leq&\int_{\Omega(u)}\exp(\beta_{0}(\alpha,N)|u|^{p'})dV\leq C.
\end{split}
\end{equation}
Combining (\ref{5.13}) and (\ref{5.14}) yields
\begin{equation*}
\begin{split}
\int_{\mathbb{X}}\Phi_{p}(\beta_{0}(\alpha,N)|u|^{p'})dV=&
=\int_{\Omega(u)}\Phi_{p}(\beta_{0}(\alpha,N)|u|^{p'})dV+\int_{\mathbb{X}\setminus\Omega(u)}\Phi_{p}(\beta_{0}(\alpha,N)|u|^{p'})dV\leq C
\end{split}
\end{equation*}
provided that  $\zeta$ satisfies   $\zeta>2\rho_{\mathbb{X}}\left|\frac{1}{p}-\frac{1}{2}\right|$.

The sharpness of the constant $\beta_{0}(\alpha,N)$ can be verified by the process similar to that in \cite{LiLuy2}.

  {\bf Acknowledgement.} The main results of this paper have been presented by the first author in the AMS special session on ``Geometric and functional inequalities and nonlinear partial differential equations" in March, 2021.

\end{document}